\documentclass[11pt,a4paper]{article}

\usepackage{amscd}

\usepackage{enumerate}

\usepackage[lmargin=2.5cm,rmargin=2.5cm,tmargin=2.8cm,bmargin=2.8cm]{geometry}

\usepackage{amsmath, amssymb, latexsym}
\usepackage{bbm}
\usepackage{authblk}
\usepackage{amsfonts}
\usepackage{amsthm}
\usepackage{relsize}
\usepackage{setspace}
\usepackage{geometry}
\usepackage{url}
\usepackage{xspace}
\usepackage{tocloft}
\usepackage{graphics}
\usepackage{graphicx}
\usepackage{lscape}
\usepackage{microtype}
\usepackage[T1]{fontenc}
\usepackage{ulem}
\usepackage{mathrsfs}

\usepackage[usenames, dvipsnames]{color}
\usepackage[utf8]{inputenc}
\usepackage{tikz}

\usepackage[pagebackref=true]{hyperref}
\usepackage[alphabetic]{amsrefs}
\usepackage[english]{babel}

\renewcommand{\phi}{\varphi}

\newcommand{\inv}{^{-1}}

\DeclareMathOperator{\tr}{tr}

\DeclareMathOperator{\sign}{sign}

\newcommand{\gr}{Grassmannian}
\newcommand{\gmle}{GE}
\newcommand{\gm}{Grassmann manifold}
\newcommand{\sa}{self-$\sig$-adjoint}

\newcommand{\R}{\mathbb{R}} 
\newcommand{\Gr}{\mathrm{G}} 
\newcommand{\pos}{\mathrm{Pos}}	
\newcommand{\E}{\mathbb{E}}

\newcommand{\sas}{{\Gr(m,r)}} 
\newcommand{\pas}{\pos^1_{sym}(m)} 
\newcommand{\pls}{{\Gr(m)}} 

\newcommand{\ls}{{\R^m}} 
\newcommand{\mat}{\R^{m\times m}} 
\newcommand{\projective}{\mathbb{P}} 
\newcommand{\ps}{{\projective^{m-1}}} 

\newcommand{\sig}{\Sigma} 

\newcommand{\p}{Pr(E_k,\sig)} 

\newcommand{\listU}{U_1,\ldots,U_n}

\newtheorem{theorem}{Theorem}[section]
\newtheorem{proposition}[theorem]{Proposition}

\newtheorem{corollary}[theorem]{Corollary}
\newtheorem{lemma}[theorem]{Lemma}

\theoremstyle{definition}
\newtheorem{definition}[theorem]{Definition}
\newtheorem{example}[theorem]{Example}
\newtheorem{remark}[theorem]{Remark}

\newtheorem*{rema*}{Remark}
\newtheorem*{remas*}{Remarks}
\newtheorem*{prop*}{Proposition}

\theoremstyle{problem}

\newcommand{\Ker}{\mathrm{Ker}}

\newcommand{\RR}{\mathbb{R}}

\newcommand{\cat}{$\mathrm{CAT}(0)$\xspace}

\newcommand{\diag}{\operatorname{diag}}

\newcommand{\id}{\operatorname{id}}

\newcommand{\Ad}{\operatorname{Ad}}

\newcommand{\SO}{\operatorname{SO}}
\newcommand{\GL}{\operatorname{GL}}
\newcommand{\ML}{\operatorname{ML}}
\newcommand{\Id}{\operatorname{Id}}

\newcommand{\SL}{\operatorname{SL}}

\newcommand{\G}{\operatorname{\mathbb{G}}}

\newcommand{\grad}{\operatorname{grad}}
\newcommand{\Veco}{\operatorname{Vec}}
\newcommand{\Hess}{\operatorname{Hess}}

\newcommand{\Sym}{\operatorname{Sym}}
\newcommand{\Pos}{\operatorname{Pos}}
\newcommand{\PosS}{\operatorname{PosS}}
\newcommand{\supp}{\operatorname{supp}}
\newcommand{\spn}{\operatorname{span}}

\newcommand{\TT}{\operatorname{\bf{T}} \mkern-4mu}

\definecolor{purple}{rgb}{0.5, 0.0, 0.5}

\newcommand{\efface}[1]{}

\def\og{\leavevmode\raise.3ex\hbox{$\scriptscriptstyle\langle\!\langle$~}}
\def\fg{\leavevmode\raise.3ex\hbox{~$\!\scriptscriptstyle\,\rangle\!\rangle$}}

\def\truc{\unskip\kern 3pt\penalty 500
\hbox{\vrule\vbox to 5pt{\hrule width 4pt\vfill\hrule}\vrule}\kern
3pt}

\def\R{{\mathbb R}}

\def\shp{{\mathcal P}}

\title{Geometrical and statistical properties of M-estimates of scatter on Grassmann manifolds  }

\author[1]{Corina Ciobotaru\thanks{corina.ciobotaru@unifr.ch}}
\author[1]{Christian Mazza\thanks{christian.mazza@unifr.ch}}

\affil[1]{Universit\'e de Fribourg, Section de Math\'ematiques, Chemin du Mus\'{e}e 23, 1700 Fribourg, Switzerland.}

\date{\today}

\begin{document}
\newcounter{qcounter}
\maketitle

\begin{abstract}  
We consider data from the Grassmann manifold $G(m,r)$ of all vector subspaces of dimension $r$ of $\R^m$, and  focus on the  Grassmannian statistical model.
Canonical Grassmannian distributions $\mathbb{G}_\sig$ on $G(m,r)$
 are indexed by parameters $\Sigma$ from the manifold 
$\mathcal{M}= \pas$ of positive definite symmetric matrices of determinant $1$. Robust  M-estimates of scatter (GE) for general probability measures $\shp$ on $G(m,r)$ are studied. Such estimators 
are defined to be the maximizers
of
the Grassmannian log-likelihood 
$-\ell_\shp(\sig)$ as  function of $\Sigma$. One of the novel features of this work is a strong use of the fact that $\mathcal{M}$ is a CAT(0) space with known visual boundary at infinity $\partial \mathcal{M}$, allowing us to study existence and unicity of GEs; along the way it is proved that 
$\ell_\shp(\gamma(t))$ is convex (and under further conditions even strictly convex) when evaluated on geodesics $\gamma(t)$ of $\mathcal{M}$. We also recall that the sample space $G(m,r)$ is a part of  $\partial \mathcal{M}$, show the distributions $\G_\sig$ are $\SL(m, \RR)$--quasi-invariant, and
that $\ell_\shp(\sig)$ is a weighted Busemann function. Let $\shp_n =(\delta_{U_1}+\cdots+\delta_{U_n})/n$ be the empirical probability measure for $n$-samples of random i.i.d. subspaces $U_i\in G(m,r)$ 
of common distribution $\shp$ with support $G(m,r)$. For $\Sigma_n$ and $\Sigma_\shp$ the \gmle s\ of $\shp_n$ and $\shp$, we show the almost sure convergence of $\Sigma_n$ toward $\Sigma$ as $n\to\infty$ using methods from geometry, and provide a central limit theorem for the rescaled process 
$C_n = \frac{m}{\tr(\Sigma_{\shp}^{-1} \Sigma_n)}g^{-1} \Sigma_n g^{-1}$,  where $\Sigma =g g$ with $g \in \SL(m, \RR)$ the unique symmetric positive-definite square root of $\Sigma$.

\end{abstract}

\section{Introduction and summary}
\label{Introduction_summary}


The deluge of current data in science, social sciences and technology is remarkable for the proliferation of new data type, and practitioners are increasingly faced with the geometries they induce. We focus  on data from  the \textbf{\gm\ $\sas$} of all vector subspaces of dimension $r$ of $\ls$ ($0<r<m$).
Such data arise for example in statistics in the setting of  likelihood-based sufficient dimension reduction \cite{Cook,Tan},
in signal processing  (see, e.g., \cite{Nokleby}, \cite{Zhang}) or  machine learning \cite{Lerman}.
Let $x_1,\ldots, x_{r}\in\ls$ be i.i.d. random vectors in $\ls$ with central normal distribution of positive definite  self-adjoint covariance matrix~$\sig$. The density of the normal law is $\exp(x^T\sig\inv x/2)$ ($x\in\ls$) up to a constant factor. We define the \textbf{\gr\ distribution of parameter $\sig$} as the law of the linear span $\langle x_1,\ldots, x_r\rangle$ of these vectors in $\ls$. It is a Borel probability measure $\G_\sig$ on $G(m,r)$. The parameter $\sig$ of a \gr\ distribution $\G_\sig$ is defined up to a positive factor only. This indeterminacy is removed by requiring the determinant of $\sig$ to be $1$. So, we parametrize the \gr\ distributions by the space $\pas$ of positive definite self-adjoint matrices $\sig\in\mat$ of determinant $1$.

Given a regular matrix $A\in\mat$, the random vectors $Ax_1,\ldots, Ax_r$ are i.i.d.\ with central normal law of covariance matrix $A\sig A^T$. Hence, the image measure of $\G_\sig$ under the transformation of $\sas$ given by
$AU=\{Ax\mid x\in U\}$ for $U\in\sas$ is
\begin{equation}\label{equivariance}
A\G_\sig=\G_{A\sig A^T}.
\end{equation}

Let us represent a point $U\in\sas$ as the linear span $U=\langle x_1,\ldots,x_r\rangle$ of
linearly independent vectors $x_1,\ldots, x_r$ of $U$ or, equivalently, as the range $U=\langle X\rangle$ of the matrix $X=(x_1,\ldots,x_r)$ of rank $r$. Then, a computation shows that the density, or Radon--Nikodym derivative, of the \gr\ distribution $\G_\sig$ ($\sig\in\pas$) with respect to the uniform distribution $\G_{\Id_m}$ on $\sas$ ($\Id_m$ = identity $m \times m$ matrix) is given by
\begin{equation}\label{density}
\frac{d\G_\sig}{d\G_{\Id_m}}(\langle
X\rangle)=
\left(\frac{\det(X^TX)}{\det(X^T\sig\inv X)}\right)^{ m/2}
\end{equation}
(see \cite{chi}). Interestingly, such probability measures are similar to recently defined determinantal probability measures on Grassmannians which are designed to model fermionic matter \cite{Kassel} or to probabilities associated with sufficient dimension reduction subspaces methods from machine learning \cite{Tan,Cook}.

When $r=1$, the \gr\ distribution $\G_\sig$ is known as the  \textbf{angular Gaussian distribution} of parameter $\sig\in\pas$ on the projective space $\ps=\Gr(m,1)$ (see \cite{AMR2005}). For any $0<r<m$, the \gm\ $\sas$ can be viewed as the space of projective subspaces of dimension $r-1$ of $\ps$ by identifying a vector $r$-subspace $U$ of $\ls$ with the projective subspace $\{y\in\ps\mid y\subseteq U\}$.
In this projective interpretation, the \gr\ distribution $\G_\sig$ on $\sas$ is the law of the projective span of i.i.d.\ random points $y_1,\ldots,y_r$ of $\ps$ with angular Gaussian distribution of parameter $\sig$.

Let $\shp$ be a Borel probability measure on $\sas$. A parameter $\sig\in\pas$ is called a \textbf{Grassmannian M-estimate of scatter} ---abbreviated \gmle\ in the sequel--- of $\shp$ if it maximizes the log-likelihood 
$\int_\sas\log(d\G_\sig/d\G_{\Id_m})\,d\shp$.
It is called a \gmle\ of a sample $\listU\in\sas$ when $\shp$ is the sample empirical measure. 

For convenience, we shall rather work with the following negative version of the log-likelihood
\begin{align}
\label{log-likelihood}\ell_\shp(\sig)&=-\frac1{ m}\int_\sas\log(d\G_\sig/d\G_{\Id_m})\,d\shp=\int_\sas\ell_U(\sig)\,d\shp(U),\\
\noalign{\noindent where the (negative) log-density $\ell_U$ is defined by}
\label{log-density}
\ell_U(\sig)&=-\frac1{m}\log\frac{d\G_\sig}{d\G_{\Id_m}}(U)
=\frac12\log\frac{\det(X^T\sig\inv X)}{\det(X^TX)}
\qquad(U=\langle X\rangle\in\sas).
\end{align}
With this notation, a \gmle\ of $\shp$ minimizes $\ell_\shp$. 

When $r=1$, \gmle\ is known as Tyler's M-estimator of scatter and is a special case of Maronna's affine invariant M-estimator
 \cite{Maronna,Ty}. It has the desirable property of being robust to outliers
 in the Huber sense \cite{Hub}, and is thus particularly suitable 
 for handling big data.
   Tyler \cite{Ty} proved that it is the most robust estimator of covariance for 
elliptical distributions. 
The
existence and unicity  of \gmle\ has been studied in 
\cite{Ty} and \cite{Kent}, see also \cite{Lutz} for 
more results on such M-estimators.
 The authors of \cite{AMR2005} studied such questions using the geometry of the parameter space $\mathcal{M}=
\pas$, which is a classical Riemannian manifold when endowed with a statistically meaningful Riemannian metric.
Within this framework, it turns out that the \textbf{M-functional} 
$\ell_\shp(\gamma(t))$ evaluated on geodesics $\gamma(t)$ of $\mathcal{M}$ is convex. 
The authors of \cite{AMR2005} then obtained precise results on existence and unicity of \gmle\ by studying the behaviour at infinity of $\ell_\shp(\gamma(t))$. Similar approaches were then developed within the signal processing literature
\cite{Wiesel,ZWG}, and in statistics \cite{DT} where a full treatment of the role of the Riemannian manifold structure of $\pas$ is given.
The random matrix regime of Maronna's and Tyler's robust estimators have been considered recently assuming similar population and sample sizes in \cite{Couillet14,Couillet15,Zhang16}. In this setting, it is shown, e.g., that the random spectral measure of properly scaled Tyler M-estimators converges weakly to the Marcenko--Pastur distribution.

For general $0<r<m$, the authors of \cite{AMR} studied the existence and uniqueness of \gmle\ using \emph{the sample space $G(m,r)$ as a subset of the boundary $\partial \mathcal{M}$ of the parameter space $\mathcal{M}$}, as observed by 
\cite{FR} when $r=1$. 

\medskip
The article is structured as follows.
Section \ref{sec::sym_space} recalls the basic notions from the Riemannian geometry of the parameter space $\mathcal{M}=\Pos_{sym}^1(m)$. Section \ref{s::quasi} shows that Grassmannian distributions correspond to quasi-invariant distributions on $G(m,r) \simeq  {\rm SL}(m,\R)/P_r \subset \partial \mathcal{M}$, where $P_r$ is a maximal parabolic subgroup of ${\rm SL}(m,\R)$. Using these notions, Section 
\ref{subsub::Grassmanian_Busemann} computes the so-called $\rho$ functions to determine quasi-invariant distributions on subsets of $\partial\mathcal{M}$.
This geometrical view point extends old results on invariant measures on Stiefel and Grasssmann manifolds, see, e.g. 
\cite{James,Herz,Muir,chi}, to quasi-invariant measures.
Furthermore, it is shown that $\ell_\shp$ is a weighted Busemann function. The statistical consequence of this property is that one can link geometrically data $U$ from $G(r,m)$, viewed as a subset of the boundary $\partial 
\mathcal{M}$ of the parameter space $\mathcal{M}$, see Figure \ref{Fig3}, and elements $\sig$ of $\mathcal{M}$:
For given data point $U=\langle X\rangle \in G(m,r)\subset \partial {\cal M}$
and fixed base point $\Sigma_0\in {\cal M}$,
the log-likelihood
 of the data point $U$ for the parameter $\sig$ is such that 
  \begin{equation}\label{BuseLikeli}
 c\ \ell_U(\sig)=-c\ \log(\frac{{\rm d}\G_\sig}{{\rm d}\G_{\id_m}}(U)) = \lim_{t\to\infty}(d(\sig_0,\gamma(t))-t),
 \end{equation}
where $\gamma(t)$ is a geodesic ray taking some $\sig$ to $U \in \partial {\cal M}$, and $c = 2\sqrt{\frac{m}{(m-1)r}}$.
Here $d(\Sigma_0,\gamma(t))$ is the geodesic distance between $\Sigma_0$ and $\gamma(t)$ in the Riemannian manifold  $\Pos_{sym}^1(m)$, see Figure \ref{Fig3}, with
$d(\Sigma_0,\gamma(t))=\vert\vert \log(\Sigma_0^{-1/2}\gamma(t)\Sigma_0^{-1/2}\vert\vert$ and
$\vert\vert C\vert\vert^2 = {\rm tr}(C^T C)$.

\begin{figure}
\centering
\includegraphics[width=10cm]{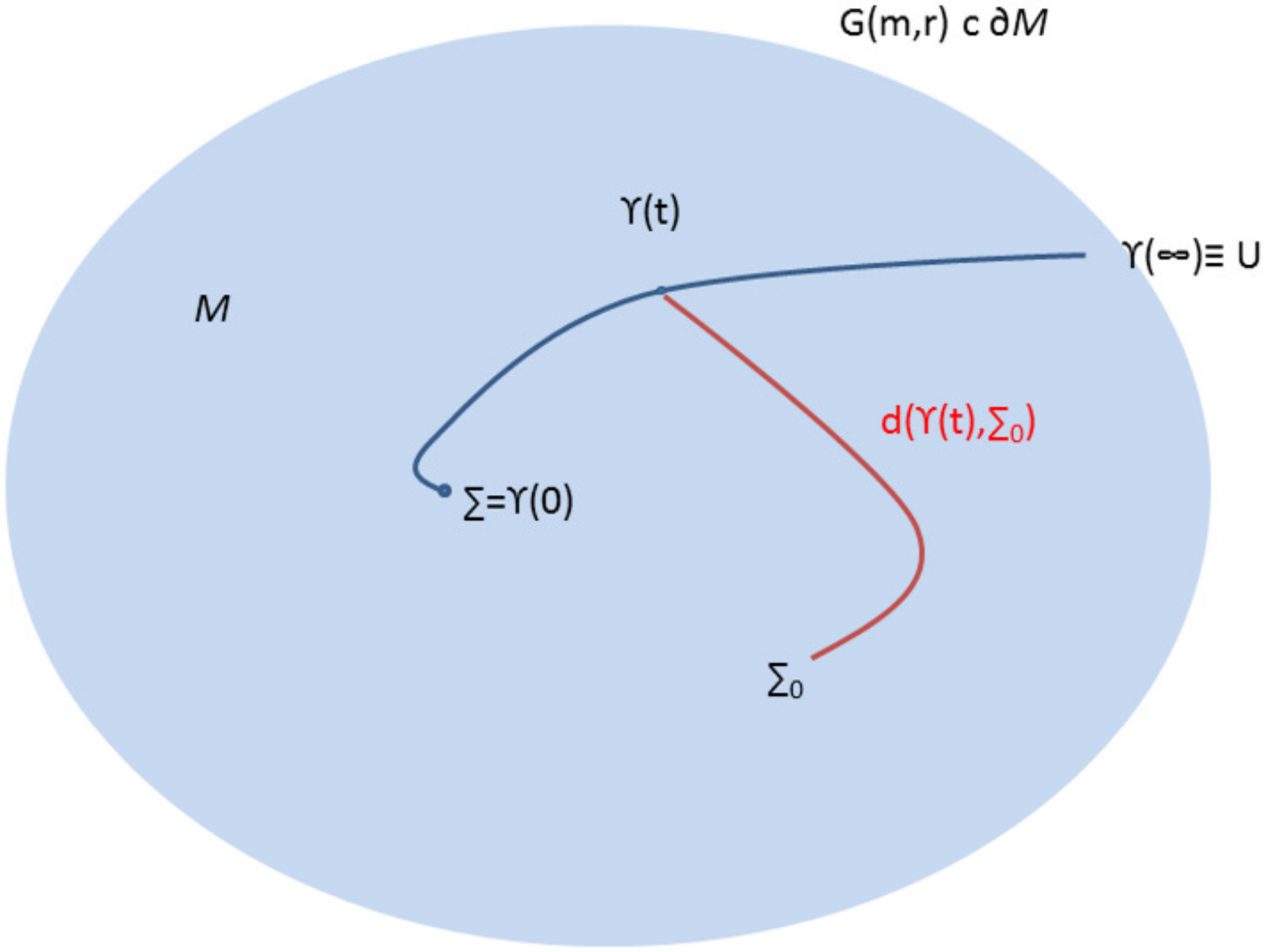}
\caption{$G(m,r)$ is viewed as a subset of the boundary $\partial \mathcal{M}$ of the parameter space $\mathcal{M}$.
The log-likelihood of the data point $U\in G(m,r)$ for the parameter $\sig$ is related to the geodesic distance in the manifold  $\Pos_{sym}^1(m)$ between a fixed base point $\sig_0$ and $\gamma(t)$, when the geodesic ray $\gamma(t)$ points to $U$ at infinity, see (\ref{BuseLikeli}).  As a consequence, statistical properties of GE are strongly related to geometrical properties of $\mathcal{M}$.}
\label{Fig3}
\end{figure}

 Section \ref{sec::Grass_max_like} studies the  gradient and covariant derivative of $\ell_{\shp}$, and Section \ref{s::GeodesicConvexity} proves that under suitable conditions $\ell_\shp$ is a strictly geodesically convex function.
Then,  general results of 
Kapovich--Leeb--Millson \cite{KLM} on Hadamard manifolds yield that
a Borel probability measure $\shp$ on the Grassmannian $\sas$ has a unique \gmle\  if and only if
\begin{equation}\label{Unicity1}
\int_\sas\dim(U\cap V)\,d\shp(U)<\frac rm\dim(V)
\end{equation}
for all nontrivial linear subspaces $V$ of $\ls$ ($0\ne V\ne\ls$). 
For a sample $\{U_i\}_{1 \leq i \leq n }$ of size $n$, let $\shp_n:= \frac{1}{n} (\delta_{U_1}+...+ \delta_{U_n})$ be its corresponding empirical probability measure. Then by \cite[Theorem 4]{AMR} for almost all samples of size $n > \frac{m^2}{r(m-r)}$ the corresponding map $\ell_{\shp_n}$ has a unique GE. When a probability measure $\shp$ or a sample does not have a unique \gmle, it may have several \gmle s or no \gmle s.  Section \ref{s.behaviourinfinity} considers geodesic coercivity using direct geometrical methods, and treats in detail the existence and unicity of \gmle s. Related to this, Corollary \ref{M-equation} shows that
a parameter $\sig\in\pas$ is a \gmle\ of a Borel probability measure~$\shp$ on $\sas$ if and only if it satisfies the
M-equation
\begin{equation}\label{Mequation}
\int_\sas \Pr(U,\sig)\,d\shp(U)=\frac{r}{m} \Id_m,
\end{equation}
where $\Pr(U,\sig)=X (X^{T} \Sigma^{-1} X)^{-1} X^{T}\sig^{-1}$ is the $\Sigma$-orthogonal projector on $U$ for the scalar product $(x\vert y)=x^T \Sigma^{-1}y$, $x$, $y\in\R^m$. 

Section \ref{sec::gen_lln} focus on the almost sure convergence of \gmle s.  More precisely, let $\{U_n\}_{n \geq 1}$ be a random sample of i.i.d.  random subspaces of $G(m,r)$, distributed according 
to a continuous Borel probability measure $\shp$ on $G(m,r)$ with support $G(m,r)$. Let $\shp_n:= \frac{1}{n} (\delta_{U_1}+...+ \delta_{U_n})$ be the corresponding empirical probability measures. Suppose the GEs $\{\Sigma_n \}_{n\geq 1}, \Sigma_\shp \subset \Pos_{sym}^1(m)$ for $\{\shp_n\}_{n \geq 1}$, respectively, $\shp$ exist. 
Theorem \ref{thm::LLN} shows that
for every $n$ large enough the $\Sigma_n$ is unique almost surely, $\Sigma_\shp \in \Pos_{sym}^1(m)$ is unique, and $\{\Sigma_n \}_{n \geq 1}$ converges almost surely to $\Sigma_\shp $.

Section \ref{sec::gen_CLT} considers the asymptotic normality of \gmle. With the above notation, let $\Sigma =g g$ with $g \in \SL(m, \RR)$ the unique symmetric positive-definite square root of $\Sigma$, and set
 $ C_n := \frac{m}{\tr(\Sigma_{\shp}^{-1} \Sigma_n)}g^{-1} \Sigma_n g^{-1}$. If the support of the continuous Borel probability measure $\shp$ is $G(m,r)$ then
Theorem \ref{CLT} shows that
$$
\sqrt{n}(\Veco(C_n - \Id_m)) \xrightarrow[distribution]{n \to \infty}
\mathcal{N}(0, \sigma_\infty^2),$$
where the limiting covariance matrix $\sigma_\infty^2$ is obtained using geometrical arguments.

\section{The geometry of the parameter space $\mathcal{M}=\Pos_{sym}^1(m)$ }
\label{sec::sym_space}
\subsection*{The symmetric space  of $\SL(m, \RR)$ }

Let $m \geq 2$ and consider the semi-simple real Lie group $\SL(m, \RR)$. It is a locally compact topological group that is also connected and with finite centre. Its associated symmetric space is $\mathcal{M}:\cong \SL(m, \RR)/\SO(m, \RR)$, where $\SO(m, \RR)$ is the special orthogonal group and is the maximal compact subgroup of $\SL(m, \RR)$. The symmetric space $\mathcal{M}$ is a Riemannian manifold and it is a \cat space (see \cite[Chapter II.10]{BH99}); in this section we want to precisely identify it. For more details and the proofs of this section, one can consult \cite[Chapter II.10]{BH99}.


Let $$\Sym(m):=\{M \text{ is an } m \times m \text{ matrix } \; \vert \; M=M^{T}\}.$$ Notice $\Sym(m)$ is a vector space isomorphic to $\RR^{m(m+1)/2}$ whose basis are the $m \times m$ matrices $E_{ij}$ having $1$ on the entries $(i,j)$ and $ (j,i)$ and $0$ otherwise. In particular, $E_{ii}$ is the $m \times m$ matrix having $1$ on the entry $(i,i)$ and zero otherwise. On $\Sym(m)$ we consider the \textbf{max-norm} denoted for what follows by $\vert \vert \cdot \vert \vert$, i.e.,  $\vert \vert A \vert \vert:= \max\limits_{i,j} \vert a_{ij} \vert $, for every $A= (a_{ij}) \in \Sym(m)$. 

Because $\Sym(m)$ is isomorphic to $\RR^{m(m+1)/2}$, thus a smooth manifold, the tangent space at every matrix $\Sigma \in \Sym(m)$ is of dimension $m(m+1)/2$. Then the max-norm topology agrees with the manifold topology on $\Sym(m)$.

Moreover,  $$\Pos_{sym}(m):=\{ M \in \Sym(m) \; \vert \; x M x^{T} >0, \forall x \neq 0 \in \RR^{m} \}$$ is an open subset of $\Sym(m)$, and so every matrix $M \in \Pos_{sym}(m)$ has the tangent space of dimension $m(m+1)/2$. 

It is a fact  $$\Pos_{sym}^1(m):= \{M \in \Pos_{sym}(m) \; \vert \; \det(M)=1\}$$ is a totally geodesic submanifold of  $\Pos_{sym}(m)$ and the tangent space at every point $ \Sigma \in \Pos_{sym}^1(m)$ is of dimension $m(m+1)/2-1= (m-1)(m+2)/2$. Moreover, one has  $\Pos_{sym}(m) \cong \Pos_{sym}^1(m) \times \RR$.

\begin{lemma}
\label{lem::action}
The group $\SL(m, \RR)$ acts on $\Pos_{sym}^1(m)$ by $g \cdot \Sigma:= g \Sigma g^{T}$, for every $g \in \SL(m, \RR)$ and every $\Sigma \in \Pos_{sym}^1(m)$. Moreover, $\SL(m, \RR)$ acts transitively on $\Pos_{sym}^1(m)$.
\end{lemma}

Now we want to define a Riemannian metric on $\Pos_{sym}^1(m)$ such that $\SL(m, \RR)$ acts on $\Pos_{sym}^1(m)$ by isometries. To do that it is enough to define a scalar product on the tangent space at the identity $\Id_m \in \Pos_{sym}^1(m)$. Then, by the transitivity of $\SL(m, \RR)$ on $\Pos_{sym}^1(m)$ we transport that scalar product on the tangent space at every point of $\Pos_{sym}^1(m)$.

\begin{lemma}
\label{lem::tangent}
The tangent space at the identity matrix $\Id_m \in \Pos_{sym}^1(m)$ is given by $$\TT_{\Id_m}\mkern-4mu\Pos_{sym}^1(m):=\{M \in \Sym(m)  \;  \vert \; \tr(M)=0\}$$ and has dimension $(m-1)(m+2)/2$.
\end{lemma}

\begin{definition}
On $\TT_{\Id_m}\mkern-4mu\Pos_{sym}^1(m)$ we define the scalar product $$\langle , \rangle_{\Id_m}: \TT_{\Id_m}\mkern-4mu\Pos_{sym}^1(m) \times \TT_{\Id_m}\mkern-4mu\Pos_{sym}^1(m) \to \RR$$ by $\langle  A, B \rangle_{\Id_m}:=\tr(AB)$, for every $A,B \in \TT_{\Id_m}\mkern-4mu\Pos_{sym}^1(m)$.
\end{definition}

\begin{remark}
\label{rem::tangent}
Let $A(t)$ be a curve $A : [-\epsilon, \epsilon] \to \Pos_{sym}^1(m)$ with $A(0)= \Id_m$. Then for every $g \in \SL(m, \RR)$, $g A(t) g^{T}=:gAg^{T}(t)$ is a curve with $g A(0) g^{T}= gg^{T}$. So $$\frac{d gAg^{T}}{dt}\vert_{t=0}= g \frac{A}{dt}g^{T}\vert_{t=0}= gBg^{T} \in \TT_{gg^{T}} \mkern-4mu \Pos_{sym}^1(m)$$ where $B = \frac{A}{dt}= A'(0) \in \TT_{\Id_n}\mkern-4mu\Pos_{sym}^1(m)$. 

In particular, notice $\TT_{gg^{T}} \mkern-4mu \Pos_{sym}^1(m)= \{A \in \Sym(m) \; \vert \; \tr(g^{-1}A (g^{T})^{-1})=0\}$.
\end{remark}

For want follows we want to define a scalar product $\langle , \rangle_{\Sigma}$ on every tangent space $\TT_{\Sigma} \mkern-4mu \Pos_{sym}^1(m)$, with $\Sigma \in \Pos_{sym}^1(m)$, and with the property that every $g \in \SL(m, \RR)$ acts as an \textbf{isometry} on $\Pos_{sym}^1(m)$:
\begin{enumerate}
\item
$dg : \TT_{\Sigma} \mkern-4mu \Pos_{sym}^1(m) \to \TT_{g \cdot \Sigma} \mkern-4mu \Pos_{sym}^1(m)$ is an isomorphism of vector spaces,
\item
$\langle dg A,  dg B \rangle_{ g \cdot \Sigma}= \langle A, B \rangle_{\Sigma}$, for every $A, B \in \TT_{\Sigma} \mkern-4mu \Pos_{sym}^1(m)$.
 \end{enumerate}
 
Notice, for every $g \in \SL(m, \RR)$ the matrix $gg^{T} $ is an element of  $\Pos_{sym}^1(m)$. Moreover, every $\Sigma \in \Pos_{sym}^1(m)$ admits a square root $g \in  \SL(m, \RR)$, not necessarily unique, i.e., $\Sigma:= gg^{T}$. 

\begin{definition}
Let $g \in \SL(m, \RR)$ and take $\Sigma:= gg^{T}$. We define the isomorphism $$dg : \TT_{\Id_n}\mkern-4mu\Pos_{sym}^1(m) \to  \TT_{\Sigma} \mkern-4mu \Pos_{sym}^1(m)$$ by $A \mapsto dg(A):= gAg^{T}$, for every $A \in\TT_{\Id_n}\mkern-4mu\Pos_{sym}^1(m)$. Then we define $$\langle A ,C \rangle_{\Sigma}:=\tr(\Sigma^{-1} A \Sigma^{-1} C),$$
for every $A,C \in  \TT_{\Sigma} \mkern-4mu \Pos_{sym}^1(m)$.
\end{definition}

\begin{lemma}
\label{lem::action}
Let $g \in \SL(m, \RR)$ and take $\Sigma:= gg^{T}$. Then $dg :  (\TT_{\Id_m} \mkern-4mu \Pos_{sym}^1(m), \langle , \rangle_{\Id_m}) \to  (\TT_{\Sigma} \mkern-4mu \Pos_{sym}^1(m), \langle , \rangle_{\Sigma})$ preserves the scalar product. Moreover, $g$ is an isometry of $\Pos_{sym}^1(m)$ and any geodesic in $\Pos_{sym}^1(m)$ passing through $\Sigma$ is of the form $g \exp(tV)g^{T}$, where $V \in\TT_{\Id_m}\mkern-4mu\Pos_{sym}^1(m)$. 

\end{lemma}

From the above lemmas one can isometrically identify $\Pos_{sym}^1(m)$ with $ \SL(m, \RR)/\SO(m, \RR)$, and  as announced in the beginning of this section, we take $\mathcal{M}= \Pos_{sym}^1(m)$.

\begin{remark}
\label{rem::same_top}
It is a fact the topology on $\Pos_{sym}^1(m)$ given by the distance induced from the Riemannian metric on $\Pos_{sym}^1(m)$ agrees with the topology of the submanifold $\Pos_{sym}^1(m)$ of $\Sym(m)$, and thus with the max-norm topology on $(\Pos_{sym}, \vert \vert \cdot \vert \vert)$ induced from $\Sym(m)$.
\end{remark}

\section{Grassmannian distributions as quasi-invariant measures on $\partial \mathcal{M}$}
\label{s::quasi}

Let $\Sigma = g g^T \in \Pos_{sym}^1(m)$ for some $g\in \SL(m, \RR)$. As we have seen, the tangent space $\TT_{\Id_m}\mkern-4mu\Pos_{sym}^1(m)$ consists in zero trace symmetric matrices, and  $\TT_{\Sigma}\mkern-4mu\Pos_{sym}^1(m)$ is obtained by considering matrices  $g A g^T$, for some 
$A\in \TT_{\Id_m}\mkern-4mu\Pos_{sym}^1(m)$. The geodesic ray $\gamma(t)$ issuing from $\Sigma$ of speed $A$  is then 
$\gamma(t)=g\exp( t A)g^T$, see Figure \ref{Fig1}. Let us first explain how one can see   $G(m,r)$ as a subset of $\partial \mathcal{M}$ in a simple way. More precise mathematical constructions are provided later in this Section.
We focus on geodesic rays issuing from $\Id_m$, and require that the geodesic is parametrized by arc length, that is, we suppose that $\langle A,A\rangle_{\Id_m}={\rm tr}(A^2)=1$ with ${\rm tr}(A)=0$. A simple example is given by
$$A= \diag(\underbrace{\lambda_r , \cdots ,\lambda_r }_{r\text{-times}}, \underbrace{-\beta_r , \cdots , -\beta_r }_{m-r\text{-times}}),$$
where $\lambda_r= \sqrt{\frac{m-r}{mr}} $ and $\beta_r= \sqrt{\frac{r}{m(m-r)}}$. This example will occur in the proof of 
Lemma \ref{Busemann1}. Then it is easy to check that
$$
\gamma(t)\sim \exp(\lambda_r t)\diag(\underbrace{1 , \cdots ,1 }_{r\text{-times}}, \underbrace{0 , \cdots , 0 }_{m-r\text{-times}}),
$$
as $t\to\infty$. The limiting geodesic is then asymptotically confined as a $m\ {\rm x}\ r$  matrix $X_0=[e_1,\cdots,e_r]$
or with the linear subspace $U_0$ generated by the unit vectors $e_1,\cdots, e_r$ of $\RR^m$. Similar constructions hold in the general case  for arbitrary tangent vectors of  $\TT_{\Id_m}\mkern-4mu\Pos_{sym}^1(m)$. Such geodesic rays will then ultimately enter some asymptotic linear subspace of $G(m,r)$.

\begin{figure}
\centering
\includegraphics[width=10cm]{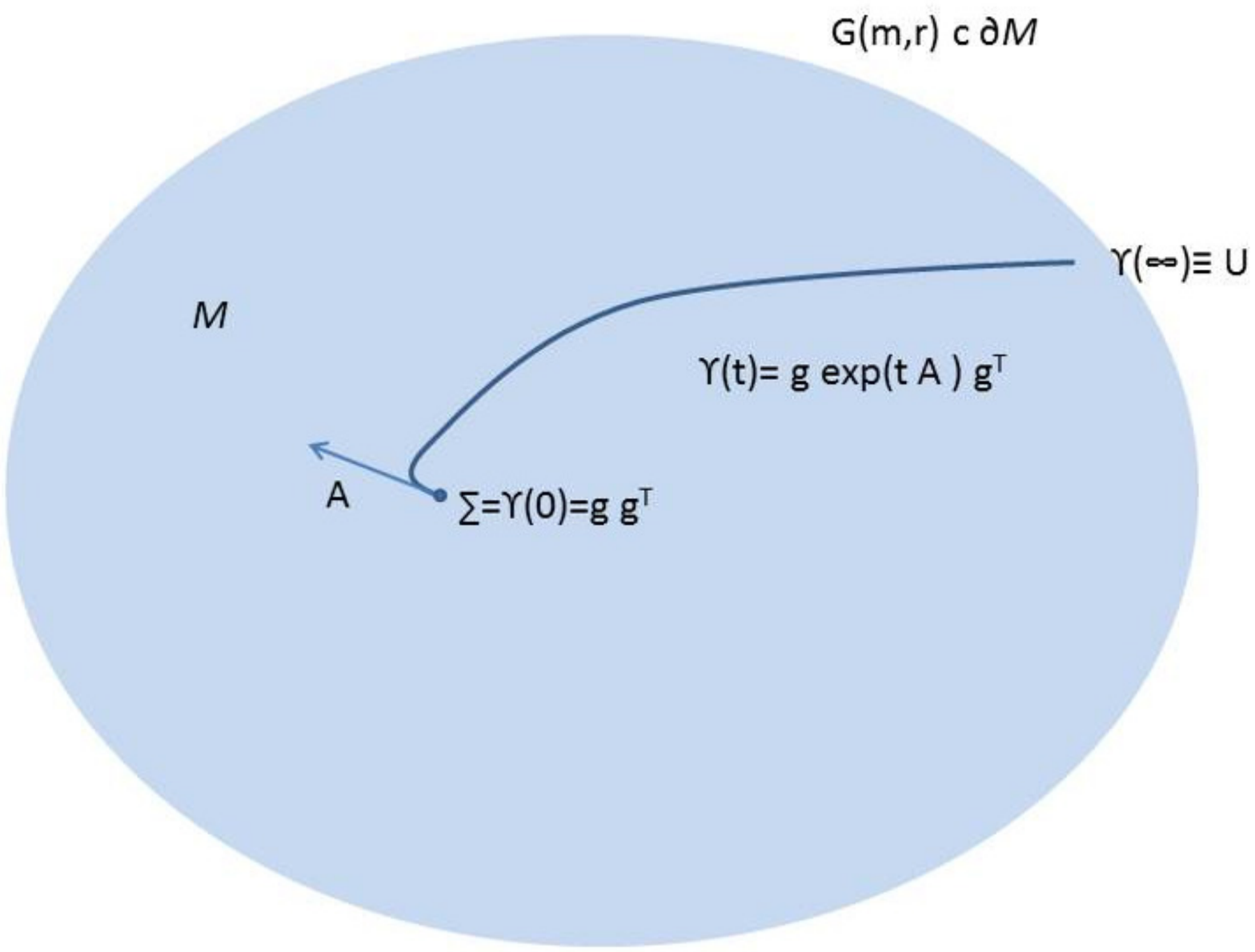}
\caption{Geodesics rays $\gamma(t)$ in $\Pos_{sym}^1(m)$ taking $\Sigma =\gamma(0)$ to a subset of the visual boundary which basically is an element of Grassmann manifold.  }
\label{Fig1}
\end{figure}

As the symmetric space $\Pos_{sym}^1(m) = \mathcal{M}$ of $\SL(m, \RR)$ is a \cat space, it has an associated  visual boundary at infinity  $\partial \mathcal{M}$.  $\partial \mathcal{M}$ is a spherical building of type $A_{m-1}$ in the sense of Tits, i.e., it has the structure of a simplicial complex that is expressed as the union of sub-complexes that are called apartments and satisfy three axioms (see \cite[Chapter II.10]{BH99}). Each apartment is tessellated  with (similar) maximal symplices that are all of the same type $A_{m-1}$ (thus same shape). Such a maximal simplex is called a chamber of $\partial \mathcal{M}$ or a spherical chamber at infinity of $\mathcal{M}$; it has dimension $m-2$ and has $m-1$ vertices. If we start coloring the vertices of a spherical chamber with $m-1$ colors, we can color the vertices of the spherical building $\partial \mathcal{M}$ in such a way that the vertices of each chamber of 
$\partial \mathcal{M}$ are differently colored, but using the same set of  $m-1$ colors. In fact, each color represents a different type of vertex and each chamber has only one vertex-representative for each color (type). 
The spherical building $\partial \mathcal{M}$ is compact with respect to  the cone topology induced from $\mathcal{M}$. If $m=2$ then the model chamber is just a point. 

\subsection{Parabolic subgroups}
\label{subsec::parab}
The group $G:=\SL(m, \RR)$ acts by isometries on $\mathcal{M}$ and continuously on $\partial \mathcal{M}$. Moreover, $G$ acts transitively on the set of all chambers of $\partial M$ by preserving the type of the vertices of the chambers.  Fix for what follows a chamber $c$ of $\partial \mathcal{M}$. The stabiliser in $G$ of a face of the chamber $c$  is called a \textbf{parabolic subgroup} of $G$, i.e., for a face $\sigma$ of $c$  the parabolic subgroup is $P_{\sigma}:= \{ g \in G \; \vert \; g(\sigma)=\sigma\}$. Note $P_{\sigma}$ is a closed subgroup of $G$. Because $G$ acts transitively on the set of all chambers of $\partial M$, the parabolic subgroups corresponding to other chambers are conjugated to those of the chamber $c$. 

For $\sigma=c$, the subgroup $P_c$ is called the Borel subgroup of $G$ and it is the minimal parabolic, i.e., contained in all the other parabolic subgroups $P_\sigma$.  When $\sigma$ is just a vertex of $c$ its corresponding parabolic is maximal, i.e., it is not contained in any larger parabolic subgroup. As the simplex $c$ has $m-1$ vertices, there are exactly $m-1$ maximal parabolic subgroups. It is well known that, up to conjugacy,  the maximal parabolic are of the form
$$P_{r} =\left\{ \begin{pmatrix} A_1 & B  \\
0 & A_2 \\
\end{pmatrix} \in \SL(m, \RR) \; \vert \; A_1 \in \GL(r, \RR), A_{2} \in \GL(m-r, \RR), B \in \RR^{r \times (m-r)}\right\},$$ where $1 \leq r \leq m-1$. Notice, the blocks $A_1, A_2$ are on the diagonal of $P_r$. 

A similar form is known for each parabolic $P_{\sigma}$, but  here we are only interested in maximal parabolics.

\subsection{The sample space $G(m,r)$ as a subset of $\partial \mathcal{M}$}
\label{subsec::grass}

The beginning of Section \ref{s::quasi} illustrates the fact that $G(m,r)$ can be seen as a subset of $\partial M$. More rigorous arguments use the preceding constructions:
It is easy to see the parabolic subgroup $P_r$ is the stabiliser in $G$ of the $r$-dimensional vector subspace of $\RR^m$ that is generated by the first $r$ vectors $\{e_1, \cdots, e_r\}$ of the canonical base  of $\RR^m$; $V$ is a point of $G(m,r)$. We can take the corresponding  $m \times r$ matrix to be $X:= (e_1, e_2, \cdots, e_r)$.

As $\SL(m, \RR)$ acts transitively on the set of all $r$-dimensional vector subspaces of $\RR^m$ that contain he origin $0 \in \RR^m$, one can conclude  the Grassmannian $G(m,r)$ equals as a set the quotient $\SL(m, \RR)/ P_r$, for every $1 \leq r \leq m-1$. Then using Section \ref{subsec::parab} the Grassmannian $G(m,r)$ equals the set of all vertices of $\partial M$ that are of the same type.  Using this identification, the action of an element $ g \in \SL(m, \RR)$ on a vertex of $\partial M$ can be interpreted as the matrix multiplication between $g$ and the corresponding matrix $X$  introduced above.

\subsection{Levi decompositions of maximal parabolics}
It is a well known fact that every maximal parabolic has the following direct product decomposition $P_r= U^{r} M_r$, called  \textbf{Levi decomposition} and where
$$ M_r :=\left\{ \begin{pmatrix} A_1 & 0  \\
0 & A_2 \\
\end{pmatrix} \in \SL(m, \RR) \; \vert \; A_1 \in \GL(r, \RR), A_{2} \in \GL(m-r, \RR)\right\}$$
and 
$$U^{r}:=\left\{\begin{pmatrix} \Id_1 & B  \\
0 & \Id_2 \\
\end{pmatrix} \; \vert \; \Id_r \in \GL(r, \RR), \Id_{2} \in \GL(m-r, \RR),  B \in \RR^{r \times (m-r)}\right\}.$$

Moreover, $U^r$ is normal in $P_r$ and is called the \textbf{unipotent radical} of  $P_r$. $M_r$ is called the \textbf{Levi factor} of $P_r$, it is a reductive group and is not connected.

Further we want to decompose the subgroup $M_r$. For $ A_1 \in \GL(r, \RR), A_{2} \in \GL(m-r, \RR)$, by considering $\frac{\sqrt[r]{\vert \det(A_1) \vert }}{\sqrt[r]{\vert \det(A_1) \vert}} A_1$ and $\frac{\sqrt[m-r]{\vert \det(A_2) \vert }}{\sqrt[m-r]{\vert \det(A_2) \vert } }A_2$ we can define a canonical group homomorphism $ \alpha : M_r \to  T^{r} \times G_r$ where 
$$T^{r}=\left\{\begin{pmatrix} \lambda_1 \Id_1 & 0  \\
0 & \lambda_2 \Id_2 \\
\end{pmatrix} \; \vert \;  \lambda_1, \lambda_2 \in \RR^{*}, \lambda_1^{r} \lambda_2^{m-r}=1 \right\}$$
and 
$$ G_r =\left\{ \begin{pmatrix} A_1 & 0  \\
0 & A_2 \\
\end{pmatrix} \in \SL(m, \RR) \; \vert \; \det(A_{1})= \pm 1= \det(A_{2})\right\}.$$
Notice $G_r$ and $T^r$ are subgroups of $M_r$, but one can see $M_r$ does not decompose as the direct product $T^{r} \times G_r$.  The subgroup $G_r$ is known to be semi-simple. Furthermore, the subgroup $T^r$ is obviously normal in $P_r$. We consider
$$P_r^{0}:= \left\{\begin{pmatrix} A_1 & B  \\
0 & A_2 \\
\end{pmatrix}  \in \SL(m, \RR) \; \vert \; \det(A_{1})= \pm 1= \det(A_{2}),  B \in \RR^{r \times (m-r)}\right\}$$
that is a subgroup of $P_r$. It is a known fact the $P_r$-orbit of $\Id_m$ in $\Pos_{sym}^1(m)$ equals $\Pos_{sym}^1(m)$ and the $P_r^{0}$-orbits in $\Pos_{sym}^1(m)$ are the horospheres corresponding to the point at infinity stabilized by $P_r$.

Similar decompositions of the corresponding Levi factor are known for each parabolic $P_{\sigma}$. 

\subsection{Measures on $\SL(m, \RR)/P_r$}

Let $G$ be a locally compact group, $H$ a closed subgroup of $G$. As we are working in the setting of locally compact groups, all Haar measures that are used in this paper are considered to be left-invariant. We denote by $dx$, respectively, $dh$ the Haar measure on $G$, respectively, $H$.

Endow $G/H$ with the quotient topology, meaning that the canonical projection $p :G \to G/H$ is continuous and open. In practice, one needs to put a measure on $G/H$ that is explained below.

\begin{definition}(See Bekka--de la Harpe--Valette~\cite[Appendix~B]{BHV})
\label{def::rho_function}
A \textbf{rho-function} of $(G,H)$ is a continuous function $\rho : G \to \RR^{*}_{+}$ satisfying the equality
\begin{equation}
\label{equ::rho_fct_cond}
\rho(xh)=\frac{\Delta_H(h)}{\Delta_G(h)} \rho(x) \text{ for all } x \in G, h \in H,
\end{equation}where $\Delta_{G}, \Delta_H$ are the modular functions on $G$, respectively on $H$. 
\end{definition}

We have the following relation between rho-functions of $(G,H)$ and $G$--quasi-invariant regular Borel measures on $G/H$. For the definition of a $G$--quasi-invariant regular Borel measure see~\cite[Appendix.~A.3]{BHV}. To briefly give the idea, suppose the topological space $G/H$ is endowed with a measure $\mu$. As $G$ acts on the space $G/H$ by the left multiplication, one can ask how the action of $G$ on $G/H$ can modify the measure $\mu$. More precisely, for each $g \in G$  what is the relation between $(G/H, \mu)$ and $G/H, g_{\star} \mu)$, where $ g_{\star} \mu$ is the pushforward measure on $G/H$ induced from the map $x \in G/H  \mapsto g \cdot x \in G/H$? In general, the passage from $g_{\star} \mu$ to $\mu$ is given by a function on $G/H$, called the Radon--Nikodym derivative between $g_{\star} \mu$ and $\mu$. If this Radon--Nikodym derivative is just one, then  $g_{\star} \mu= \mu$, and $\mu$ is called \textbf{$g$--invariant}.  If $\mu$ is $g$--invariant for every $g \in G$, then $\mu$ is called \textbf{$G$--invariant}, otherwise $\mu$ is called \textbf{$G$--quasi-invariant}.

\begin{theorem}(\cite[Thm.~B.1.4]{BHV})
\label{thm::measure_rho_function}
Let $G$ be a locally compact group and $H$ be a closed subgroup of $G$. Then there exists a rho-function of $(G,H)$. 

Moreover, with a rho-function $\rho$ of $(G,H)$ there is associated a $G$--quasi-invariant regular Borel measure $\mu$ on $G/H$ whose corresponding Radon--Nikodym derivative satisfies the relation $\frac{dg_{\star} \mu}{d\mu}(xH)=\frac{\rho(gx)}{\rho(x)}$, for every $g,x \in G$ and such that 
\begin{equation}
\label{eq::measure_mu}
\int\limits_{G} f(x)\rho(x) dx= \int\limits_{G/H} \int\limits_{H} f(xh)dh d\mu(xH)
\end{equation}
for every $f \in C_{c}(G)$, continuous complex-valued function on $G$ with compact support. 

Conversely, with a continuous $G$--quasi-invariant regular Borel measure $\mu$ on $G/H$ there is associated a rho-function of $(G,H)$, where continuous means the Radon--Nikodym derivative of $\mu$ is a continuous map.
\end{theorem}

As by Theorem~\ref{thm::measure_rho_function} rho-functions of $(G,H)$ always exist, we have that $G$--quasi-invariant regular Borel measures always exist on $G/H$. We do not intend to clarify here the terminology (e.g., the rho-function associated with a measure) used in Theorem~\ref{thm::measure_rho_function}. Those are explained in~\cite{BHV} and the references therein. 

\begin{remark}
\label{rem::measure_Ri}
Given a rho-function $\rho$ of $(G,H)$, its corresponding $G$--quasi-invariant regular Borel measure $\mu$ on $G/H$ (given by relation \eqref{eq::measure_mu}) is obtained by applying the Riesz--Markov--Kakutani representation theorem to a specific positive linear functional on $C_{c}(G/H)$. More precisely, by \cite[Lem.~B.1.2]{BHV} the linear mapping $$T_{H}: C_{c}(G) \to C_{c}(G/H),  \; \; f \mapsto T_{H}(f) \; \; \text{ given by } (T_{H}(f))(xH):= \int\limits_{H} f(xh)dh$$
is surjective. Moreover, by \cite[Lem.~B.1.3 (iii)]{BHV} and because $T_{H}$ is surjective, the mapping $$T_{H}(f) \mapsto  \int\limits_{G} f(x) \rho(x)dx, \; \;  \text{ for } f \in C_{c}(G)$$
is a well-defined positive linear functional on $C_{c}(G/H)$. Then apply the Riesz--Markov--Kakutani representation theorem to obtain the regular Borel measure $\mu$ on $G/H$ that is also $G$--quasi-invariant.
\end{remark}

\medskip

For the rest of the article we apply Theorem \ref{thm::measure_rho_function} to the case $$G:=\SL(m, \RR) \text{ and } H:=P_r$$ a maximal parabolic subgroup of $\SL(m, \RR)$. Recall $G=\SL(m, \RR)$ is unimodular, thus $\Delta_G(g)=1$ for every $g \in G$. Also, as $\SL(m, \RR)= K P_r$, where $K=\SO(m,\RR)$, one example of rho-functions for $(G,P_r)$ is the $K$-invariant function $\rho: G \to \RR^{*}_{+}$ such that $\rho(k)=1$ for every $k \in K$ and $\rho(h)=\Delta_{P_r}(h)$ for every $h \in P_r$. This  function  $\rho$ gives the $G$--quasi-invariant regular Borel measure $\mu$ on $G/P_r$ that is $K$-invariant and whose corresponding Radon--Nikodym derivative satisfies the relation 
\begin{equation}
\label{equ::Pr_rho}
\frac{dg_{\star} \mu}{d\mu}(xP_r)=\frac{\rho(gx)}{\rho(x)} \text{ for every } g,x \in G.
\end{equation}

We claim the family of \gr\ distributions $\G_\sig$ on $\sas$ of parameter $\sig \in \Pos_{sym}^1(m)$  is in fact the family of $G$--quasi-invariant regular Borel measures $g_{\star}\mu$ on $G/P_r$ of parameter $g \in G$. This is the goal of the next two Sections  \ref{subsub:: modular function} and \ref{subsub::Grassmanian_Busemann}.

\subsubsection{The modular function $\Delta_{P_r}$ of $P_r$}
\label{subsub:: modular function}

In order to have an explicit formula for the $\rho$ function defined above, and so an explicit description of the Radon--Nikodym derivative, we need to compute the modular function $\Delta_{P_r}$ of $P_r$. 
 
By \cite[Proposition 8.27]{Knpp} the modular function of the real Lie group $P_r$ is given by $$ h \in P_r \mapsto \Delta_{P_r}(h)=\vert \det(\Ad(h))\vert, $$ where $\Ad :P_r \to \GL(\mathfrak{p}_r)$ is the adjoint representation of $P_r$ on its Lie algebra $\mathfrak{p}_r \subset \mathfrak{sl(m,\RR)} $. Recall $\Ad(h)X=hXh^{-1}$, for every $h \in P_r$ and $X \in \mathfrak{p}_r $ and 
$$\mathfrak{p}_r  =\{ \begin{pmatrix} A_1 & B  \\
0 & A_2 \\
\end{pmatrix} \; \vert \; A_1 \in \ML(r, \RR), A_{2} \in \ML(m-r, \RR), tr(A_1) + tr(A_2)=0, B \in \RR^{r \times (m-r)}\}.$$
 
\begin{lemma} 
For every $g \in G_r$ and every $u \in U^{r}$ we have $\Delta_{P_r}(g)= \Delta_{P_r}(u)=1$.
\end{lemma}
\begin{proof}
 This is because $G_r$ is generated by unipotent elements and those have modular function $1$ in $P_r$. 
 \end{proof}

Therefore, in order to compute the modular function of $P_r$ it is enough to compute the modular function on the part $T^r$ of $P_r$.  
 
 \begin{lemma} \label{lem::3.5}
 Let $t \in T^{r}, t=\begin{pmatrix} \lambda_1 \Id_1 & 0 \\
0 & \lambda_2 \Id_2 \end{pmatrix}$ such that $\lambda_1, \lambda_2 \in \RR^{*}, \lambda_1^{r} \lambda_2^{m-r}=1$. Then $\Delta_{P_r}(t)= \vert \det(\Ad(t))\vert = \vert (\lambda_1 \lambda_2^{-1})^{r(m-r)} \vert =\vert \lambda_1^{mr} \vert$.
\end{lemma}
\begin{proof}
 By applying such $t$ to $\mathfrak{p}_r$ and computing  $\Ad(t)$ one obtains $\Ad(t)$ is diagonal as a matrix of $\GL(\mathfrak{p}_r)$ and $\det(\Ad(t))= (\lambda_1 \lambda_2^{-1})^{r(m-r)}=\lambda_1^{mr}$.
 \end{proof}

\subsubsection{Rho-function, Grassmanian distributions and Busemann functions}
\label{subsub::Grassmanian_Busemann}

As mentioned before, $P_r$ is the stabiliser in $G$ of the $r$-dimensional vector subspace  $U_0$ of $\RR^m$ that is generated by the first $r$ vectors $\{e_1, \cdots, e_r\}$ of the canonical base of $\RR^m$; $U_0$ is a point of $G(m,r)$. We can take the corresponding  $m \times r$ matrix to be $X_0:= (e_1, e_2, \cdots, e_r)$.
This Section deals with Grassmanian and quasi-invariant distributions, and one of its main results states the following.

\begin{proposition}
\label{rem::Grass_distribution_rho}
Let $g \in G= \SL(m, \RR)$ and $  \Sigma:=g g^{T} \in \Pos_{sym}^1(m)$. Then the family of \gr\ distributions $\G_\sig$ on $\sas$ of parameter $\sig \in \Pos_{sym}^1(m)$  is the same as the family of $G$--quasi-invariant regular Borel measures $g_{\star}\mu$ on $G/P_r$ of parameter $g \in G$. More precisely, equality \eqref{density} from the Introduction can be written  
\begin{equation}
\label{equ::rho_grass_densities}
\frac{dg_{\star} \mu}{d\mu}(hP_r)=\frac{\rho(gh)}{\rho(h)}= \frac{d\G_\sig}{d\G_{\Id_m}}(\langle
X\rangle)= \left( \frac{\det(X^TX)}{\det(X^T\sig\inv X)} \right)^{ m/2}, 
\end{equation}
for every  $h \in \SL(m, \RR)$, where $X = h^{-1} X_0$.
\end{proposition}

\begin{proof}
This proposition is a direct consequence of equality \eqref{equ::Pr_rho} above,  and Lemma \ref{lem::3.7} and equality \eqref{equ::rho_fct_det} below. Indeed, by equality \eqref{equ::rho_fct_det} below and because $\Id_{r}=X_0^{T}X_0$ one has $\rho(h)= \left( \frac{\det(X_0^TX_0)}{\det(X_0^{T}(h^T)^{-1} h^{-1} X_0)} \right)^{ m/2}= \left( \frac{1}{\det(X^{T} X)} \right)^{ m/2}$. By the same quality \eqref{equ::rho_fct_det} below and as $\Id_{m}= X_0 X_0^T$ we also have
\begin{equation*}
\begin{split}
\rho(gh)&=\left( \frac{\det(X_0^TX_0)}{\det(X_0^{T}(h^Tg^{T})^{-1} (gh)^{-1} X_0)} \right)^{ m/2}= \left( \frac{1}{\det(X_0^{T}(g^{T})^{-1} (h^{T})^{-1} h^{-1} g^{-1} X_0)} \right)^{ m/2} \\
&= \left( \frac{1}{\det(X_0^{T}(g^{T})^{-1} X_0 X_0^{T}(h^{T})^{-1} X_0 X_0^{T} h^{-1} X_0 X_0^{T} g^{-1} X_0)} \right)^{ m/2}\\
&=  \left( \frac{1}{\det(X_0^{T}(g^{T})^{-1} X_0) \cdot  \det(X_0^{T}(h^{T})^{-1} X_0) \cdot  \det(X_0^{T} h^{-1} X_0) \cdot \det(X_0^{T} g^{-1} X_0)} \right)^{ m/2}\\
&= \left( \frac{1}{ \det(X_0^{T}(h^{T})^{-1} X_0) \cdot \det(X_0^{T}(g^{T})^{-1} X_0)  \cdot \det(X_0^{T} g^{-1} X_0) \cdot \det(X_0^{T} h^{-1} X_0)} \right)^{ m/2}\\
&= \left( \frac{1}{ \det(X_0^{T}(h^{T})^{-1}(g^{T})^{-1} g^{-1} h^{-1} X_0)} \right)^{ m/2}=  \left( \frac{1}{ \det(X \Sigma^{-1} X)} \right)^{ m/2}.
\end{split}
\end{equation*}
 \end{proof}
 

\begin{remark}
By using appropriate parabolic subgroups of $\SL(m, \RR)$ and computing their corresponding rho-functions one can give a similar geometrical/Lie theoretical interpretation as in Proposition \ref{rem::Grass_distribution_rho} for known density functions on Stiefel manifolds. 
\end{remark}

\begin{lemma}\label{lem::3.7}
 Let $1 \leq r \leq m-1$. Let  $X_0:= (e_1, e_2, \cdots, e_r)$ be the $m \times r$ matrix, where $\{e_1, \cdots, e_r\} \subset \RR^m$ are the first $r$ vectors of the canonical base of $\RR^m$.
Then for every $h \in G= K P_r $ we have 
\begin{equation}
\label{equ::rho_fct_det}
\rho(h)=  \left( \frac{\det(X_0^TX_0)}{\det(X_0^{T}(h^T)^{-1} h^{-1} X_0)} \right)^{ m/2}.
\end{equation}
\end{lemma}

\begin{proof}
Notice $\Id_{m}= X_0 X_0^T$, $\Id_{r}=X_0^{T}X_0$,  $X_0^{T}(h^T)^{-1} h^{-1} X_0$ is positive semidefinite,  and $\left( \det(X_0^{T}(h^T)^{-1} h^{-1} X_0) \right)^{-m/2}$ is a continuous function on $G$. Thus $\det(X_0^TX_0)=1$.

If $h\in K= \SO(m, \RR)$ then $(h^T)^{-1} h^{-1}=\Id_m$ and as, by definition, $\rho(h)=1$ equality \eqref{equ::rho_fct_det}  follows. 

Now let  $h \in P_r$ and recall the fact that $\det(AB)= \det(A) \det(B)$ for square matrices. By writing $h= u t h_r \in P_r$ with $u \in U^r$, $t \in T^r$ and  $h_r \in G_r$ we have $$X_0^{T}h^{-1} X_0= X_0^{T}h_r^{-1} X_0 X_0^{T} t^{-1} X_0 X_0^{T} u^{-1} X_0 $$
and analogously for $X_0^{T}(h^{T})^{-1} X_0$. Then just by matrix multiplication it is easy to verify that $\det(X_0^{T}(h_r^T)^{-1} X_0X_0^{T}h_r^{-1} X_0)=1$ and $\det(X_0^{T}(u^{T})^{-1} X_0 X_0^{T} u^{-1} X_0)=1$. Therefore, $$\det(X_0^{T}(h^T)^{-1} h^{-1} X_0)= \det(X_0^{T}(t^T)^{-1} t^{-1} X_0).$$

Next we show that $\left( \det(X_0^{T}(t^T)^{-1} t^{-1} X_0) \right)^{-m/2}= \vert \det(\Ad(t))\vert $ for every $t \in T^r.$  Indeed, for $t \in T^{r}, t=\begin{pmatrix} \lambda_1 \Id_1 & 0 \\
0 & \lambda_2 \Id_2 \end{pmatrix}$ such that $\lambda_1, \lambda_2 \in \RR^{*}, \lambda_1^{r} \lambda_2^{m-r}=1$, it is just an easy computation of matrix multiplication to obtain $\det(X_0^{T}(t^T)^{-1} t^{-1} X_0) = \lambda_1^{-2r}$. And, by Lemma \ref{lem::3.5}  we have $\rho(h)= \left( \frac{\det(X_0^TX_0)}{\det(X_0^{T}(h^T)^{-1} h^{-1} X_0)} \right)^{ m/2}$, when  $h \in P_r$.

By the Definition \ref{def::rho_function} of the rho-function it remains to show
 $$\left( \det(X_0^{T}((xh)^T)^{-1} (xh)^{-1} X_0) \right) ^{-m/2}= \Delta_{P_r}(h) \left( \det(X_0^{T}(x^T)^{-1} x^{-1} X_0) \right)^{-m/2}$$ for every $x \in G$ and $h \in P_r$. But this follows from the above computation.
  
\end{proof}


By taking the $\log$ of the rho-function  $\rho$ given by equality \eqref{equ::rho_fct_det} we obtain a Busemann function; this gives a different interpretation of the rho-function $\rho$. For the definition and main properties of  Busemann functions see \cite[Chapter II.8]{BH99}.

In our setting, it is sufficient to be aware that, using the notations of the beginning of Section \ref{s::quasi},
for given data point $U=\langle X\rangle \in G(m,r)\subset \partial {\cal M}$
and fixed base point $\Sigma_0\in {\cal M}$,
the log-likelihood
 of the data point $U$ for the parameter $\sig$ is such that (see (\ref{BuseLikeli}) and  Figure \ref{Fig2})
  $$
 c\ \ell_U(\sig)=-c\ \log(\frac{{\rm d}\G_\sig}{{\rm d}\G_{\Id_m}}(U)) = \lim_{t\to\infty}(d(\sig_0,\gamma(t))-t),
$$
where $\gamma(t)$ is a geodesic ray taking some $\sig$ to $U \in \partial {\cal M}$, and $c = 2\sqrt{\frac{m}{(m-1)r}}$.
$d(\Sigma,\gamma(t))$ is the geodesic distance between $\Sigma_0$ and $\gamma(t)$ of $\Pos_{sym}^1(m)$, with
$d(\Sigma_0,\gamma(t))=\vert\vert \log(\Sigma_0^{-1/2}\gamma(t)\Sigma_0^{-1/2}\vert\vert$ and
$\vert\vert C\vert\vert^2 = {\rm tr}(C^T C)$. General results on Busemann functions for CAT(0) spaces show that such functions are generically convex.

\begin{figure}
\centering
\includegraphics[width=10cm]{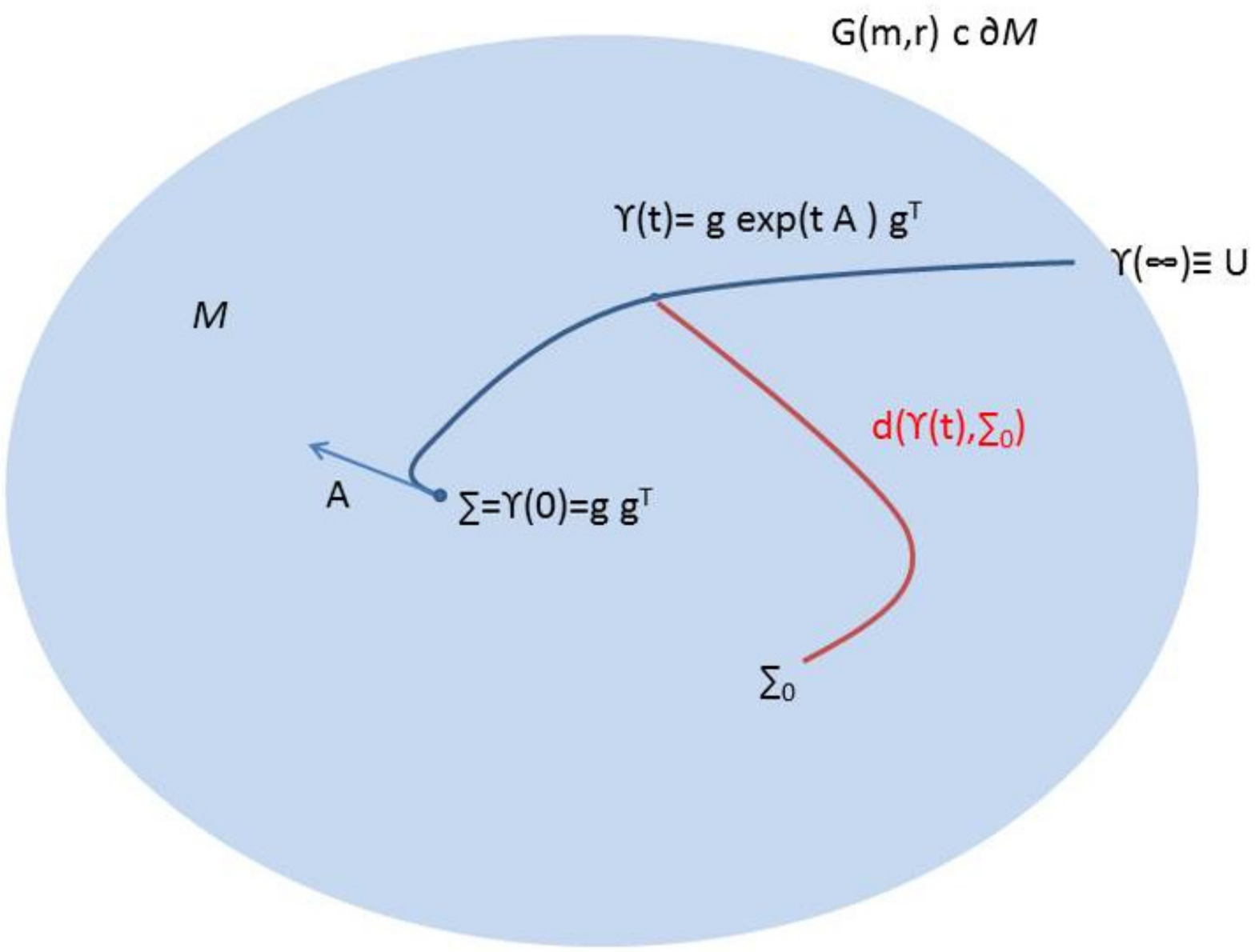}
\caption{The log-likelihood of the data point $U\in G(m,r)$ for the parameter $\sig$ is related to the geodesic distance in the manifold  $\Pos_{sym}^1(m)$ between a fixed base point $\sig_0$ and $\gamma(t)$, when the geodesic ray $\gamma(t)$ issuing from $\sig$ of speed $A$ points to $U$ at infinity, see (\ref{BuseLikeli}). }
\label{Fig2}
\end{figure}

\begin{lemma}\label{Busemann1}
Let $1 \leq r \leq m-1$. Let $U \in G(m, r)$ be a linear span of $r$ linearly independent vectors say $\{x_1, x_2, \cdots, x_r\} \subset \RR^m$, with associated $m \times r$ matrix $X$ given by $(x_1, x_2, \cdots, x_r)$, with $x_i$ the columns of $X$ that are written in the canonical base of $\RR^m$.
Then the  function $$b_{U} : \Pos_{sym}^1(m) \to \RR$$ 
$$ \Sigma \mapsto b_{U}(\Sigma):= \sqrt{\frac{m}{(m-r)r}} \cdot \log \frac{\det(X^{T} \Sigma^{-1} X)}{\det(X^{T}X)}$$
is a Busemann function on $\Pos_{sym}^1(m) $.
\end{lemma}
\begin{proof}
The lemma follows from Dru\c{t}u \cite[Section 2.3, Lemma 3.2.1 ]{Dr} or from Fl\"{u}ge-- Ruh \cite[Theorem 12]{FR}. For completeness we still give the idea of the proof.


We claim the value of $\frac{\det(X^{T} \Sigma^{-1} X)}{\det(X^{T}X)}$ does not depend on the chosen base $\{x_1, x_2, \cdots, x_r\}$ $\subset \RR^m$ of $U$, and thus it does not depend on the matrix $X$.  Indeed, this is because a change of base of $U$ will give a $r \times r$ matrix $B$ and the new base of $U$ will have the corresponding matrix equal to $XB^T$. Then by using $\det(CD)= \det(C) \det(D)$, where $C,D$ are square matrices, the claim follows. 

Secondly, we prove the lemma in the special case when $U$ is the $r$-dimensional vector subspace of $\RR^m$ that is generated by the first $r$ vectors $\{e_1, \cdots, e_r\}$ of the canonical base of $\RR^m$. We denote this subspace by $U_0$ and we take the corresponding  $m \times r$ matrix to be $X_0:= (e_1, e_2, \cdots, e_r)$. Then $\det(X_0^{T}X_0)=1$ and we claim $ b_{U_0}(\Sigma)= \sqrt{\frac{m}{(m-r)r}} \cdot \log \det(X_0^{T} \Sigma^{-1} X_0)$ is a Busemann function corresponding to the point at infinity $U_0 \in G(m,r) \subset \partial \Pos_{sym}^1(m)$.  

To prove that, we consider the diagonal matrix $A:= \diag(\underbrace{\lambda_r , \cdots ,\lambda_r }_{r\text{-times}}, \underbrace{-\beta_r , \cdots , -\beta_r }_{m-r\text{-times}})$, where $\lambda_r= \sqrt{\frac{m-r}{mr}} $ and $\beta_r= \sqrt{\frac{r}{m(m-r)}}$. One can see $A \in \TT_{\Id_m} \mkern-4mu \Pos_{sym}^1(m)$ and $ \langle A, A\rangle_{\Id_m}= \tr(AA)=1$. Then the map $\gamma_r: \RR \to \Pos_{sym}^1(m)$  given by $t \mapsto \gamma_r(t):= \exp(tA)$ is the bi-infinite geodesic line parameterized with respect to the arc length and such that $\gamma_r(0)= \Id_m \in \Pos_{sym}^1(m)$ and  $\gamma_r(\infty)= U_0 \in  G(m,r)$. Notice $P_r$ is the parabolic subgroup corresponding to $U_0 \in G(m,r)$. By Dru\c{t}u \cite[Section 2.3]{Dr} we need to prove:
\begin{enumerate}
\item[i)] 
$b_{U_0}: \Pos_{sym}^1(m) \to \RR$ is $P_r^{0}$-invariant, i.e., for every $g \in P_r^{0}$ and for every $\Sigma \in \Pos_{sym}^1(m)$ we have $b_{U_0}(\Sigma)= b_{U_0}(g \Sigma g^{T})$. Indeed, as $g^{-1} = \begin{pmatrix} A_1 & B  \\
0 & A_2 \\
\end{pmatrix}  \in \SL(m, \RR)$ with $\det(A_{1})= \pm 1= \det(A_{2})$, and $B \in \RR^{r \times (m-r)}$, then $g^{-1}X_0= \begin{pmatrix} A_1 \\
0  \\
\end{pmatrix}= X_0 A_1$.  So $X_0^{T}(g^{-1})^{T}= A_1^{T}X_0^{T}$ that gives a change of base for the subspace $U_0$.  As $\det(A_1^{T} A_1)=1$ the $P_r^{0}$-invariance follows.
\item[ii)]
$b_{U_0}(\gamma_r(t))=-t$, for every $t \in \RR$. Indeed, just by matrix multiplication we have $$\det (X_0^{T} \gamma_r(t)^{-1} X_0)= \det(X_0^{T} \exp(-tA) X_0)= e^{-t \cdot r \cdot \lambda_r}.$$ So $b_{U_0}(\gamma_r(t))=-t \cdot r \cdot \lambda_r \cdot \sqrt{\frac{m}{(m-r)r}} =-t $.
\end{enumerate}
Thus, for the spacial case of $U_0$ the lemma holds true.

It remains to prove the lemma in the general case of $U$.  Indeed, as $ \SL(m, \RR)=\SO(m, \RR) P_r$ acts transitively on $G(m,r)$,  there is $k \in K=\SO(m, \RR)$ such that $U= k U_0$. Then by an easy computation we have the equality $b_{U}(\Sigma)= b_{U_0}(k^{-1}\Sigma (k^{T})^{-1})-  b_{U_0}((k^{T}k)^{-1})= b_{U_0}(k^{-1}\Sigma (k^{T})^{-1})$, for every $\Sigma \in \Pos_{sym}^1(m)$.  Similarly as for the case $U_0$ we need to prove $b_{U}(kk^{T})=0$, $b_{U}(\Sigma)$ is $kP_r^{0}k^{-1}= kP_r^{0}k^{T}$-invariant and $b_{U}(k \gamma_r(t) k^{T})=-t$, for every $t \in \RR$. But this follows easily from $b_{U}(\Sigma)= b_{U_0}(k^{-1}\Sigma (k^{T})^{-1})$, for every $\Sigma \in \Pos_{sym}^1(m)$.
\end{proof}

Let $\shp$ be a Borel probability measure on $G(m,r)$. Then we remark  the map $$b_{\shp} : \Pos_{sym}^1(m) \to \RR$$  $$\Sigma \mapsto b_{\shp}(\Sigma): = \int_{G(m,r)} b_{U}(\Sigma) d\shp(U)$$  is a weighted Busemann function in the sense of Kapovich--Leeb--Millson \cite{KLM}.
The necessary and sufficient condition (\ref{Unicity1})  was given in \cite{AMR} without explicit proof, as a Corollary of  general results of \cite{KLM}. We will provide more explicit results in Section 
\ref{s.behaviourinfinity} using direct geometrical methods.
The interested reader can consult \cite{AMR} where examples are provided.

\section{Gradient and covariant derivative of $\ell_{\shp}$}
\label{sec::Grass_max_like}

%
%

For what follows we want to find the \gmle\  that minimises the map $\ell_{\shp} $ (see (\ref{log-likelihood})).  To do that we need to compute the gradient of $\ell_{\shp}(\Sigma)$, solve the equation $\grad \ell_{\shp}(\Sigma) =0$ and find the critical points $\Sigma \in \Pos_{sym}^1(m)$. Then it is enough to compute the gradient of $\ell_{U}(\Sigma)$ as:
$$  \grad \ell_{\shp}(\Sigma) = \int\limits_{G(m,r)} \grad \ell_{U}(\Sigma) d\shp(U).$$


\subsection{The gradient of $\ell_{U}(\Sigma)$}
Let $\Sigma \in \Pos_{sym}^1(m)$. Fix for that follows $g \in \SL(m, \RR)$ such that $\Sigma= gg^{T}$, i.e., such square root $g$ of $\Sigma$ always exists. Recall from Section \ref{sec::sym_space} (Lemma \ref{lem::action}) that any geodesic in $\Pos_{sym}^1(m)$ passing through $\Sigma$ is of the form $g \exp(tV)g^{T}$, where $V \in\TT_{\Id_m}\mkern-4mu\Pos_{sym}^1(m)$. 

\begin{lemma}
\label{lem::diff_det}
For a curve $A: \RR \to  \SL(m, \RR)$ we have $$\frac{d}{dt} \det(A(t))= \det(A(t)) \cdot \tr(A(t)^{-1} \frac{d}{dt} A(t)).$$
\end{lemma}
\begin{proof}
The result is well known.
\end{proof}

\begin{lemma}
\label{lem::diff_log_like}
Let $g \exp(tV)g^{T}$ in $\Pos_{sym}^1(m)$ be a geodesic curve passing through $\Sigma= gg^{T}$, where $V \in\TT_{\Id_m}\mkern-4mu\Pos_{sym}^1(m)$. 
Then $$d(\ell_{U})_{\Sigma}(gVg^{T})=\frac{d \ell_{U}(g \exp(tV)g^{T})}{dt} \vert_{t=0}= -\frac{1}{2} \cdot \tr(X (X^{T} \Sigma^{-1} X)^{-1} X^{T} (g^{T})^{-1} Vg^{-1} ).$$
\end{lemma}

\begin{proof}
The proof is provided in Section \ref{sec::Appendix}.
\end{proof}

In order to find the formula for the gradient of $\ell_{U}$ we have to use the definition of the gradient that says that
 $$\langle \grad \ell_{U}(\Sigma) , W \rangle_{\Sigma}= d(\ell_{U})_{\Sigma}(W)$$ 
 for every $\Sigma=  gg^{T} \in \Pos_{sym}^1(m)$  and every $W \in \TT_{\Sigma} \mkern-4mu \Pos_{sym}^1(m)= g (\TT_{\Id_m} \Pos_{sym}^1(m))g^{T}$.
 
 \begin{lemma}
 \label{lem::grad_formula}
Let $\grad \ell_{U}: \Pos_{sym}^1(m) \to  \TT  \Pos_{sym}^1(m)$ be the gradient of $\ell_{U}$. Then for every $\Sigma \in  \Pos_{sym}^1(m)$ we have $$\grad \ell_{U}(\Sigma)= \frac{r}{2m} \Sigma-\frac{1}{2} X (X^{T} \Sigma^{-1} X)^{-1} X^{T} \in \TT_{\Sigma} \mkern-4mu \Pos_{sym}^1(m).$$
 \end{lemma}

\begin{proof}
The proof is provided in Section \ref{sec::Appendix}.
\end{proof}

\subsection{The covariant derivative of $\grad \ell_{U}(\Sigma)$}

Given a vector field $Z$ in $\TT \Pos_{sym}^1(m)$, let us now compute the Levi--Civita connection $\nabla_{Z}$ associated with $\Pos_{sym}^1(m)$ (also called covariant derivative), and then apply it to the the gradient of $\ell_{U}$  $$\grad \ell_{U} : \Pos_{sym}^1(m) \to \TT\Pos_{sym}^1(m)$$ 
at a point $\Sigma \in \Pos_{sym}^1(m)$.. 

Recall also the definition of the covariant derivative is saying that
\begin{equation}
\label{equ::hess_formula}
 \langle \nabla_{Z} \grad \ell_{U}(\Sigma) , W \rangle_{\Sigma}=  \Hess  \ell_{U}(\Sigma)(Z,W)
\end{equation}
for every $Z,W \in \TT_{\Sigma} \mkern-4mu \Pos_{sym}^1(m)$. Then we will have
 $$ \nabla_{Z} \grad \ell_{\shp}(\Sigma) = \int\limits_{G(m,r)} \nabla_{Z} \grad \ell_{U}(\Sigma) d\shp(U)$$
 and
 $$ \langle \nabla_{Z} \grad \ell_{\shp}(\Sigma) , W \rangle_{\Sigma}= \int\limits_{G(m,r)} \langle \nabla_{Z} \grad \ell_{U}(\Sigma) , W \rangle_{\Sigma} d\shp(U).$$

\begin{lemma}
\label{lem::covariant_deriv}
Let $Z, Y$ be vector fields in $\TT \Pos_{sym}^1(m)$. Then for every $\Sigma \in \Pos_{sym}^1(m)$ 
$$(\nabla_{Z} Y)_\Sigma= (Z(Y))_{\Sigma} - \frac{1}{2} Z\Sigma^{-1}Y- \frac{1}{2} Y\Sigma^{-1}Z \in  \TT_{\Sigma} \mkern-4mu \Pos_{sym}^1(m).$$
In particular, for $Y = \grad \ell_{U}$ we have
$$\nabla_{Z} \grad \ell_{U}(\Sigma) = \frac{1}{4}Z\pi_{U}(\Sigma) \Sigma+  \frac{1}{4} \Sigma \pi_{U}(\Sigma) Z -\frac{1}{2}  \Sigma \pi_{U}(\Sigma) Z \pi_{U}(\Sigma)\Sigma$$
where $\pi_{U}(\Sigma):= \Sigma^{-1} X (X^{T} \Sigma^{-1}  X)^{-1} X^{T} \Sigma^{-1}$. 
\end{lemma}
\begin{proof}
The proof is provided in Section \ref{sec::Appendix}.
\end{proof}

\section{Geodesic convexity of M-functionals}
\label{s::GeodesicConvexity}
From Section \ref{subsub::Grassmanian_Busemann} we hence know that our M-functionals $\ell_\shp$ are weighted  Busemann functions in the sense of \cite{KLM}, and they are geodesically convex functions by the same \cite{KLM}. The reader is also invited to look at \cite{DT} for a full treatment of geodesics coercivity for a large class of models involving M-estimates of scatter. This section deals further with the strict geodesic convexity of the M-functionals $\ell_\shp$ and gives sufficient conditions to imply that (see Proposition \ref{prop::strict_conv_P} and Lemmas \ref{lem::strict_conv_1}, \ref{lem::strict_conv_x}), this being a property also used in the proof of Theorem \ref{thm::LLN}.  The key ingredients used in this section are the (self-)$\Sigma$-adjoint linear maps and the $\Sigma$-orthogonal projections  ${\rm Pr}(U, \Sigma)$, where $\Sigma \in \pas$. Along the way we also give a proof of the geodesic convexity of $\ell_\shp$ (see Proposition \ref{prop::strict_conv_P}).

A function $f: \pas \to \RR$ is called \textbf{geodesically convex} (resp., \textbf{strictly geodesically convex}) if its restriction
$f(\gamma(t))$ ($t\in\R$) to any geodesic $\gamma$ in $\pas$ is convex (resp., strictly convex) in the  
usual sense. This amounts to saying that the Hessian $\nabla^2f$ is positive
semidefinite (resp., positive-definite), i.e.,
$\nabla^2_vf(\sig)=\langle \nabla_v\grad f(\sig),v\rangle\ge0$ (resp., $>0$) for all
$\sig\in\pas$ and $v \neq 0_m \in \TT_{\Sigma} \mkern-4mu \Pos_{sym}^1(m)$, since
\begin{equation}\label{secondder}
\frac {d^2}{dt^2}f(\gamma(t))
=(\nabla^2_vf)(\gamma(t))=\langle \nabla_v\grad f(\gamma(t)),v\rangle,
\end{equation}
where $v$ is the velocity of the geodesic $\gamma$. 

\medskip
For this section we fix  an element $U$ of $G(m,r)$.  Recall $U$ is  the linear span of $r$ linearly independent vectors say $\{x_1, x_2, \cdots, x_r\}$. With $U = \langle x_1, x_2, \cdots, x_r  \rangle$ we associate the $m \times r$ matrix $X$ given by $(x_1, x_2, \cdots, x_r)$, with $x_i$ the columns of $X$ that are written in the canonical base of $\RR^m$.

First we give a geometric interpretation of the matrix $\pi_{U}(\Sigma):= \Sigma^{-1} X (X^{T} \Sigma^{-1}  X)^{-1} X^{T} \Sigma^{-1}$ introduced in Lemma \ref{lem::covariant_deriv}.

Following the ideas in \cite[Section 4]{AMR}, given $\Sigma \in \Pos_{sym}^1(m)$ we introduce on $\RR^m$ the following \textbf{$\Sigma$-scalar product} 
\begin{equation}
\label{equ::Sigma_scalar_product}
(x\vert y )_{\Sigma}:= x^{T} \Sigma^{-1}y,  \ \ \ \ \text{for  every } x,y \in \RR^m.
\end{equation}

As $U$ is a linear subspace of $\RR^m$, we denote by ${\rm Pr}(U, \Sigma): \RR^m \to U$ the \textbf{$\Sigma$-orthogonal projection} onto $U$ which is a linear map defined by ${\rm Pr}(U, \Sigma)(u)=u$ for every $u \in U$ and ${\rm Pr}(U, \Sigma)(v)=0$ for every $v \in \RR^m$ that is $\Sigma$-orthogonal to $U$, i.e., $v \in U^{\perp_{\Sigma}}: =\{ w \in \RR^{m} \; \vert \; (u\vert w )_{\Sigma}=0 \;  \forall u \in U\}$.

\begin{remark}
\label{rem::Sigma_scalar_prod}
For $\Sigma = \Id_m$ the scalar product $(x\vert y )_{\Id_m}$ is just the canonical Euclidean scalar product $\langle , \rangle$ on $\RR^m$.
As for $\Sigma \in  \Pos_{sym}^1(m)$ there is a unique symmetric positive-definite square roots $g=g^T \in \SL(m, \RR)$ with $\Sigma= gg^{T}$, then $(x\vert y )_{\Sigma}= x^{T} \Sigma^{-1}y= \langle g^{-1}x ,g^{-1}y \rangle $  for  every  $x,y \in \RR^n$. In particular we obtain $U^{\perp_{\Sigma}}$ is a linear subspace of dimension $m-r$. 
\end{remark}

\begin{lemma}
\label{lem::orth_proj}
Given $\Sigma$ and $U$ as above, we have ${\rm Pr}(U, \Sigma)= \Sigma \pi_{U}(\Sigma)$.
\end{lemma}

\begin{proof}
Let $w \in \RR^m$ and using the linearly independent vectors $x_1, x_2, \cdots, x_r$ spanning $U$ there exists a unique $\eta \in \RR^r$ such that ${\rm Pr}(U, \Sigma)(w)= X \eta \in U$. Then $w - {\rm Pr}(U, \Sigma)(w)= w - X \eta$ is $\Sigma$-orthogonal to $U$ and so for every $u \in U$ we have 
$$(u\vert w - X \eta)_{\Sigma}=0 \Leftrightarrow u^{T} \Sigma^{-1} w= u^{T} \Sigma^{-1} X \eta.$$
In particular, by taking $u \in \{x_1, ..., x_r\}$ we obtain $X^{T}\Sigma^{-1} w= X^{T}\Sigma^{-1} X \eta$ implying ${\rm Pr}(U, \Sigma)= X (X^{T} \Sigma^{-1}  X)^{-1} X^{T} \Sigma^{-1}$, and the conclusion follows.
\end{proof}

\begin{definition}
\label{def::Sigma_adj}
Let $\Sigma \in \Pos_{sym}^1(m)$. For a linear map $A : \RR^m \to \RR^m$ its \textbf{$\Sigma$-adjoint} with respect to the $\Sigma$-scalar product $(\cdot \vert  \cdot)_{\Sigma}$ is the linear map $\Sigma A^{T} \Sigma^{-1}$. $A$ is called \textbf{self-$\Sigma$-adjoint} if $(A(x)\vert y )_{\Sigma}= (x\vert A(y) )_{\Sigma}$, for every $x, y \in \RR^m$, or equivalently, if $A= \Sigma A^{T} \Sigma^{-1}$. 
\end{definition}

\begin{remark}
\label{rem::first_rem_adj}
Let $\Sigma \in  \Pos_{sym}^1(m)$. Then $\Sigma^{-1}, {\rm Pr}(U, \Sigma)$ and $W \Sigma^{-1}$ are self-$\Sigma$-adjoint, for every $W \in \TT_{\Sigma} \mkern-4mu \Pos_{sym}^1(m)=\{A \in \Sym(m) \; \vert \; \tr(g^{-1}A (g^{T})^{-1})=0\}$. Indeed, for example $W \Sigma^{-1}= \Sigma \Sigma^{-1} W \Sigma^{-1}= \Sigma (W \Sigma^{-1})^{T} \Sigma^{-1}$.
\end{remark}

\begin{remark}
\label{rem::second_rem_adj}
For a linear map $A : \RR^m \to \RR^m$ the $\Id_m$-adjoint is just the usual transpose of $A$. Moreover, for $\Sigma \in  \Pos_{sym}^1(m)$,  if $A$ is self-adjoint, i.e., $A=A^{T}$, then $A \Sigma^{-1}$ is self-$\Sigma$-adjoint. And vice versa, if $A$ is self-$\Sigma$-adjoint then $A \Sigma$ is self-adjoint.
\end{remark}


\begin{lemma}
\label{lem::sigma_adj_trace}
Let $\Sigma= g g^{T} \in \Pos_{sym}^1(m)$. Let $\begin{pmatrix} v_{11} & v_{12} \\
v_{21} & v_{22} \\
\end{pmatrix} : \RR^{m} \to \RR^m$ be a self-$\Sigma$-adjoint map where $v_{11} : U \to U$, $v_{22} : U^{\perp_{\Sigma}} \to U^{\perp_{\Sigma}}$, $v_{12} : U^{\perp_{\Sigma}} \to U$, and $v_{21} :U \to U^{\perp_{\Sigma}}$.

Then $\begin{pmatrix} 0 & v_{12} \\
v_{21} & 0 \\
\end{pmatrix}$, 
$\begin{pmatrix} v_{11} & 0\\
0 & 0 \\
\end{pmatrix}$ and $\begin{pmatrix} 0 & 0 \\
0 & v_{22} \\
\end{pmatrix} : \RR^{m} \to \RR^m$
are self-$\Sigma$-adjoint. In particular, we have $\tr(v_{21} v_{12}) \geq 0$, resp., $\tr(v_{11} v_{11}) \geq 0$, with equality if and only if both $v_{12}, v_{21}$ vanish, resp., $v_{11}$ vanishes.
\end{lemma}

\begin{proof}
Take  $C \in \{\begin{pmatrix} 0 & v_{12} \\
v_{21} & 0 \\
\end{pmatrix}, \begin{pmatrix} v_{11} & 0\\
0 & 0 \\
\end{pmatrix},  \begin{pmatrix} 0 & 0 \\
0 & v_{22} \\
\end{pmatrix}\}$. To prove $C$ is the $\Sigma$-adjoint of $C$, it is enough to verify $(C(x)\vert y )_{\Sigma}= (x\vert C(y) )_{\Sigma}$, for every $x, y \in \RR^m$. Indeed, as $A:=\begin{pmatrix} v_{11} & v_{12} \\
v_{21} & v_{22} \\
\end{pmatrix} $ is self-$\Sigma$-adjoint, for every $x=(x_1,x_2)^{T}, y=(y_1,y_2)^{T} \in \RR^m = U \bigoplus U^{\perp_{\Sigma}}$ we have
\begin{equation}
\label{equ::sigma_adj}
\begin{split}
(A(x)\vert y)_{\Sigma} &=((v_{11}x_1+v_{12}x_2,v_{21}x_1+v_{22}x_2)^{T}\vert(y_1,y_2)^{T})_{\Sigma}\\
&=((v_{11}x_1+v_{12}x_2,0)^{T}\vert(y_1,0)^{T})_{\Sigma}+ ((0,v_{21}x_1+v_{22}x_2)^{T}\vert(0,y_2)^{T})_{\Sigma}\\
&=(x\vert A(y))_{\Sigma}\\
&=((x_1,x_2)^{T}\vert (v_{11}y_1+v_{12}y_2,v_{21}y_1+v_{22}y_2)^{T})_{\Sigma}\\
&=((x_1,0)^{T}\vert(v_{11}y_1+v_{12}y_2,0)^{T})_{\Sigma}+ ((0,x_2)^{T}\vert(0,v_{21}y_1+v_{22}y_2)^{T})_{\Sigma}.\\
\end{split}
\end{equation}

If we take  in equality (\ref{equ::sigma_adj}) $x_1=y_1=0$, respectively, $x_2=y_2=0$, one obtains, respectively, 
\begin{equation}
\label{equ::sigma_adj_5}
((0,v_{22}x_2)^{T}\vert(0,y_2)^{T})_{\Sigma}= ((0,x_2)^{T}\vert(0,v_{22}y_2)^{T})_{\Sigma}
\end{equation}
\begin{equation}
\label{equ::sigma_adj_6}
((v_{11}x_1,0)^{T}\vert(y_1,0)^{T})_{\Sigma}= ((x_1,0)^{T}\vert(v_{11}y_1,0)^{T})_{\Sigma}.
\end{equation}
Equalities (\ref{equ::sigma_adj_5}), (\ref{equ::sigma_adj_6}) prove $\begin{pmatrix} v_{11} & 0\\
0 & 0 \\
\end{pmatrix},  \begin{pmatrix} 0 & 0 \\
0 & v_{22} \\
\end{pmatrix}$ are self-$\Sigma$-adjoint.

Taking $x_1=0$, respectively, $y_1=0$ in equality (\ref{equ::sigma_adj})  one obtains
\begin{equation}
\label{equ::sigma_adj_1}
\begin{split}
&((v_{12}x_2,0)^{T}\vert(y_1,0)^{T})_{\Sigma}+ ((0,v_{22}x_2)^{T}\vert(0,y_2)^{T})_{\Sigma}= ((0,x_2)^{T}\vert(0,v_{21}y_1+v_{22}y_2)^{T})_{\Sigma}=\\
&= ((0,x_2)^{T}\vert(0,v_{21}y_1)^{T})_{\Sigma} + ((0,x_2)^{T}\vert(0,v_{22}y_2)^{T})_{\Sigma}\\
\end{split}
\end{equation}
respectively
\begin{equation}
\label{equ::sigma_adj_2}
\begin{split}
((0,v_{21}x_1+v_{22}x_2)^{T}\vert(0,y_2)^{T})_{\Sigma}&= ((0,v_{21}x_1)^{T}\vert(0,y_2)^{T})_{\Sigma} + ((0,v_{22}x_2)^{T}\vert(0,y_2)^{T})_{\Sigma}\\
&=((x_1,0)^{T}\vert((v_{12}y_2,0)^{T})_{\Sigma}+ ((0,x_2)^{T}\vert(0,v_{22}y_2)^{T})_{\Sigma}.\\
\end{split}
\end{equation}

Because we have proved $\begin{pmatrix} 0 & 0 \\
0 & v_{22} \\
\end{pmatrix}$ is self-$\Sigma$-adjoint, from equalities (\ref{equ::sigma_adj_1}), (\ref{equ::sigma_adj_2}) we have
\begin{equation}
\label{equ::sigma_adj_3}
((v_{12}x_2,0)^{T}\vert(y_1,0)^{T})_{\Sigma}=((0,x_2)^{T}\vert(0,v_{21}y_1)^{T})_{\Sigma}
\end{equation}
\begin{equation}
\label{equ::sigma_adj_4}
((0,v_{21}x_1)^{T}\vert(0,y_2)^{T})_{\Sigma}=((x_1,0)^{T}\vert((v_{12}y_2,0)^{T})_{\Sigma}.
\end{equation}
Adding up equalities (\ref{equ::sigma_adj_3}), (\ref{equ::sigma_adj_4})  we obtain exactly the desired equality $(C(x)\vert y )_{\Sigma}= (x\vert C(y) )_{\Sigma}$, for every $x=(x_1,x_2)^{T}, y=(y_1,y_2)^{T} \in \RR^m = U \bigoplus U^{\perp_{\Sigma}}$.

Thus $C= \Sigma C^{T} \Sigma^{-1}= \Sigma \begin{pmatrix} 0 & v_{21}^{T} \\
v_{12}^{T} & 0 \\
\end{pmatrix} \Sigma^{-1}$.
Now $\tr(C C)= 2 \tr(v_{21}v_{12})= \tr(C \Sigma C^{T} \Sigma^{-1})= \tr(g^{-1}C g g^{T} C^{T} (g^{T})^{-1})= \tr(g^{-1}C g(g^{-1}C g)^{T})\geq 0$. One sees we have equality if and only if $g^{-1}C g=0 \Leftrightarrow C=0$.

As well, if you take $C=\begin{pmatrix} v_{11} & 0\\
0 & 0 \\
\end{pmatrix}$ in the above calculation, the conclusion of the lemma follows.
\end{proof}

\begin{lemma}
\label{lem::strict_conv}
Let  $U$ be an element of $G(m,r)$, $\Sigma \in \Pos_{sym}^1(m)$ and $Z \in \TT_{\Sigma} \mkern-4mu \Pos_{sym}^1(m) \setminus \{0_m\}$. Then for the function $\ell_U$ given in (\ref{log-density}) we have $$ \langle \nabla_{Z} \grad \ell_{U}(\Sigma) , Z \rangle_{\Sigma} \geq 0.$$ The equality holds if and only if 
$$Z \Sigma^{-1}=\begin{pmatrix} v_{11} & 0 \\
0 & v_{22} \\
\end{pmatrix} : U \bigoplus U^{\perp_{\Sigma}} \to U \bigoplus U^{\perp_{\Sigma}}.$$
\end{lemma}
\begin{proof}
By Remark \ref{rem::first_rem_adj} $Z \Sigma^{-1}$ is self-$\Sigma$-adjoint and moreover we can write $$Z \Sigma^{-1}=\begin{pmatrix} v_{11} & v_{12} \\
v_{21} & v_{22} \\
\end{pmatrix} : \RR^{m} \to \RR^m$$ where $v_{11} : U \to U$, $v_{22} : U^{\perp_{\Sigma}} \to U^{\perp_{\Sigma}}$, $v_{12} : U^{\perp_{\Sigma}} \to U$ and $v_{21} :U \to U^{\perp_{\Sigma}}$. Then by Lemma \ref{lem::sigma_adj_trace} $\tr(v_{21} v_{12}) \geq 0$, resp., $\tr(v_{11} v_{11}) \geq 0$, with equality if and only if both $v_{12}, v_{21}$ vanish, resp., $v_{11}$ vanishes.

As ${\rm Pr}(U, \Sigma)= \Sigma \pi_{U}(\Sigma)$ is the $\Sigma$-orthogonal projection onto $U$ we can write $${\rm Pr}(U, \Sigma) =\begin{pmatrix} 1 & 0 \\
0 & 0 \\
\end{pmatrix} : U \bigoplus U^{\perp_{\Sigma}} \to U \bigoplus U^{\perp_{\Sigma}}.$$

By Lemma \ref{lem::covariant_deriv} and an easy computation
\begin{equation}
\label{equ::grad_trace_proj}
\begin{split}
\langle \nabla_{Z} \grad \ell_{U}(\Sigma) , Z \rangle_{\Sigma}&=\frac{1}{2}\tr((\Sigma^{-1}- \pi_{U}(\Sigma)) Z\pi_{U}(\Sigma)Z)\\
&= \frac{1}{2} \tr((\Id_m- {\rm Pr}(U, \Sigma)) Z \Sigma^{-1} {\rm Pr}(U, \Sigma) Z \Sigma^{-1})\\
&=  \frac{1}{2}\tr(\begin{pmatrix}  0 & 0 \\
0 & 1 \\
\end{pmatrix} \begin{pmatrix} v_{11} & v_{12} \\
v_{21} & v_{22} \\
\end{pmatrix} 
\begin{pmatrix} 1 & 0 \\
0 & 0 \\
\end{pmatrix}
\begin{pmatrix} v_{11} & v_{12} \\
v_{21} & v_{22} \\
\end{pmatrix})\\
&= \frac{1}{2}\tr(v_{21}v_{12}) \geq 0.\\
\end{split}
\end{equation}
\end{proof}



\begin{proposition}
\label{prop::strict_conv_P}
Let $\shp$ be a probability measure on $G(m,r)$, $\Sigma \in \Pos_{sym}^1(m)$ and $Z \in \TT_{\Sigma} \mkern-4mu \Pos_{sym}^1(m) \setminus \{0_m\}$. Then the log-likelihood $\ell_{\shp}$ of $\shp$ given by (\ref{log-likelihood}) verifies 
$$\langle \nabla_{Z} \grad \ell_{\shp}(\Sigma) , Z \rangle_{\Sigma} \geq 0.$$
The equality holds if and only if 
$$Z \Sigma^{-1}=\begin{pmatrix} v_{11} & 0 \\
0 & v_{22} \\
\end{pmatrix} : U \bigoplus U^{\perp_{\Sigma}} \to U \bigoplus U^{\perp_{\Sigma}}$$
for $\shp$-almost all $U \in G(m,r)$. In particular, we have $Z \Sigma^{-1}(U^{\perp_{\Sigma}}) \subseteq U^{\perp_{\Sigma}}$ and $Z \Sigma^{-1}(U) \subseteq U$, for $\shp$-almost all $U \in G(m,r)$.
\end{proposition}
\begin{proof}
 By using the equation 
$$\langle \nabla_{Z} \grad \ell_{\shp}(\Sigma) , Z \rangle_{\Sigma}= \int\limits_{G(m,r)} \langle \nabla_{Z} \grad \ell_{U}(\Sigma) , Z \rangle_{\Sigma} d\shp(U)$$ it is a consequence of Lemma \ref{lem::strict_conv}.
\end{proof}
 As by Proposition \ref{prop::strict_conv_P} the log-likelihood function $\ell_\shp$ is convex, its minima are exactly the zeroes of its gradient:

\begin{corollary}[M-equation]\label{M-equation}
A parameter $\sig\in\pas$ is a \gmle\ of a probability measure~$\shp$ on $\sas$ if and only if it satisfies the
\textbf{M-equation}
\begin{equation}
\int_\sas {\rm Pr}(U,\sig)\,d\shp(U)=\frac{r}{m} \Id_m,
\end{equation}
where ${\rm Pr}(U,\sig)=X (X^{T} \Sigma^{-1} X)^{-1} X^{T}\sig^{-1}$.
\end{corollary}

Let $U_1, \cdots, U_n \subset G(m,r)$ and let $\shp_n := \frac{1}{n} (\delta_{U_1}+...+ \delta_{U_n})$ be the empirical  probability measure on $G(m,r)$ corresponding to the samples $\{U_1, ..., U_n\}$. In general, given $n>0$ large enough  one would like to have a sufficient condition on $U_1, \cdots, U_n  \subset G(m,r)$ so that  $\ell_{\shp_n}$ to be a strictly geodesically convex function. This is not the case due to the following example.
\begin{example}
\label{ex::not_suff}
Take $m=3$, $r=2$ and $\Sigma = \Id_3 \in \Pos_{sym}^1(3)$. Consider $$Z = \begin{pmatrix} 1/2 & 0 & 0 \\
0 & 1/2 & 0 \\
0 & 0 & -1 \\
\end{pmatrix}  \in \TT_{\Id_3} \mkern-4mu \Pos_{sym}^1(m) = \{A \in \Sym(m) \; \vert \; \tr(A)=0\}. $$ 
As $Z \Sigma^{-1}= Z$ is self-$\Sigma$-adjoint, thus just self-adjoint, and its eigenvalues are $1/2$ and $-1$, the corresponding eigenspaces $E_{1/2}$ and $E_{-1}$ are pairwise $\Sigma$-orthogonal and $\RR^3= E_{1/2}\oplus E_{-1}$.

For every $v \in E_{1/2}$ define $U(v):= \RR v \oplus E_{-1} \in G(3,2)$. Then notice $Z(U(v)) \subseteq U(v)$ and $Z(U(v)^{\perp_{\Sigma}}) \subseteq U(v)^{\perp_{\Sigma}}$ for every $v \in E_{1/2}$. Therefore by Proposition \ref{prop::strict_conv_P}, for every $n>0$ and for every sequence of vectors $\{v_j\}_{1 \leq j\leq n} \subset E_{1/2}$ the corresponding $\ell_{\shp_n} $ associated with $\{U(v_j)\}_{1 \leq j\leq n} \subset G(3,2)$ is not strictly geodesically convex at $\Sigma= \Id_3$ and $Z$ as above. Still, by Example \ref{ex::cond_1} below $\shp_n$ has a unique GE (almost surely) for $n$ large enough.
\end{example}

\begin{lemma}
\label{lem::strict_conv_1}
Let $\shp$ be a continuous Borel probability measure on $G(m,r)$ with support $G(m,r)$. Then $\ell_{\shp}$ is a strictly geodesically convex function. 
\end{lemma}

\begin{proof}
Suppose by contradiction that $\ell_{\shp}$ is not a strictly geodesically convex function. By Proposition \ref{prop::strict_conv_P} we know $\ell_{\shp}$ is geodesically convex and there exist $\Sigma \in \Pos_{sym}^1(m)$ and $Z \in \TT_{\Sigma} \mkern-4mu \Pos_{sym}^1(m)  \setminus \{0_m\}$ such that $\langle \nabla_{Z} \grad \ell_{\shp}(\Sigma) , Z \rangle_{\Sigma} = 0$. In particular, by the same Proposition \ref{prop::strict_conv_P} we must have $Z \Sigma^{-1}=\begin{pmatrix} v_{11} & 0 \\
0 & v_{22} \\
\end{pmatrix} : U \bigoplus U^{\perp_{\Sigma}} \to U \bigoplus U^{\perp_{\Sigma}}$ and so $Z \Sigma^{-1}(U) \subseteq U $, $Z \Sigma^{-1}(U^{\perp_{\Sigma}}) \subseteq U^{\perp_{\Sigma}}$ for almost every $U$ in the support of $\shp$. 

As $Z \Sigma^{-1}$ is self-$\Sigma$-adjoint, its eigenvalues are real numbers $\lambda_1>\lambda_2>\ldots>\lambda_{s}$  
($s \geq 1$), the corresponding eigenspaces $E_1,\ldots,E_{s}$ are  
pairwise $\sig$-orthogonal, and $\R^m=E_1\oplus\cdots\oplus E_{s}$. As $\tr(Z \Sigma^{-1})= 0$ and $Z \neq 0_m$ we must have $s\geq 2 $.  

We claim there are (many) linearly independent vectors $v_1, v_2,...,v_r \in \R^m$ such that for every $i \in \{1,...,r-1\}$ the vector $v_i$ in some $E_j$, and $v_r$ is such that $\Pr(E_i,\Sigma)(v_r) \neq 0$ for every $i  \in \{1,...,s\}$. Then $ Z \Sigma^{-1}(v_r) \notin \RR v_r$. Then $U :=span(v_1,...,v_r) \in G(m,r)$ has the property that  $Z \Sigma^{-1}(U)$ is not contained in $U$.

Now because $\shp$ is a continuous Borel probability measure on $G(m,r)$ with support $G(m,r)$ and $Z \Sigma^{-1}$ induces a continuous map on $G(m,r)$, there is a small open neighborhood $\mathcal{U} \subset G(m,r)$ of $U$,  thus  $\mathcal{U}$  of positive $\shp$-measure, such that for every $V \in \mathcal{U}$ we have $Z \Sigma^{-1}(V)$ is not contained in $V$ as well. This is a contradiction and the lemma follows.
\end{proof}

Examples of continuous Borel probability measures on $G(m,r)$ such that Lemma \ref{lem::strict_conv_1} holds true are  the \gr\ distributions $\G_\sig$, for $\Sigma \in \Pos_{sym}^1(m)$. As the probability measure $\G_\sig$ is $K$-invariant, where $K= g \SO(m, \RR)g^{-1}$ is the maximal compact subgroup of $\SL(m, \RR)$ stabilizing the point $\Sigma = gg^{T} \in \Pos_{sym}^1(m)$, one can see $\ell_{\G_\sig}$ has $\Sigma$ as a unique GE.

\begin{lemma}
\label{lem::strict_conv_x}
Let $\shp$ be a continuous Borel probability measure on $G(m,r)$ with support $G(m,r)$. Let $U_i$, $i=1,\cdots,n$, be i.i.d. random elements of $G(m,r)$ of law $\shp$, and let 
$\shp_n:= \frac{1}{n} (\delta_{U_1}+...+ \delta_{U_n})$ be the related (random) empirical measure.
Then, if $nr>m$,  $\ell_{\shp_n}$ is a.s. strictly geodesically  convex. 
\end{lemma}

\begin{proof}
Consider the sample empirical log-likelihood map
$$\ell_{\shp_n}=\frac{1}{n}\sum_{i=1}^n \ell_{U_i}.$$
In what follows we suppose without loss of generality that $\Sigma =\Id_m$. 
According to Proposition \ref{prop::strict_conv_P}, the fact that the random map 
$\ell_{\shp_n}$ is not strictly convex at $\Sigma =\Id_m$ implies the existence
 of some $Z \in \TT_{\Id_m} \mkern-4mu \Pos_{sym}^1(m) \setminus \{0_m\}$ such that
$Z U_i \subseteq U_i$, $i=1,\cdots,n$.  Moreover, as $Z$ is a tangent vector at $\Id_m$, for every scalar $\lambda \neq 0 \in \RR$ we also have $\lambda Z U_i \subseteq \lambda U_i= U_i$. Therefore it is enough to consider $Z$ in the unit tangent space $\TT_{\Id_m} \mkern-4mu \Pos_{sym}^1(m)$ that is a compact space.  

As $Z$ is self-$\Id_m$-adjoint, thus just self-adjoint, its eigenvalues are real numbers $\lambda_1>\lambda_2>\ldots>\lambda_{s}$  
($s \geq 1$), the corresponding eigenspaces $E_1,\ldots,E_{s}$ are  
pairwise $\Id_m$-orthogonal, thus just orthogonal, and $\R^m=E_1\oplus\cdots\oplus E_{s}$. As $\tr(Z )= 0$ and $Z \neq 0_m$ we must have $s\geq 2 $.  
Using this decomposition, we see that 
$\{Z U_i \subseteq U_i\}$ and $Z$ being self-adjoint implies $U_i=U_i^{1}\oplus\cdots\oplus U_i^{s}$ with $U_i^{k} \subseteq E_k$ or $U_{i}^{k}=\{0_m\}$, for $i=1,\cdots, n$ and $k =1,\cdots,s$. Suppose furthermore that
${\rm dim}(E_1)\ge {\rm dim}(E_2)\ge \cdots \ge {\rm dim}(E_s)>0$.
As $U_i$ has a continuous Borel law on $G(m,r)$, it is sufficient that the probability that $\shp_n$ is not strictly convex vanishes when the $U_i$ are uniform on $G(m,r)$, of law 
$\G_{\Id_m}$. In this particular case each $U_i=\langle x_{i1},\cdots,x_{ir}\rangle$ is generated by i.i.d. normal multivariate random vectors $x_{i1},\cdots, x_{ir}\in\R^m$  of covariance matrix
$\sig =\Id_m$. 

In summary, the event of interest is included in the following
\begin{eqnarray*}
A&=&\{\exists \R^m=E_1\oplus\cdots\oplus E_{s},\ s\ge 2, \\
  & & U_i=U_i^{1}\oplus\cdots\oplus U_i^{s},\ U_i^{k} \subseteq E_k\ i=1,\cdots, n,\ k=1,\cdots,s\}.
  \end{eqnarray*}
 
Let $W\in {\rm O}(m)$ be an orthogonal matrix transforming the orthogonal subspaces $E_k$ into canonical ones. We can then rewrite $A$ as follows, with the various $E_k$ being canonical orthogonal subspaces
\begin{eqnarray*}
\label{eq::A_ortho}
A&=&\{\exists \R^m=E_1\oplus\cdots\oplus E_{s},\ s\ge 2, \hbox{ and }W\in {\rm O}(m),\\
  & & V_i =W^{-1}U_i=V_i^{1}\oplus\cdots\oplus V_i^{s},\ V_i^{k} \subseteq E_k,\ k=1,\cdots,s,\\
  & & V_i =\langle W^{-1}x_{i1},\cdots, W^{-1} x_{ir}\rangle,\ i=1,\cdots, n,\}.
    \end{eqnarray*}  
The random vectors $y_{ij}=W^{-1}x_{ij}$, $i=1,\cdots,n$, $j=1,\cdots,r$ are again i.i.d. standard normal in $\R^m$. Let $F_k=E_1\oplus\cdots\oplus E_{k}$, $k=1,\cdots,s$. We orthogonally decompose each $y_{ij}$ according to this flag by setting $y_{ij}=(y_{ij}^1,\cdots,y_{ij}^s)$ with $y_{ij}^k \in V_i^k$, $k=1,\cdots,s$. Then the orthogonal projection of $y_{ij}$ onto $F_{s-1}$ is just $(y_{ij}^1,\cdots,y_{ij}^{s-1})$. 
\begin{itemize}
\item{} Suppose that there is some $1\le i\le n$ such that $V_i^s \ne 0$. The $r$ orthogonal projections
$(y_{ij}^1,\cdots,y_{ij}^{s-1})$ are contained in the linear subspace 
$V_i^{1}\oplus\cdots\oplus V_i^{s-1}\subset F_{s-1}$ of dimension ${\rm dim}(V_i^{1}\oplus\cdots\oplus V_i^{s-1})=r-{\rm dim}(V_i^s)< r$, and are thus linearly dependent.
 The probability that the $r$ multivariate standard normal random projections 
$(y_{ij}^1,\cdots,y_{ij}^{s-1})$ are linearly dependent is zero. 
\item{} If all the linear subspaces $V_i$ are such that $V_i^s =0$, the necessarily all of the $n r$ random vectors $y_{ij}$ belong to the linear subspaces $F_{s-1}$ with ${\rm dim}(F_{s-1})< m$, and the probability of such an event is zero when $nr > m$.
\end{itemize}
 \end{proof}

Recall, strictly geodesically convex functions do not necessarily admit a global/ local minima, but if they admit one this must be unique.

\begin{lemma}
\label{lem::strictly_convex_fct}
Let $f : \Pos_{sym}^1(m) \to \RR$ be a differential function whose Hessian $\Hess(f)$ at a point $\Gamma \in \Pos_{sym}^1(m)$ is invertible. Then $\Gamma $ is a critical point for $\langle \grad(f)(\Sigma), \grad(f)(\Sigma) \rangle_{\Sigma}$, where $\Sigma \in \Pos_{sym}^1(m)$, if and only if $\grad(f)(\Gamma)=0$.
\end{lemma}

Recall, strictly geodesically convex functions $f$ have their Hessian $\Hess(f)$ invertible at every point. Still, the converse is not true in general.

\begin{proof}
Let $g : \Pos_{sym}^1(m) \to \RR$ be given by $g(\Sigma): = \langle \grad(f)(\Sigma), \grad(f)(\Sigma) \rangle_{\Sigma}$. By definition $\Gamma$ is a critical point of $g$ if and only if $\grad(g)(\Gamma)=0$. 

Suppose $\Gamma$ is a critical point for $g$, and we want to prove $\grad(f)(\Gamma)=0$. As $\grad(g)(\Gamma)=0$,  for every $X \in \TT_{\Gamma}\Pos_{sym}^1(m)$ we have $$0=\nabla_{X}(g)(\Gamma)= 2 \langle \nabla_{X} \grad(f)(\Gamma), \grad(f)(\Gamma) \rangle_{\Gamma}= 2\Hess(f)_{\Gamma}(X, \grad(f)(\Gamma))$$
which implies that $\grad(f)(\Gamma) \in \Ker (\Hess(f)_{\Gamma})$. As the Hessian $\Hess(f)_{\Gamma}$ at $\Gamma$ is invertible then $\grad(f)(\Gamma) \in \Ker (\Hess(f)_{\Gamma})=0$.

Now suppose $\grad(f)(\Gamma)=0$; we want to prove $\Gamma$ is a critical point for $g$. Indeed, as for every $X \in \TT_{\Gamma}\Pos_{sym}^1(m)$ we have $$ \langle \grad(g)(\Gamma), X\rangle_{\Gamma}= \nabla_{X}(g)(\Gamma)= 2 \langle \nabla_{X} \grad(f)(\Gamma), \grad(f)(\Gamma) \rangle_{\Gamma}=0$$
one obtains $\grad(g)(\Gamma)=0$ and the conclusion follows.

\end{proof}
\section{Behaviour at infinity of M-functionals}
\label{s.behaviourinfinity}

The following Theorem is a refinement of results of \cite{KLM} giving (\ref{Unicity1}) as necessary and sufficient condition for the existence of \gmle. It covers situations where (\ref{Unicity1}) is not satisfied, and its proof is based on our previous constructions.
\begin{theorem}\label{mainadd}
Let $\shp$ be a  Borel probability measure on the \gr\ $\sas$. Denote by $\pls=\bigcup_{s=1}^{m-1}\Gr(m,s)$ the set of all the proper vector subspaces $V$ of $\ls$ ($0\ne V\ne\ls$) and set
\begin{equation}
I_\shp(V)=\frac rm\dim(V)-\int_\sas\dim(U\cap V)\,d\shp(U)
\end{equation} 
for $V\in\pls$.
\begin{enumerate}
\item If $I_\shp(V)>0$ for all $V\in\pls$, then $\shp$ has a unique \gmle.
\item If $I_\shp(V)<0$ for some $V\in\pls$, then $\shp$ has no \gmle.
\item (Limit case) Suppose that $I_\shp(V)\ge0$ for all $V\in\pls$ and $I_\shp(V)=0$ for some $V\in\pls$.
\begin{enumerate}
\item If each $V\in\pls$ with $I_\shp(V)=0$ has a complement $V'\in\pls$ ($V\oplus V'=\ls$) such that $I_\shp(V')=0$ too (or, equivalently, such that $U=(U\cap V)\oplus(U\cap V')$ for $\shp$-almost all $U\in\sas$), then $\shp$ has infinitely many \gmle s. More precisely, the \gmle s of $\shp$ form a submanifold of dimension $d-1$ of $\pas$ if $\ls=V_1\oplus\cdots\oplus V_d$ is a maximal decomposition of $\ls$ into subspaces $V_i\in\pls$ such that $I_\shp(V_i)=0$ for $i=1,\ldots,d$.
\item Otherwise, $\shp$ has no \gmle.
\end{enumerate}
 \end{enumerate}

\end{theorem}

\begin{example}
\label{ex::cond_1}
Condition 1. of Theorem \ref{mainadd} corresponds  to inequality (\ref{Unicity1}) from the Introduction.
For a sample $\{U_i\}_{1 \leq i \leq n }$ of size $n$, let $\shp_n:= \frac{1}{n} (\delta_{U_1}+...+ \delta_{U_n})$ be its corresponding empirical probability measure. Using methods from algebraic geometry, the authors of \cite[Theorem 4]{AMR} proved that for almost all samples of size $n > \frac{m^2}{r(m-r)}$ 
inequality (\ref{Unicity1}) is satisfied.
\end{example}

\begin{example} The limit case 3. of Theorem \ref{mainadd} occurs, for example, when $n=m $ and $\shp := \frac{1}{n} (\delta_{U_1}+...+ \delta_{U_n})$ is the empirical Borel probability measure on $G(m,1)$ corresponding to samples $\{U_1, ..., U_n\}$ that are in general position in the projective space $\Gr(m,1)=\ps$. Then, $I_\shp(V)=0$ if and only if the linear subspace $V\in\pls$ contains exactly $\dim(V)$ points of the sample, so that condition (a) is satisfied.

Moreover, for a sample $\{U_1, ..., U_n\}$ in general position in the projective space $\Gr(m,1)=\ps$ and  $\shp := \frac{1}{n} (\delta_{U_1}+...+ \delta_{U_n})$, Theorem~\ref{mainadd} gives:
{\it \begin{enumerate}
\item A sample of size $n>m$ in general position in $\ps$ has a unique GE. 
\item A sample of size $n<m$ in general position $\ps$ has no GE. 
\item When $n=m$, and the sample $\{U_1, ..., U_n\}$ of size  $n=m$ is in general position in $\ps$, then $\ls=U_1\oplus\cdots\oplus U_m$ and so its corresponding GEs form a submanifold of dimension $n-1$ in the parameter space $\pas$.
\end{enumerate}}
\end{example}

To prove Theorem \ref{mainadd}, besides convexity, we also need the following asymptotic property of the log-likelihood.
As we have seen in Section \ref{sec::sym_space} (Lemma \ref{lem::action}), any geodesic $\gamma(t)$ issuing at $\Sigma=g g^T\in\Pos_{sym}^1(m)$ takes the generic form
$\gamma(t)=g e^{tv}g^T$, with $v\in \TT_{\Id_m} \Pos_{sym}^1(m)$. 

For $ \Sigma=g g^T\in \Pos_{sym}^1(m) $  let $S_\Sigma:=  g \TT_{\Id_m} \Pos_{sym}^1(m) g^{-1}$. Then $S_\Sigma$  is a linear vector space and by Defintion \ref{def::Sigma_adj}, for every $w= g v g^{-1} \in S_\Sigma$, where $v \in \TT_{\Id_m} \Pos_{sym}^1(m)$, we have $\tr(w)=0$ and
$w$ is self-$\Sigma$-adjoint,  that is
$\Sigma w^T \Sigma^{-1}=w$. One has furthermore that
$$e^{tw}\Sigma =\left(\sum_{k=0}^\infty \frac{t^k}{k!}g v^k g^{-1} \right)g g^T =g e^{tv}g^T =\gamma(t), \; \; \; \forall t \in \RR.$$ One can thus consider geodesics of the generic form $\gamma(t)=e^{tw}\Sigma$, where $w$ belongs to the linear space $S_\Sigma$ of self-$\Sigma$-adjoint matrices of zero trace. 
As the linear operator $w$ is \sa, its  
eigenvalues are real numbers $\lambda_1>\lambda_2>\ldots>\lambda_{s+1}$  
($s \geq 0$), the corresponding eigenspaces $E_1,\ldots,E_{s+1}$ are  
pairwise $\sig$-orthogonal, and $\R^m=E_1\oplus\cdots\oplus E_{s+1}$.  

\begin{proposition}\label{thasymptotic}
 Let $\shp$ be a  Borel probability measure on the \gr\ $\sas$.  Let $w \in S_\Sigma$ and $\gamma(t)=e^{wt}\Sigma$ be the geodesic of velocity $w$ issuing from
$\sig\in\pas$.  Then $w$ decomposes in a unique
way as
\begin{equation}\label{decompose}
w=\sum_{k=1}^s\alpha_k\bigl({\rm Pr}(V_k,\sig)-\frac 1 m\dim(V_k) 
\Id_m\bigr),
\end{equation}
 where 
$V_1,\ldots,V_s$ are vector subspaces of $\ls$ such that
$V_k=E_1\oplus\cdots\oplus E_k$, with
\[
0\subset V_1\subset V_2\subset\cdots\subset V_s\subset V_{s+1}=\ls 
\]
and where $\alpha_k = \lambda_k -\lambda_{k+1}>0$.
Moreover, 
\begin{equation}\label{asymptotic}
\lim_{t\to\infty}\frac{d}{dt}\ell_\shp(\gamma(t))
=\sum_{k=1}^s\alpha_k I_\shp(V_k).
\end{equation}
\end{proposition}

\begin{proof}

According to the spectral theorem,
\[
w=\lambda_1 {\rm Pr}(E_1,\sig)+\cdots+\lambda_{s+1} {\rm Pr}(E_{s+1},\sig),
\]
and
\[
e^{-tw}=e^{-\lambda_1t}{\rm Pr}(E_1,\sig)+\cdots+e^{-\lambda_{s+1}t}{\rm Pr}(E_{s+ 
1},\sig).
\]

${\rm Pr}(E_1,\sig)={\rm Pr}(V_1,\sig)$,
${\rm Pr}(E_k,\sig)={\rm Pr}(V_k,\sig)-{\rm Pr}(V_{k-1},\sig)$ for  
$k=2,\ldots,s-1$ and ${\rm Pr}(E_{s+1},\sig)=\Id_m-{\rm Pr}(V_s,\sig)$, hence
\[
w=(\lambda_1-\lambda_2){\rm Pr}(V_1,\sig)+\cdots
+(\lambda_s-\lambda_{s+1}){\rm Pr}(V_s,\sig)+\lambda_{s+1}\Id_m.
\]
This is formula~\eqref{decompose} with $\alpha_k=\lambda_k-\lambda_{k+1}$ since $\tr(w)=0$ and $\tr {\rm Pr}(V_k,\sig)=\dim(V_k)$. 

Let $X=(x_1,\ldots,x_r)$ be a basis of $U\in\sas$. In view of the equation~\eqref{log-density},
\[
\frac d{dt}\ell_U(\gamma(t))=\frac12\,\frac d{dt}\log\det(a(t)),
\]
where
\[
a(t)=X^T\gamma(t)\inv X=X^T\sig\inv e^{-tw}X
=\sum_{k=1}^{s+1}e^{-\lambda_kt}X^T\sig\inv {\rm Pr}(E_k,\sig) X.
\]
The entries of the $r \times r$ matrix $a(t)$ are
\[
a_{ij}(T)=\sum_{k=1}^{s+1}e^{-\lambda_kt}(\p x_i|\p x_j)_\sig,
\] where $x_i, x_j$ are the vectors of $X=(x_1,\ldots,x_r)$.
Next,
\[
\det(a(t))=\sum_{\sigma_\in S_r}\sign(\sigma)a_{1\sigma(1)}(t)\cdots a_{r\sigma(r)}(t)
\]
where $S_r$ denotes the permutation group of $\{1,\ldots,r\}$, so that $\det(a(t))$ is a sum of exponential functions of $t$. It remains to find the dominant term of this sum, say $e^{-\beta t}$ up to a nonzero factor. Then
\[
\lim_{t\to\infty}\frac d{dt}\ell_U(\gamma(t))=\frac12\lim_{t\to\infty}\frac d{dt}\log\det(a(t))=-\beta/2.
\]
To obtain that, we construct a  new basis $X=(y_1,\ldots,y_r)$ of $U$ that is adapted to the vector subspaces
\[
0\subseteq U\cap V_1\subseteq U\cap V_2\subseteq\cdots
\subseteq U\cap V_s\subseteq U=U\cap V_{s+1}.
\]
Indeed, start with a $\sig$-orthonormal basis  $y_i$, $i\in I_1$, of $U\cap V_1=U\cap E_1$. Then choose
$|I_2|=\dim(U\cap V_2)-\dim(U\cap V_1)$
vectors $y_i\in U\cap V_2$, $i\in I_2$, whose
projections ${\rm Pr}(E_2,\sig)y_i$ form a $\sig$-orthonormal 
system in $E_2$. Iterating this procedure, we get at stage
$k$ a family of $|I_k|=\dim(U\cap V_k)-\dim(U\cap V_{k-1})$
vectors $y_i\in U\cap V_k$, $i\in I_k$, with $\sig$-orthonormal projections
$\p y_i$.

The family $I_1,\ldots,I_{s+1}$ forms a partition of $\{1,\ldots,r\}$ into possibly empty subsets.
Given $i\in\{1,\ldots,r\}$, we denote by $m_i$ the unique index such that $i\in I_{m_i}$. Then $\p y_i=0$ if $k>m_i$ since $y_i\in V_{m_i}$. Therefore, the dominant term of $a_{ij}(t)$ is at most $\exp(-t\lambda_{\min\{m_i,m_j\}})$, and the dominant term of the product
\[
a_{1\sigma(1)}(t)\cdots a_{r\sigma(r)}(t)
\qquad(\sigma\in S_r)
\]
in the expansion of $\det(a(t))$ is at most
\[
\prod_{i=1}^r\exp(-\lambda_{\min\{m_i,m_{\sigma(i)}\}}).
\]
The largest exponent is obtained when $m_{\sigma(i)}=m_i$ for all $i=1,\ldots,r$, i.e., when the permutation $\sigma$ leaves each of the blocks $I_1,\ldots,I_{s+1}$ invariant. For such a permutation, the dominant term is at most $e^{-\beta t}$, where
\[
\beta:=\sum_{i=1}^r\lambda_{m_i}=\sum_{k=1}^{s+1}\lambda_k|I_k|.
\]
The dominant term corresponding to $e^{-\beta t}$ in $\det(a(t))$ is the sum
\[
\sum_\sigma\sign(\sigma)a_{1\sigma(1)}(t)\cdots a_{r\sigma(r)}(t)
\]
over those permutations $\sigma\in S_r$ leaving the blocks $I_1,\ldots,I_{s+1}$ invariant. Since,  by construction $(\p y_i | \p y_j)_\sig=\delta_{ij}$, for $i,j\in I_k$, the coefficient of $e^{-\beta t}$ is $1$ . We conclude  the dominant term of $\det(a(t))$ is actually $e^{-\beta t}$, hence
\begin{multline*}
-\lim_{t\to\infty}\frac d{dt}\ell_U(\gamma(t))=\beta=\sum_{k=1}^{s+1}\lambda_k|I_k|
=\sum_{k=1}^{s+1}\lambda_k \bigl(\dim(U\cap V_k)-\dim(U\cap V_{k-1})\bigr)\\
=\sum_{k=1}^s \alpha_k \dim(U\cap V_k)+\lambda_{s+1}r
=\sum_{k=1}^s\alpha_k\bigl(\dim(U\cap V_k)-\frac r m\dim(V_k)\bigr).
\end{multline*}
Formula~\eqref{asymptotic} follows by integrating with respect to $\shp$ and using the Lebesgue dominated convergence theorem to interchange integration and limit.
\end{proof}

\proof[Proof of Theorem~\ref{mainadd}]

\noindent(1)
Suppose that $I_\shp(V)>0$ for all $V\in\pls$.  Then, according to Proposition~\ref{thasymptotic}, $\lim\limits_{t\to\infty}(d/dt)\ell_\shp(\gamma(t))>0$ hence $\lim\limits_{t\to\infty}\ell_\shp(\gamma(t))=\infty$
for any nonconstant geodesic $\gamma$ in $\pas$. We conclude that the log-likelihood $\ell_\shp$, as a convex function tending to $+\infty$ along any geodesic, admits a minimum $\sig_0\in\pas$. Suppose it had another minimum $\sig_1$. Let $\gamma$ be the geodesic joining $\sig_0=\gamma(t_0)$ to $\sig_1=\gamma(t_1)$. Since the convex function $\ell_\shp(\gamma(t))$ takes its minimum at both~$t_0$ and~$t_1$, it should be constant for $t_0\le t\le t_1$, hence for all $t\in\R$ by analyticity of $\ell_\shp(\gamma(t))$ , in contradiction with $\lim\limits_{t\to\infty}\ell_\shp(\gamma(t))=\infty$.

\smallskip\noindent(2)
Suppose that $I_\shp(V)<0$ for some $V\in\pls$. By Lemma \ref{lem::grad_formula} and the definition of $\pi_{V}(\Sigma)$ from Lemma \ref{lem::covariant_deriv},  let $\gamma(t)= e^{tv} $ be the geodesic of velocity 
$v:= \pi_{V}(\Id_m) -\frac{\dim(V)}{m}\Id_m \in \TT_{\Id_m} \Pos_{sym}^1(m)$ issuing from $\Id_m \in \pas$. By Proposition~\ref{thasymptotic}, $\lim\limits_{t\to\infty}(d/dt)\ell_\shp(\gamma(t))=I_\shp(V)<0$ hence $\lim\limits_{t\to\infty}\ell_\shp(\gamma(t))=-\infty$, so that the log-likelihood function $\ell_\shp$ has no minimum.

\smallskip\noindent(3)
Suppose that $I_\shp(V)\ge0$ for all $V\in\pls$ and $I_\shp(V)=0$ for some $V\in\pls$. Let us prove
(b). By hypothesis, there exists $V\in\pls$ with ${I_\shp(V)=0}$ such that $I_\shp(V')>0$ for every complementary subspace $V'$ of $V$. Suppose that $\sig\in\pas$ was a minimum of $\ell_\shp$.  Let $\gamma(t)= g e^{tv}g^T $ be the geodesic of velocity  $v:= g^{T}\pi_{V}(\sig)g -\frac{\dim(V)}{m}\Id_m \in \TT_{\Id_m} \Pos_{sym}^1(m)$ issuing from $\sig= gg^{T} \in \pas$. 
 The velocity of the reverse geodesic $t\mapsto\gamma(-t)$  is $-v=g^{T}\pi_{V'}(\sig) g-\frac{\dim(V')}{m} \Id_m$, where 
$V'=\{x'\in\ls\mid(x'|x)_\sig=0, \; \forall x \in V\}$ is the $\sig$-orthogonal complement of $V$. 
According to Proposition~\ref{thasymptotic}, $\lim\limits_{t\to\infty}(d/dt)\ell_\shp(\gamma(t))=I_\shp(V)=0$, and  
$\lim\limits_{t\to\infty}(d/dt)\ell_\shp(\gamma(-t))=I_\shp(V')>0$. The convex function $\ell_\shp(\gamma(t))$ is thus decreasing, in contradiction with the minimality of $\ell_\shp(\sig)$.
We omit the lengthy proof of (a) since it is similar to the proof for the case $r=1$ given in Section 8 of \cite{AMR2005}.

\endproof

\section{Almost sure convergence}
\label{sec::gen_lln}

The following theorem generalises Tyler's convergence result of \gmle\ (\cite[Theorem 3.1]{Ty})  to the case of Grassmannian $G(m,r)$, for every $1 \leq r \leq m-1 $.

\begin{theorem}
\label{thm::LLN}
Let $\{U_n\}_{n\geq 1}$ be  i.i.d. 
random elements of $G(m,r)$, distributed
according to a continuous Borel probability measure $\shp$ on $G(m,r)$ with support $G(m,r)$. Let $\shp_n= \frac{1}{n} (\delta_{U_1}+...+ \delta_{U_n})$ be the  empirical probability measure corresponding to the random sample $\{U_1, ..., U_n\}$.  Suppose  GEs $\{\Sigma_n \}_{n \geq 1}, \Sigma_\shp \subset \Pos_{sym}^1(m)$ for $\{\shp_n\}_{n \geq 1}$, respectively, $\shp$ exist. Then for every $n$ large enough the $\Sigma_n$ is unique almost surely, $\Sigma_\shp \in \Pos_{sym}^1(m)$ is unique, and $\{\Sigma_n \}_{n \geq 1}$ converges almost surely to $\Sigma_\shp$.
\end{theorem}

\medskip
The main idea of the proof, inspired from Tyler \cite[Theorem 3.1]{Ty}, is to study the square of the norm of the gradients of the log-likelihood functions $\{\ell_{\shp_n}\}_{n \geq 1}, \ell_{\shp} : \Pos_{sym}^1(m) \to \RR$ associated, respectively, with the probability measures $\{\shp_n\}_{n \geq 1}$ and $\shp$. Notice those gradients are vectors in the tangent bundle $\TT\Pos_{sym}^1(m)$ and their norm are given by the scalar product which is the trace on matrices.

Let $\{h_n\}_{n \geq 1}, h : \Pos_{sym}^1(m) \to \RR$ be  those functions 
$$h_n(\Gamma):=\langle \grad(\ell_{\shp_n})(\Gamma), \grad(\ell_{\shp_n})(\Gamma) \rangle_{\Gamma}$$
$$h(\Gamma):=\langle \grad(\ell_{\shp})(\Gamma), \grad(\ell_{\shp})(\Gamma) \rangle_{\Gamma}.$$

By Lemma \ref{lem::strict_conv_1}, our hypothesis on $\shp$, and without loss of generality we can suppose the GE of $\shp$ is $\Sigma_\shp= \Id_m$. Moreover, from Example \ref{ex::cond_1} for every $n$ large enough the log-likelihood function $\ell_{\shp_n}$ has a unique GE, denoted  $\Sigma_n$, almost surely. 
Then taking a compact neighbourhood $C$ of $\Id_m$ we need to prove for every  $n$ large enough $\Sigma_n$ is in $C$ almost surely.  The two main ingredients are Lemma \ref{lem::equicont}, and Lemma \ref{lem::grad_h_n} together with Remark \ref{rem::main} (or for a second proof using strict convexity of $\ell_{\shp_n}$ from Lemmas \ref{lem::strict_conv_x} and \ref{lem::strictly_convex_fct}). Using Lemma \ref{lem::equicont}, for $n$ large enough  we show the compact neighbourhood $C$ contains a local minimum of $h_n$, almost surely.  Then for a first proof as in \cite[Theorem 3.1]{Ty} this local minimum is a solution to $\grad(h_n)(\Gamma)=0$, which by Remark \ref{rem::main}  must be the GE  of $\ell_{\shp_n}$ and that is exactly  $\Sigma_n$. As a second proof one can use the strict convexity of $\ell_{\shp_n}$ from Lemma \ref{lem::strict_conv_x} and then Lemma \ref{lem::strictly_convex_fct} to conclude the local minimum of $h_n$ is the GE  of $\ell_{\shp_n}$ almost surely.


\medskip
Let us now make explicit the main idea of the proof.
By definition $\{\Sigma_n\}_{n \geq 1}$ and $\Sigma_\shp \in \Pos_{sym}^1(m)$ satisfy, respectively, the following M-equations (see equality \eqref{Mequation} or Corollary \ref{M-equation})
\begin{equation}
\label{equ::gradient_i}
\frac{r}{2m} \Sigma_n= \frac{1}{n} \sum_{j =1}^{n}\frac{1}{2} X_j (X_j^{T} \Sigma_n^{-1} X_j)^{-1} X_j^{T}
\end{equation}

\begin{equation}
\label{equ::gradient_p}
\frac{r}{2m} \Sigma_\shp= \int_{G(m,r)}\frac{1}{2} X (X^{T} \Sigma_\shp^{-1} X)^{-1} X^{T} d\shp(U),
\end{equation}
where $X_n$ and $X$ are matrices that correspond to $U_n$, $U$, as presented in the beginning of Section \ref{Introduction_summary}.
\begin{definition}
\label{def::gradient_functional}
Let $\PosS_{sym}(m):=\{ M \in \Sym(m) \; \vert \; x M x^{T} \geq 0, \forall x \neq 0 \in \RR^{m} \}$ the set of all symmetric and positive semidefinite $m \times m$ matrices. For every $n \geq 1$ we define 
$$M_n : \Pos_{sym}^1(m) \to  \PosS_{sym}(m)$$
$$ \Gamma =g g^{T}  \mapsto M_n(\Gamma) := \frac{1}{n} \sum_{j =1}^{n}g^{-1}X_j (X_j^{T} \Gamma^{-1} X_j)^{-1} X_j^{T} (g^{T})^{-1}.$$

In the same way we define 
$$M : \Pos_{sym}^1(m) \to  \PosS_{sym}(m)$$
$$ \Gamma =g g^{T}  \mapsto M(\Gamma) := \int_{G(m,r)}g^{-1}X (X^{T} \Gamma^{-1} X)^{-1} X^{T} (g^{T})^{-1} d\shp(U).$$
\end{definition}

\begin{lemma}
\label{lem::well_def_M}
The maps $M, \{M_n\}_{n \geq 1}$ are well defined. In particular, for every $\Gamma \in \Pos_{sym}^1(m)$ we have $M(\Gamma), \{M_n(\Gamma)\}_{n \geq 1}$ are symmetric and positive semidefinite $m\times m$ real matrices and $\tr(M(\Gamma))=\tr(M_n(\Gamma))= r$.
\end{lemma}
\begin{proof}
The symmetric and positive semidefinite properties of $M(\Gamma), \{M_n(\Gamma)\}_{n\geq 1}$ follow easily.
Using the fact that $\tr(AB)=\tr(BA)$ one sees $\tr(M(\Gamma))=\tr(M_n(\Gamma))= r$.
\end{proof}

The following is a well-known result.

\begin{lemma}
\label{lem::trace}
Let $A, B$ be two positive semidefinite $m \times m$ real matrices. Then $$0 \leq \tr(A B) \leq \tr(A) \cdot \tr(B).$$
\end{lemma}
\begin{proof}
See \cite{LT}, page 269.
\end{proof}

\begin{lemma}
\label{lem::general_fact}
Let $A$ be a symmetric $m \times m$ matrix such that $\tr(A)=r$. Then $\frac{r^2}{m} \leq \tr(A^2)$, with equality if and only if $A=\frac{r}{m}\Id_m$.
\end{lemma}
\begin{proof}
Notice $(A-\frac{r}{m}\Id_m)$ is a symmetric matrix and $\tr((A- \frac{r}{m}\Id_m)^2) \geq 0$. Then as $(A- \frac{r}{m}\Id_m)^2=A^2 +\frac{r^2}{m^2}\Id_m- 2 \frac{r}{m} A$ we have $\tr(A^2 + \frac{r^2}{m^2}\Id_m- 2\frac{r}{m} A)=\tr(A^2)+\frac{r^2}{m}-2\frac{r^2}{m}\geq 0$ and the conclusion follows.

Suppose $\tr(A^2)=\frac{r^2}{m}$. Then it is easy to see $A= \frac{r}{m} \Id_m$ verifies that. Now by writing explicitly the trace of $(A- \frac{r}{m}\Id_m)^2$ and imposing the condition $\tr((A- \frac{r}{m} \Id_m)^2)=0$ we obtain $A= \frac{r}{m}\Id_m$.
\end{proof}

By Lemmas \ref{lem::well_def_M}, \ref{lem::trace} and \ref{lem::general_fact} we obtain:
\begin{lemma}
\label{lem::ineg_trace_square}
Let $\Gamma \in \Pos_{sym}^1(m)$. Then for every $n \geq 1$ we have
$$\frac{r^2}{m} \leq \tr(M(\Gamma)^2) \leq r^2$$
$$\frac{r^2}{m} \leq \tr(M_n(\Gamma)^2) \leq r^2.$$
Moreover, $\frac{r^2}{m} = \tr(M(\Gamma)^2)$, $\frac{r^2}{m}= \tr(M_n(\Gamma)^2)$, respectively, if and only if $M(\Gamma)= \frac{r}{m}\Id_m$, $M_n(\Gamma)= \frac{r}{m}\Id_m$, respectively.
\end{lemma}

\begin{remark}
Let $\{U_n\}_{n \geq 1}$ be i.i.d. random elements of $G(m,r)$ distributed according to a continuous Borel probability measure $\shp$. 
Recall, the log-likelihood of $\{\shp_n\}_{n \geq 1}$ and $\shp$, respectively, are the continuous and differentiable maps given by: 
\begin{equation}
\label{equ::log_like_equ}
\begin{split}
& \{\ell_{\shp_n}\}_{n \geq 1}, \ell_{\shp} : \Pos_{sym}^1(m) \to \RR \\
&\Gamma \mapsto \ell_{\shp_n}(\Gamma) = \frac{1}{n}\sum_{j=1}^{n} \frac{1}{2} \log \frac{\det(X_j^{T} \Gamma^{-1} X_j)}{\det(X_j^{T}X_j)},\\
&\Gamma \mapsto \ell_{\shp}(\Gamma)= \int_{G(m,r)} \frac{1}{2} \log \frac{\det(X^{T} \Gamma^{-1} X)}{\det(X^{T}X)} d\shp(U).\\
\end{split}
\end{equation}

The gradients of $\{\ell_{\shp_n}\}_{n \geq 1}, \ell_{\shp} : \Pos_{sym}^1(m) \to \RR$, respectively, are the maps given by:
\begin{equation}
\label{equ::log_like_grad}
\begin{split}
& \{\grad(\ell_{\shp_n})\}_{n \geq 1}, \grad(\ell_{\shp}) : \Pos_{sym}^1(m) \to \TT \Pos_{sym}^1(m) \\
&\Gamma \mapsto \grad(\ell_{\shp_n})(\Gamma) = \frac{r}{2m} \Gamma- \frac{1}{2n}\sum_{j=1}^{n} X_j (X_j^{T} \Gamma^{-1} X_j)^{-1} X_j^{T}\\
&\Gamma \mapsto \grad(\ell_{\shp})(\Gamma) =  \frac{r}{2m} \Gamma- \frac{1}{2} \int_{G(m,r)} X (X^{T} \Gamma^{-1} X)^{-1} X^{T} d\shp(U).\\
\end{split}
\end{equation}

As $\Pos_{sym}^1(m)$ is a Riemannian manifold, we can isometrically transport the gradient maps $\{ \grad(\ell_{\shp_n})\}_{n \geq 1}, \grad(\ell_{\shp})$ to the tangent space $\TT_{\Id_m} \Pos_{sym}^1(m)$ at the identity matrix $\Id_m \in \Pos_{sym}^1(m)$:
\begin{equation}
\label{equ::log_like_grad_id_1}
\begin{split}
\Gamma = gg^{T} \mapsto g^{-1}\grad(\ell_{\shp_n})(\Gamma)(g^{T})^{-1} &= \frac{r}{2m} \Id_m- \frac{1}{2n}\sum_{j=1}^{n} g^{-1}X_j (X_j^{T} \Gamma^{-1} X_j)^{-1} X_j^{T}(g^{T})^{-1}\\
&= \frac{r}{2m} \Id_m- \frac{1}{2} M_n(\Gamma)\\
\end{split}
\end{equation}

\begin{equation}
\label{equ::log_like_grad_id_2}
\begin{split}
\Gamma =gg^{T} \mapsto g^{-1} \grad(\ell_{\shp})(\Gamma) (g^{T})^{-1} &=  \frac{r}{2m} \Id_m- \frac{1}{2} \int_{G(m,r)} g^{-1} X (X^{T} \Gamma^{-1} X)^{-1} X^{T} (g^{T})^{-1} d\shp(U)\\
&= \frac{r}{2m} \Id_m-  \frac{1}{2} M(\Gamma).\\
\end{split}
\end{equation}
\end{remark}

\begin{definition}
\label{def::trace_grad}
Let $\{U_n\}_{n \geq 1} $ be i.i.d. random elements of $G(m,r)$, distributed according to a  probability measure $\shp$ on $G(m,r)$. 
As defined above, we have: $$\{h_n\}_{n \geq 1}, h : \Pos_{sym}^1(m) \to \RR$$ 
\begin{equation}
\label{def::trace_grad_like_i}
\begin{split}
\Gamma = gg^{T} \mapsto h_n(\Gamma)&= \langle  g^{-1}\grad(\ell_{\shp_n})(\Gamma)(g^{T})^{-1}, g^{-1}\grad(\ell_{\shp_n})(\Gamma)(g^{T})^{-1} \rangle_{\Id_m} \\
& =\langle \grad(\ell_{\shp_n})(\Gamma), \grad(\ell_{\shp_n})(\Gamma) \rangle_{\Gamma}\\
&= \tr(\Gamma^{-1}\grad(\ell_{\shp_n})(\Gamma)\Gamma^{-1}\grad(\ell_{\shp_n})(\Gamma))\\
&= \frac{r^2}{4m}-\frac{r^2}{2m} +  \frac{1}{4} \tr(M_n(\Gamma)^2)\\
\end{split}
\end{equation}
 
\begin{equation}
\label{def::trace_grad_like}
\begin{split}
\Gamma = gg^{T} \mapsto h(\Gamma)&= \langle  g^{-1}\grad(\ell_{\shp})(\Gamma)(g^{T})^{-1}, g^{-1}\grad(\ell_{\shp})(\Gamma)(g^{T})^{-1} \rangle_{\Id_m} \\
& =\langle \grad(\ell_{\shp})(\Gamma), \grad(\ell_{\shp})(\Gamma) \rangle_{\Gamma}\\
&= \tr(\Gamma^{-1}\grad(\ell_{\shp})(\Gamma)\Gamma^{-1}\grad(\ell_{\shp})(\Gamma))\\
&= \frac{r^2}{4m}-\frac{r^2}{2m} + \frac{1}{4}\tr(M(\Gamma)^2).\\
\end{split}
\end{equation}
\end{definition}

\begin{lemma}
\label{lem::grad_h_n}
Let $\{U_i\}_{1 \leq i \leq n } \subset G(m,r)$ be a sample of size $n>0$. Let $\shp_n:= \frac{1}{n} (\delta_{U_1}+...+ \delta_{U_n})$ be the corresponding empirical probability measure and $\ell_{\shp_n}$ the associated log-likelihood function. Then the gradient of the function $h_n: \Pos_{sym}^1(m) \to \RR$ defined by $\Gamma \mapsto h_n(\Gamma):= \langle \grad(\ell_{\shp_n})(\Gamma), \grad(\ell_{\shp_n})(\Gamma) \rangle_{\Gamma}$ is given by
\begin{equation*}
\begin{split}
\grad(h_n)(\Gamma)&= \frac{1}{2n^2}\Gamma \left(\sum\limits_{j=1}^{n} \pi_{U_{j}}(\Gamma) \Gamma \left(\sum\limits_{i=1}^n \pi_{U_{i}}(\Gamma)\right) \Gamma  \pi_{U_{j}}(\Gamma)  \right) \Gamma \\
& \; \; - \frac{1}{2n^2} \Gamma \left(\sum\limits_{i=1}^{n} \pi_{U_{i}}(\Gamma)\right) \Gamma \left(\sum\limits_{i=1}^{n} \pi_{U_{i}}(\Gamma)\right) \Gamma,
\end{split}
\end{equation*}
where $\pi_{U}(\Gamma):= \Gamma^{-1} X (X^{T} \Gamma^{-1}  X)^{-1} X^{T} \Gamma^{-1}$.

In particular, $\grad(h_n)(\Gamma)=0$ if and only if
\begin{equation}
\label{equ::grad_h_n1}
\sum\limits_{j=1}^{n} \pi_{U_{j}}(\Gamma) \Gamma \left(\sum\limits_{i=1}^n \pi_{U_{i}}(\Gamma)\right) \Gamma  \pi_{U_{j}}(\Gamma)  =  \left(\sum\limits_{i=1}^{n} \pi_{U_{i}}(\Gamma)\right) \Gamma \left(\sum\limits_{i=1}^{n} \pi_{U_{i}}(\Gamma)\right).
\end{equation}
\end{lemma}
\begin{proof}
The proof is provided in Section \ref{sec::Appendix}.
\end{proof} 

\begin{lemma}
\label{lem::inverse_pi_u}
Let $\{U_i\}_{1 \leq i \leq n } \subset G(m,r)$ be a sample of size $n>0$ such that $\RR^m$ is spanned by $\{U_1,...,U_n\}$.
Then the $m\times m$ real matrix $\sum\limits_{i=1}^{n} \pi_{U_{i}}(\Gamma)$ is invertible for every $\Gamma \in \Pos_{sym}^1(m)$.
\end{lemma}
\begin{proof}
Let $\Gamma=gg^{T} \in \Pos_{sym}^1(m)$, and notice by the definition $\sum\limits_{i=1}^{n} \pi_{U_{i}}(\Gamma)$ is a symmetric matrix.  Moreover, by Lemma \ref{lem::orth_proj} we also have $\sum\limits_{i=1}^{n} \pi_{U_{i}}(\Gamma)= \Gamma^{-1}  \sum\limits_{i=1}^{n} {\rm Pr}(U_i, \Gamma)$.

It is enough to prove $\sum\limits_{i=1}^{n} g^{-1}{\rm Pr}(U_i, \Gamma)g$ is invertible, this is equivalent to proving the kernel of $\sum\limits_{i=1}^{n} g^{-1}{\rm Pr}(U_i, \Gamma)g$ contains only the zero vector in $\RR^m$. Indeed, by its definition ${\rm Pr}(U, \Gamma)$ is the $\Gamma$-orthogonal projection onto $U$, and so $g^{-1}{\rm Pr}(U, \Gamma)g$ is the $\Id_m$-orthogonal projection onto $g^{-1}U \in G(m,r)$, and $g^{-1}U^{\perp_{\Sigma}} = (g^{-1}U)^{\perp_{\Id_m}}=(g^{-1}U)^{\perp}$. Let $v \in \RR^m \setminus \{0_m\}$ and let $v^{\perp} \subset \RR^m$ be the $\Id_m$-orthogonal hyperplane on $v$. Because $\RR^m$ is spanned by $\{g^{-1}U_1,...,g^{-1}U_n\}$ there is at least one $g^{-1}U_j$ that is not a subset of $v^{\perp}$. In particular, for that $j$ we have $g^{-1}{\rm Pr}(U_j, \Gamma)g(v)\neq 0_m$. Moreover, if not zero, the vector $g^{-1}{\rm Pr}(U_i, \Gamma)g(v)$ is on the same side of the hyperplane $v^{\perp}$ as $v$, and so $\sum\limits_{i=1}^{n} g^{-1}{\rm Pr}(U_i, \Gamma)g(v) \neq 0_m$. 
\end{proof}

\begin{remark}
\label{rem::main}
Let $\{U_i\}_{1 \leq i \leq n } \subset G(m,r)$ be a sample of size $n>0$ such that $\RR^m$ is spanned by $\{U_1,...,U_n\}$. Then by Lemma \ref{lem::inverse_pi_u} equation (\ref{equ::grad_h_n1}) becomes 
\begin{equation}
\label{equ::grad_h_n2}
 \sum\limits_{j=1}^{n} \pi_{U_{j}}(\Gamma) \Gamma \left(\sum\limits_{i=1}^n \pi_{U_{i}}(\Gamma)\right) \Gamma  \pi_{U_{j}}(\Gamma) \left(\sum\limits_{i=1}^{n} \pi_{U_{i}}(\Gamma)\right)^{-1} =  \left(\sum\limits_{i=1}^{n} \pi_{U_{i}}(\Gamma)\right) \Gamma
\end{equation}
and taking the trace in (\ref{equ::grad_h_n2}) we obtain
\begin{equation}
\label{equ::grad_h_n3}
\begin{split}
 \tr \left(\sum\limits_{j=1}^{n} \pi_{U_{j}}(\Gamma) \Gamma \left(\sum\limits_{i=1}^n \pi_{U_{i}}(\Gamma)\right) \Gamma  \pi_{U_{j}}(\Gamma) \left(\sum\limits_{i=1}^{n} \pi_{U_{i}}(\Gamma)\right)^{-1} \right)&= \tr\left(\sum\limits_{i=1}^{n} \pi_{U_{i}}(\Gamma) \Gamma \right)\\
 &=n\tr(\Id_{r\times r})\\
 &=n\cdot r.\\
 \end{split}
\end{equation}
Therefore
\begin{equation}
\label{equ::grad_h_n4}
 \tr \left(\frac{1}{n}\sum\limits_{j=1}^{n} \pi_{U_{j}}(\Gamma) \Gamma \left(\sum\limits_{i=1}^n \pi_{U_{i}}(\Gamma)\right) \Gamma  \pi_{U_{j}}(\Gamma) \left(\sum\limits_{i=1}^{n} \pi_{U_{i}}(\Gamma)\right)^{-1} \right)=r.
\end{equation}
By the Definition \ref{def::gradient_functional} of $M_n(\Gamma)$ the above equality (\ref{equ::grad_h_n4}) is equivalent to
\begin{equation}
\label{equ::grad_h_n5}
 \tr \left(\frac{1}{n}\sum\limits_{j=1}^{n} g^{-1} X_j (X_{j}^{T} \Gamma^{-1}  X_j)^{-1} X_{j}^{T} (g^{T})^{-1}  M_{n}(\Gamma) \;g^{-1} X_j (X_{j}^{T} \Gamma^{-1}  X_j)^{-1} X_{j}^{T} (g^{T})^{-1} M_{n}(\Gamma)^{-1}\right)=r.
\end{equation}

Notice, for every $j \in \{1,...,n\}$ the $r \times r$ matrix $X_{j}^{T} \Gamma^{-1}  X_j$ is symmetric and also positive definite, as $\Gamma^{-1}= (g^{T})^{-1}g^{-1}$. Therefore, one can write $X_{j}^{T} \Gamma^{-1}  X_j= h_j^{T}h_j$, with  $h_j$ an $r\times r$ real matrix.

Denoting $g^{-1} X_j h_j^{-1}=: \theta_j$, equation (\ref{equ::grad_h_n5}) is further equivalent to
\begin{equation}
\label{equ::grad_h_n6}
 \tr \left(\frac{1}{n}\sum\limits_{j=1}^{n} \theta_j^{T} M_{n}(\Gamma)\theta_j \theta_j^{T} M_{n}(\Gamma)^{-1} \theta_j \right)=r.
\end{equation}
\end{remark}

\begin{lemma}
\label{lem::equicont}
Let $\{U_n\}_{n \geq 1} $ be  i.i.d. random elements of $G(m,r)$, distributed
according to a continuous Borel probability measure $\shp$ on $G(m,r)$.  
Let $C \subset  (\Pos_{sym}^1(m), \vert \vert \cdot \vert \vert)$ be a compact set. Then 
$$\sup_{\Gamma \in C} \vert h_n(\Gamma)- h(\Gamma) \vert \xrightarrow[n \to \infty]{} 0 \ \ \ \text{almost surely}. $$
\end{lemma}
\begin{proof}
This is the same proof as in Tyler \cite[Statement (3.2)]{Ty}.

Recall an element $U$ of $G(m,r)$ is the linear span of $r$ linearly independent vectors say $\{x_1, x_2, \cdots, x_r\}$. With $U = \langle x_1, x_2, \cdots, x_r  \rangle$ we associate the $m \times r$ matrix $X$ given by $(x_1, x_2, \cdots, x_r)$, with $x_i$ the columns of $X$ that are written in the canonical base of $\RR^m$. Then we define the map 

$$G : G(m,r) \times  \Pos_{sym}^1(m) \to \PosS_{sym}(m)$$
$$ (U, \Gamma = g g^{T}) \mapsto G(U,\Gamma):=g^{-1}X (X^{T} \Gamma^{-1} X)^{-1} X^{T} (g^{T})^{-1}.$$

As $G(m,r) \times C$ is compact and the map $G$ is continuous, being given by multiplication of matrices (one can prove that using the topologies on $G(m,r) $ and $\Pos_{sym}^1(m)$), the map $G$ is uniformly continuous on $G(m,r) \times C$. In particular, for every $\epsilon >0$ there exists $\delta_{\epsilon}>0$, that does not depend on $U \in G(m,r)$, $\Gamma_0 \in C$, nor $\Gamma \in C$, such that if $\vert \vert \Gamma_0 - \Gamma \vert \vert < \delta_\epsilon$ then $\vert \vert G(U,\Gamma_0) -G(U,\Gamma) \vert \vert  < \epsilon$, for every $ U \in G(m,r)$.

Since $C$ is compact, for every $\delta_\epsilon >0$ there exist a finite partition of $C$, say $C_{\epsilon, k}$ such that $\vert \vert \Gamma_1 - \Gamma_2 \vert \vert  < \delta_\epsilon$ for every $\Gamma_1, \Gamma_2 \in C_{\epsilon, k}$. Choose one element from each sets $C_{\epsilon, k}$, say $\Gamma_{\epsilon, k}$, and label this new formed set by $C_\epsilon$.

Since $\shp(X=0)=0$, it follows from the uniform continuity of $G$ and the Law of Large Numbers that: 
\begin{equation*}
\begin{split}
\sup_{\Gamma \in C} \vert \vert M_n(\Gamma)- M(\Gamma) \vert \vert &\leq \max_{k} \sup_{\Gamma \in C}(\vert \vert M_n(\Gamma) -M_n(\Gamma_{\epsilon, k})\vert \vert + \vert \vert M(\Gamma) - M(\Gamma_{\epsilon, k})\vert \vert ) +\\
&\; \; + \max_{k} \vert \vert  M_n(\Gamma_{\epsilon, k}) - M(\Gamma_{\epsilon, k})\vert \vert  \\
& \leq 2 \epsilon + \max_{k} \vert \vert M_n(\Gamma_{\epsilon, k})- M(\Gamma_{\epsilon, k})\vert \vert  \xrightarrow{n \to \infty} 2 \epsilon\\
\end{split}
\end{equation*}
almost surely. Since $\epsilon$ can be made arbitrarily small, it follows that $$\sup_{\Gamma \in C} \vert \vert M_n(\Gamma)-M(\Gamma) \vert \vert  \xrightarrow{n \to \infty} 0$$ almost surely. Then $\sup_{\Gamma \in C} \vert \tr(M_n(\Gamma)^2)- \tr(M(\Gamma)^2) \vert \xrightarrow{i \to \infty} 0$ almost surely. So the conclusion follows.
\end{proof}

\begin{proof}[Proof of Theorem \ref{thm::LLN}]
The proof goes as in Tyler \cite[Theorem 3.1]{Ty}.

By Lemma \ref{lem::strict_conv_1} the map $\ell_{\shp}$ is strictly geodesically convex. As by hypothesis a GE for  $\ell_{\shp}$ exists it must be unique. Moreover, without loss of generality we can suppose the GE of $\shp$ is $\Sigma_\shp= \Id_m$. Then $M(\Sigma_\shp)= M(\Id_m)=  \frac{r}{m} \Id_m$.

From Example \ref{ex::cond_1} for every $n$ large enough the (random) log-likelihood function $\ell_{\shp_n}$ has a unique GE, almost surely.  As by hypothesis a GE for  $\ell_{\shp_n}$ exists, this is $\Sigma_n $.

It remains to prove $\{\Sigma_n\}_{n \geq 1}$ converges almost surely to $\Sigma_\shp= \Id_m$. Indeed, let $C$ be a compact neighbourhood of $\Id_m$, with respect to the max-norm topology on $\Pos_{sym}^1(m)$ induced from $\Sym(m)$. Moreover, we choose $C$ such that $\Id_m$ is in the interior of $C$.  By the hypotheses on $\ell_{\shp}$ and $\Sigma_\shp$, from Lemma \ref{lem::ineg_trace_square} we have $\frac{r^2}{m} = \tr(M(\Id_m)^2) < \tr(M(\Gamma)^2)$ for every $\Gamma \in C \setminus \{\Id_m\}$. As the map $h$ is continuous, Lemma \ref{lem::equicont} implies with probability one that for every large enough $n$, $h_n(\Id_m)$ is less than $h_n(\Gamma)$, for every any $\Gamma$ on the boundary of $C$. Hence, with probability one, the map $h_n$ contains a local minimum in $C$ for every $n$ large enough.  As a first variant of proof and following the ideas from \cite[Theorem 3.1]{Ty}, this local minimum $\Gamma_n$ must be a critical point of $h_n$ and so a solution to $\grad(h_n)(\Gamma)=0$. Indeed, for $n>m$ and because $\{U_n\}_{n \geq 1} $ are of i.i.d. random elements of $G(m,r)$, distributed
according to a continuous Borel probability measure $\shp$ on $G(m,r)$, with probability one the space $\RR^m$ is spanned by $\{U_n\}_{n \geq 1}$. By Lemma \ref{lem::inverse_pi_u} and Remark \ref{rem::main} the local minimum $\Gamma_n$ verifies equality (\ref{equ::grad_h_n6}) almost surely. 

Notice, $M_n(\Gamma)$ is symmetric and positive definite, and by the matrix version of Kantorovich's inequality \cite[Formula (6)]{MarOl}, keeping the notation from (\ref{equ::grad_h_n6}), one has 
$$\tr\left(\theta_j^{T} M_{n}(\Gamma)\theta_j \theta_j^{T} M_{n}(\Gamma)^{-1} \theta_j \right) \geq r$$
with equality if and only if $\theta_j^{T}M_{n}(\Gamma)= \theta_j^{T}$.

Therefore, equality (\ref{equ::grad_h_n6}) holds if and only if  $\theta_j^{T}M_{n}(\Gamma_n)= \theta_j^{T}$, for every $j \in\{1,...,n\}$. And this is true only if $M_{n}(\Gamma_n)$ is a multiple of $\Id_m$. As $\tr(M_n(\Gamma_n))=r$ we must have $M_n(\Gamma_n))=\frac{r}{m}\Id_m$. Thus the local minimum $\Gamma_n$ is the GE of $\ell_{\shp_n}$, which is exactly  $\Sigma_n$.  This implies with probability one that $\Sigma_n \in C$. 

As a second variant of proof, again the local minimum $\Gamma_n \in C$ of $h_n$ must be a critical point of $h_n$ and so a solution to $\grad(h_n)(\Gamma)=0$. Indeed, one can use the strict convexity of $\ell_{\shp_n}$ from Lemma \ref{lem::strict_conv_x} and then Lemma \ref{lem::strictly_convex_fct} to conclude the local minimum of $h_n$ is the GE  $\Sigma_n$ of $\ell_{\shp_n}$ almost surely.  As $C$ can be chosen arbitrarily small, the conclusion of the theorem follows.
\end{proof}

\section{Asymptotic normality }
\label{sec::gen_CLT}

Let $\{U_n\}_{n \geq 1}$  be a  sample of i.i.d. random elements, distributed according to a  Borel probability measure $\shp$ on $G(m,r)$. Let $\{\shp_n:= \frac{1}{n} (\delta_{U_1}+...+ \delta_{U_n})\}_{n \geq 1}$ be the (random) empirical probability measures corresponding to the random samples $\{U_1, ..., U_n\}$. Suppose  GEs $\{\Sigma_n \}_{n \geq 1}, \Sigma_\shp \subset \Pos_{sym}^1(m)$ for $\{\shp_n\}_{n \geq 1}$, respectively, $\shp$ exist. 

Write $\{\Sigma_n= g_n g_n\}_{n \geq 1}$ and $\Sigma_\shp = g g$, with $\{ g_n \}_{n \geq 1},  g  \in \SL(n, \RR) $ the unique symmetric positive-definite square roots of $\{\Sigma_n\}_{n \geq 1}$ and $\Sigma_\shp $, respectively.  Recall  $\{\Sigma_n\}_{n \geq 1}, \Sigma_\shp $ satisfy equations (\ref{equ::gradient_i}), (\ref{equ::gradient_p}), respectively:
\begin{equation*}
\frac{r}{m} \Id_m= \frac{1}{n} \sum_{j =1}^{n} g_n^{-1} X_j (X_j^{T} \Sigma_n^{-1} X_j)^{-1} X_j^{T}g_n^{-1}= M_n(\Sigma_n),
\end{equation*}

\begin{equation*}
\frac{r}{m} \Id_m= \int_{G(m,r)}g^{-1} X (X^{T} \Sigma_\shp^{-1} X)^{-1} X^{T} g^{-1} d\shp(U)= M(\Sigma_\shp),
\end{equation*}
where $X_j$ and $X$ are matrices that correspond to $U_j$, $U$, as presented in the beginning of Section \ref{Introduction_summary}. 


\begin{theorem}\label{CLT}
Let $\{U_n\}_{n \geq 1}$  be a  sample of i.i.d. random elements, distributed according to a continuous Borel probability measure $\shp$ on $G(m,r)$ with support $G(m,r)$.
Let $ C_n := \frac{m}{\tr(\Sigma_{\shp}^{-1} \Sigma_n)}g^{-1} \Sigma_n g^{-1}$.
Then
\begin{equation}\label{CLT2}
\sqrt{n}(\Veco(C_n - \Id_m)) \xrightarrow[distribution]{n \to \infty}
\mathcal{N}(0, \sigma_\infty^2) ,
\end{equation}
where the limiting covariance matrix $\sigma_\infty^2$ is given by $ \left[ Q L_0 Q \right]^+ \sigma^2 \Big(\left[ Q L_0 Q \right]^+\Big)^T$, where $Q$ is the orthogonal projection onto ${\rm Im}(\sigma^2)\in\R^{m^2}$, $\left[ A \right]^+$ denotes the Moore--Penrose inverse of a matrix $A$, $L_0 :=\frac{r}{m} \Id_{m^{2}} - \Sigma_0  \left( \Id_m  \otimes \Id_m \right)$, $\Sigma_0:=\E_\shp(\Theta (\Theta^T\Theta)^{-1}\Theta^T\otimes \Theta (\Theta^T\Theta)^{-1}\Theta^T)$, and $\Theta :=g^{-1}X$.
\end{theorem}

The proof of Theorem \ref{CLT} makes strong use of 
$$M_n(\Sigma_\shp)=\frac{1}{n}\sum_{j=1}^n \Theta_j (\Theta_j^T\Theta_j)^{-1}\Theta_j^T,$$
where we set $\Theta_j :=g^{-1}X_j$, for each $n\ge 1$ and $j\in\{1,\cdots,n\}$, see Definition \ref{def::gradient_functional}.
Notice, for each $j \in \{1,\cdots, n\}$, $ g\Theta_j (\Theta_j^T\Theta_j)^{-1}\Theta_j^T g^{-1}$ equals the projector  ${\rm Pr}(U_j,\Id_m)$ onto the linear subspace $U_j$ generated  by the columns of $\Theta_j$ (see Lemma \ref{lem::orth_proj}). Then $M_n(\Sigma_\shp)$
is a sum of i.i.d. random matrices of mean $\frac{r}{m} \Id_m$.
The multivariate central limit theorem yields that
$\sqrt{n}(\Veco(M_n(\Sigma_\shp)-  \frac{r}{m} \Id_m))$
is asymptotically multivariate normal centred of 
covariance matrix
\begin{equation}
\label{eq::covariance}
\begin{split}
\sigma^2 &:= \E_\shp(\Veco(\Theta (\Theta^T\Theta)^{-1}\Theta^T - \frac{r}{m} \Id_m)(\Veco(\Theta (\Theta^T\Theta)^{-1}\Theta^T - \frac{r}{m} \Id_m)^{T})\\
&= \E_\shp(\Veco(\Theta (\Theta^T\Theta)^{-1}\Theta^T)
\Veco(\Theta (\Theta^T\Theta)^{-1}\Theta^T)^T -\frac{r^2}{m^2} \Veco(\Id_m) \Veco(\Id_m)^T,\\
\end{split}
\end{equation}
where $\E_\shp(\cdot)$ denotes the expectation on $G(m,r)$ with respect to the probability measure $\shp$.
Recall the following formula regarding Kronecker product and $\Veco$ operator for arbitrary  $n \times m$ matrix $A$, $m \times k$ matrix $X$ and $k\times l$ matrix $B$: 
\begin{equation}
\label{Kronecker_prod}
\Veco(AXB)= (B^{T} \otimes A) \Veco(X),
\end{equation}
see, e.g., \cite{Muir}.
A direct application of (\ref{Kronecker_prod}) to the above expression yields
\begin{equation}
\label{eq::covariance_kron}
\sigma^2 = \E_\shp(\Theta\otimes\Theta \Veco((\Theta^T\Theta)^{-1})
\Veco((\Theta^T\Theta)^{-1})^T \Theta^T\otimes\Theta^T)
 -\frac{r^2}{m^2}\Veco(\Id_m) \Veco(\Id_m)^T.
 \end{equation}
 We next give two lemmas which are needed to prove Theorem \ref{CLT}:

\begin{lemma}
\label{lem::ker_sigma}
Let  $\shp$ be a Borel probability measure on $G(m,r)$ and let 
\begin{equation*}
W_{\shp}:= \spn\{g^{-1} X (X^{T} \Sigma_\shp^{-1} X)^{-1} X^{T} g^{-1}-  \frac{r}{m} \Id_m \; \vert \; U \in \supp(\shp) \subset G(m,r) \}
\end{equation*}
where $X$ is a matrix corresponding to  $U \in G(m,r) $ as presented in the beginning of Section \ref{Introduction_summary}. 
Then $W_{\shp}$ is a vector subspace of $(\TT_{\Id_m} \mkern-4mu \Pos_{sym}^1(m), \langle , \rangle_{\Id_m}) \subset (\RR^{m^2}, \langle , \rangle_{\Id_m})$, $\sigma^2$ is positive semidefinite, 
$$\Ker(\sigma^2)= W_{\shp}^{\perp}:= \{ v \in \RR^{m^2} \; \vert \;  \langle v,u, \rangle_{\Id_m}=0, \; \forall u \in W_{\shp}\}
\hbox{ and }{\rm Im}(\sigma^2)=W_\shp.
$$
Moreover, if $\supp(\shp)=G(m,r)$, then $W_\shp = \TT_{\Id_m} \mkern-4mu \Pos_{sym}^1(m)$.
\end{lemma}

\begin{proof}
That $W_{\shp}$ is a vector subspace of $(\TT_{\Id_m} \mkern-4mu \Pos_{sym}^1(m), \langle , \rangle_{\Id_m}) \subset (\RR^{m^2}, \langle , \rangle_{\Id_m})$ is a immediate consequence of  the fact that, by Lemma \ref{lem::grad_formula}, the matrix $g^{-1} X (X^{T} \Sigma_\shp^{-1} X)^{-1} X^{T} g^{-1}-  \frac{r}{m} \Id_m  \in \TT_{\Id_m} \mkern-4mu\Pos_{sym}^1(m)$. Notice, the scalar product $ \langle A,B \rangle_{\Id_m}= \tr(AB)$, for every  $A,B \in \TT_{\Id_m} \mkern-4mu \Pos_{sym}^1(m)$, also gives a scalar product on $\RR^{m^2}$. By equation \eqref{eq::covariance} it is clear that $\sigma^2$ is positive semidefinite.

Let us prove $W_{\shp}^{\perp} \subseteq \Ker(\sigma^2)$. Indeed, let $B \in W_{\shp}^{\perp}$, then $ \langle A,B \rangle_{\Id_m}= \tr(AB)=0$, for every $A \in W_{\shp}$. Thus by the property that $\Veco(A)^{T} \Veco(B)= \tr(AB)$ of $\Veco$ operator  we have
\begin{equation*}
\begin{split}
\sigma^2 \Veco(B)&= \E_\shp(\Veco(\Theta (\Theta^T\Theta)^{-1}\Theta^T - \frac{r}{m} \Id_m)(\Veco(\Theta (\Theta^T\Theta)^{-1}\Theta^T - \frac{r}{m} \Id_m)^{T})\Veco(B))\\
&= \E_\shp(\Veco(\Theta (\Theta^T\Theta)^{-1}\Theta^T - \frac{r}{m} \Id_m)\tr((\Theta (\Theta^T\Theta)^{-1}\Theta^T - \frac{r}{m} \Id_m)B))\\
&= 0.\\
\end{split}
\end{equation*}

Let us now prove $\Ker(\sigma^2) \subseteq W_{\shp}^{\perp}$. Indeed, let $B \in \RR^{m^2}$ such that $\sigma^2 \Veco(B)=0$. Then by the same property $\Veco(A)^{T} \Veco(B)= \tr(AB)$ of $\Veco$ operator, applied twice, we have
\begin{equation*}
0=\Veco(B)^{T}\sigma^2 \Veco(B)= \E_\shp(\tr((\Theta (\Theta^T\Theta)^{-1}\Theta^T - \frac{r}{m} \Id_m)B)^2). 
\end{equation*}
Thus, for $\shp$-almost every $U \in G(m,r)$ we have obtained $$\tr((g^{-1} X (X^{T} \Sigma_\shp^{-1} X)^{-1} X^{T} g^{-1}-  \frac{r}{m} \Id_m)B)=0$$
and so $B \in W_{\shp}^{\perp}$. 

Next, $\sigma^2$ being symmetric, ${\rm Ker}(\sigma^2)$ is orthogonal to the column space of $\sigma^2$, which is ${\rm Im}(\sigma^2)$. The assertion then follows since ${\rm Ker}(\sigma^2)= W_{\shp}^{\perp}$.

Finally, assume that $\supp(\shp)=G(m,r)$, and let $v$ be a zero trace symmetric matrix, thus $v \in  \TT_{\Id_m} \mkern-4mu \Pos_{sym}^1(m)$.
The spectral Theorem gives that 
$v=\sum_{i=1}^m \lambda_i {\rm Pr}(u_i,\Id_m)$,
where the $\lambda_i$ are the real eigenvalues of $v$ with $\sum_{i=1}^m \lambda_i =0$, and where the $u_i$ are the orthonormal eigenvectors of $v$. Let $I$ denotes a generic subset of $\{1,...,m\}= : [m]$ of size $r$. We show that one can find coefficients
$\alpha_I$ associated to each subset $I$ such that
\begin{equation}
\label{equ::want_to_prove}
v =\sum_{I\subset [m]}^* \alpha_I (\frac{r}{m}\Id_m -{\rm Pr}(U_I,\Id_m)),
\end{equation}
where $U_I$ denotes the linear subspace generated by the vectors $u_i$, $i\in I$, and where $*$ indicates that we consider $r$-elements subsets of $[m]$. Indeed, equality \eqref{equ::want_to_prove} equivalent to
$$v=\sum_{i=1}^m {\rm Pr}(u_i,\Id_m)\Big(\frac{r}{m}\sum_{I\subset [m]}^* \alpha_I -\sum_{i\in I\subset [m]}^* \alpha_I\Big),$$
where we have used the identities $\Id_m =\sum_{i=1}^m {\rm Pr}(u_i,\Id_m)$ and
${\rm Pr}(U_I,\Id_m)=\sum_{i\in I\subset [m]}{\rm Pr}(u_i,\Id_m)$.
We will prove the statement if one can find coefficients $\alpha_I$ such that
$$\lambda_i = \frac{r}{m}\sum_{I\subset [m]}^* \alpha_I -\sum_{i\in I\subset [m]}^* \alpha_I,\; \;  i \in \{1,\cdots, m\}.$$
Let $k={r-1\choose m-1}$ and $l={r-2 \choose m-2}$. A direct computation shows that 
$\alpha_I = -\sum_{i\in I}\lambda_i /(k-l)$ satisfies the above equation. 
\end{proof}

\begin{lemma}\label{l.2} 
Let $\shp$ be a Borel probability measure on $G(m,r)$. Let $Q$ be the orthogonal projection onto ${\rm Im}(\sigma^2)\subset \RR^{m^2}$, and  let $\Veco(Z)$ be normal of law 
$P_Z :=\mathcal{N}(0, \sigma^2)$.
Then $\Veco(Z)\in W_{\shp}$, 
$Q \Veco(Z)=\Veco(Z)$, and $\supp(P_Z)=W_\shp$.
\end{lemma}
\begin{proof} Recall, by Lemma \ref{lem::ker_sigma} we have ${\rm Im}(\sigma^2)=W_\shp$.
By construction 
$\sigma^2 = A J_{k,m^2}A^T$, where  $k={\rm dim}({\rm Im}(\sigma^2))$, $A$ is invertible and where $J_{k,m^2}$ has the block diagonal form 
$$J_{k,m^2}=\begin{pmatrix} \Id_{k} & 0  \\
0 & 0 \\
\end{pmatrix} .$$
Then, $\Veco(Z)=A J_{k,m^2}Y$, where the entries of $Y$ are i.i.d. standard normal. Hence, $\Veco(Z)=A J_{k,m^2}A^T ((A^{-1})^T Y)= \sigma^2  ((A^{-1})^T Y)\in {\rm Im}(\sigma^2)=W_\shp$  so that
$Q \Veco(Z)=\Veco(Z)$, as required.

To prove the remaining part of the lemma, let $D$ be a Borel subset of the linear subspace $W_\shp$, such that
$\nu_k(D)>0$, where $\nu_k$ denotes the $k$-dimensional Lebesgue measure on $W_\shp$.
As $W_\shp = {\rm Im}(\sigma^2)$, there is some subset $\bar D \subset \RR^{m^2}$ such that
$D = A J_{k,m^2}A^T \bar D$ and $\nu_k(J_{k,m^2}A^T\bar D)>0$.
Hence, for each Borel subset $D$ of $W_\shp$ we have
\begin{eqnarray*}
P(\Veco(Z)\in D)&=&P(A J_{k,m^2} Y \in D)=P(J_{k,m^2} Y \in A^{-1}D )\\
                &=&P(J_{k,m^2} Y \in J_{k,m^2}A^T \bar D)\\
                &=&P((Y_1,\cdots,Y_k)^T\in J_{k,m^2}A^T \bar D)\\ 
                &=&\int_{J_{k,m^2}A^T \bar D}\Big(\frac{1}{\sqrt{2\pi}}\Big)^{k}\exp(-\frac{1}{2}\sum_{i=1}^k y_i^2) {\rm d}y >0.
 \end{eqnarray*}
As $\Veco(Z)\in W_{\shp}$ , the equality $\supp(P_Z)=W_\shp$ follows.
\end{proof}

\begin{proof}[Proof of Theorem \ref{CLT}]

Let $V_{n}:= \frac{m}{\tr(\Sigma_{\shp}^{-1} \Sigma_n)} \Sigma_n $.

For each $n \geq 1$ and  $j \in \{1,..., n\}$ let $$A_{jn}:= X_j^{T} V_n^{-1}  X_j \; \; \; \text{ and } \; \; \;  B_j := X_j^{T} \Sigma_\shp^{-1} X_j.$$ Then it is easy to see that 
\begin{equation}
\label{equ::AB}
A_{jn}^{-1}= B_j^{-1}- A_{jn}^{-1}(A_{jn}-B_j)B_j^{-1} = B_j^{-1}- A_{jn}^{-1}X_j^{T} (V_n^{-1}-\Sigma_\shp^{-1} ) X_jB_j^{-1}.
\end{equation}

Notice  $\frac{r}{m} \Id_m= M_n(\Sigma_n)$ is equivalent to  $\frac{r}{m} \Id_m= M_n(V_n)$, as multiplication by a scalar does not affect the equality.

We apply equation (\ref{equ::AB}) to $\frac{r}{m} \Id_m= M_n(V_n)$ and we want to obtain an equation involving $M_n(\Sigma_\shp)$. Indeed:
\begin{equation}
\label{equ::transform}
\begin{split}
\frac{r}{m} V_n &= \frac{1}{n} \sum_{j =1}^{n}  X_j A_{jn}^{-1} X_j^{T}\\
& = \frac{1}{n} \sum_{j =1}^{n}  X_j \left( B_j^{-1}- A_{jn}^{-1}X_j^{T} (V_n^{-1}-\Sigma_\shp^{-1} ) X_jB_j^{-1} \right) X_j^{T}\\
&= \frac{1}{n} \sum_{j =1}^{n}   X_j  B_j^{-1} X_j^{T} - \frac{1}{n} \sum_{j =1}^{n}  X_j  A_{jn}^{-1}X_j^{T} (V_n^{-1}-\Sigma_\shp^{-1} ) X_j B_j^{-1} X_j^{T}.\\
\end{split}
\end{equation}
By multiplying equation (\ref{equ::transform}) on left by $g^{-1}$ and on the right by $g^{-1}$ we obtain:

\begin{equation}
\label{equ::transform_2}
\begin{split}
\frac{r}{m}g^{-1} V_n  g^{-1} = M_n(\Sigma_\shp)- \frac{1}{n} \sum_{j =1}^{n} g^{-1} X_j  A_{jn}^{-1}X_j^{T} (V_n^{-1}-\Sigma_\shp^{-1} ) X_j B_j^{-1} X_j^{T} g^{-1}.\\
\end{split}
\end{equation}
Set 
\begin{equation}\label{Notation}
 C_n := g^{-1} V_n g^{-1} \; \; \; \text{and} \; \; \;  \Theta_j := g^{-1}X_j,
 \end{equation}
for each $n \geq 1$ and  $j \in \{1,..., n\}$. 
Then,
$$A_{jn}^{-1}=(X_j^{T} V_n^{-1}  X_j)^{-1}= (X_j^{T} g^{-1} C_n^{-1} g^{-1}  X_j)^{-1}= (\Theta_j^{T} C_n^{-1} \Theta_j)^{-1} $$
and
$$X_j^{T} (V_n^{-1}-\Sigma_\shp^{-1} ) X_j = X_j^{T} g^{-1} g(V_n^{-1}-\Sigma_\shp^{-1} ) g g^{-1} X_j=  \Theta_j^{T} (C_n^{-1}- \Id_n )\Theta_j $$
and
$$B_j^{-1} := (X_j^{T} g^{-1} g^{-1} X_j)^{-1}= (\Theta_j^{T} \Theta_j)^{-1}.$$

Now equation (\ref{equ::transform_2}) becomes
\begin{equation*}
\frac{r}{m}C_n = M_n(\Sigma_\shp)- \frac{1}{n} \sum_{j =1}^{n} \Theta_j  (\Theta_j^{T} C_n^{-1} \Theta_j)^{-1} \Theta_j^{T} (C_n^{-1}- \Id_m )\Theta_j  (\Theta_j^{T} \Theta_j)^{-1}  \Theta_j ^{T}.
\end{equation*}
and by adding $- \frac{r}{m} \Id_m$  to both terms of the above equation we obtain 
\begin{equation}
\label{equ::transform_3}
M_n(\Sigma_\shp) -  \frac{r}{m} \Id_m = \frac{r}{m}C_n  -  \frac{r}{m} \Id_m + \frac{1}{n} \sum_{j =1}^{n} \Theta_j  (\Theta_j^{T} C_n^{-1} \Theta_j)^{-1} \Theta_j^{T} (C_n^{-1}- \Id_m )\Theta_j  (\Theta_j^{T} \Theta_j)^{-1}  \Theta_j ^{T}.
\end{equation}

Notice, $C_n$ is in $ \Pos_{sym}^1(m)$, thus it admits a  unique symmetric positive-definite square root that we denote by $C_n^{1/2}$. Then by applying the $\Veco$ operator and equality (\ref{Kronecker_prod})  to 
$$C_n ^{-1} -  \Id_m=C_n ^{-1/2}C_n ^{-1/2} - \Id_m=C_n ^{-1/2} \Id_m C_n ^{-1/2} - \Id_m$$
we obtain:
\begin{equation}
\label{vec_C_i}
\begin{split}
\Veco(C_n ^{-1} -  \Id_m)&= \Veco(C_n ^{-1/2} \Id_m C_n ^{-1/2})- \Veco(\Id_m)\\
&= \left( C_n ^{-1/2}  \otimes C_n ^{-1/2} \right) \Veco(\Id_m) - \Veco(\Id_m) \\
&=\left[ \left( C_n ^{-1/2}  \otimes C_n ^{-1/2} \right)- \Id \right]\Veco(\Id_m).
\end{split}
\end{equation}
Using equality (\ref{vec_C_i}), applying $\Veco$ operator and equality (\ref{Kronecker_prod}) to $$\Theta_j  (\Theta_j^{T} C_n^{-1} \Theta_j)^{-1} \Theta_j^{T} (C_n^{-1}- \Id_m )\Theta_j  (\Theta_j^{T} \Theta_j)^{-1}  \Theta_j ^{T},$$ where $A= \Theta_j  (\Theta_j^{T} C_n^{-1} \Theta_j)^{-1} \Theta_j^{T}$, $X =C_n^{-1}- \Id_m $, and $B= \Theta_j  (\Theta_j^{T} \Theta_j)^{-1}  \Theta_j ^{T}$, we obtain
\begin{equation}
\label{vec_theta}
\begin{split}
&\Veco(\Theta_j  (\Theta_j^{T} C_n^{-1} \Theta_j)^{-1} \Theta_j^{T} (C_n^{-1}- \Id_m )\Theta_j  (\Theta_j^{T} \Theta_j)^{-1}  \Theta_j ^{T})=\\
&= \left( \Theta_j  (\Theta_j^{T} \Theta_j)^{-1}  \Theta_j ^{T} \otimes  \Theta_j  (\Theta_j^{T} C_n^{-1} \Theta_j)^{-1} \Theta_j^{T}\right)\Veco(C_n^{-1}- \Id_m)\\
&= \left( \Theta_j  (\Theta_j^{T} \Theta_j)^{-1}  \Theta_j ^{T} \otimes  \Theta_j  (\Theta_j^{T} C_n^{-1} \Theta_j)^{-1} \Theta_j^{T}\right) \left[ \left(C_n ^{-1/2} \otimes C_n ^{-1/2} \right)- \Id_{m^2} \right]\Veco(\Id_m).\\
\end{split}
\end{equation}

Also 
\begin{equation}
\label{vec_Id}
\Veco(\Id_m)= \Veco( C_n ^{-1/2} C_n  C_n ^{-1/2} )=  \left(  C_n ^{-1/2}\otimes C_n ^{-1/2} \right) \Veco(C_i).
\end{equation}

We denote 
\begin{equation}
\label{vec_Sigma_1_i}
\Sigma_{1,n}:= \frac{1}{n} \sum_{j =1}^{n}   \Theta_j  (\Theta_j^{T} \Theta_j)^{-1}  \Theta_j ^{T} \otimes  \Theta_j  (\Theta_j^{T} C_n^{-1} \Theta_j)^{-1} \Theta_j^{T}.
\end{equation}

By applying $\Veco$ operator and equation (\ref{vec_theta}) to equality (\ref{equ::transform_3}) we obtain
\begin{equation}
\label{vec_M_i_Sigma}
\begin{split}
\Veco(M_n(\Sigma_\shp) -  \frac{r}{m} \Id_m)&= \frac{r}{m} \Veco(C_n - \Id_m) \\
& \; \; \; +\frac{1}{n} \sum_{j =1}^{n} \Veco(\Theta_j  (\Theta_j^{T} C_n^{-1} \Theta_j)^{-1} \Theta_j^{T} (C_n^{-1}- \Id_m )\Theta_j  (\Theta_j^{T} \Theta_j)^{-1}  \Theta_j ^{T})\\
&= \frac{r}{m} \Veco(C_n - \Id_m) + \Sigma_{1,n}  \left(C_n ^{-1/2}  \otimes C_n ^{-1/2} \right) \Veco(\Id_m)\\
& \; \; \; \; -  \Sigma_{1,n} \Veco(\Id_n)\\
&=\frac{r}{m} \Veco(C_n - \Id_m) + \Sigma_{1,n}  \left( C_n ^{-1/2} \otimes C_n ^{-1/2} \right) \Veco(\Id_m)\\
& \; \; \; \;-  \Sigma_{1,n}\left(  C_n ^{-1/2}\otimes C_n ^{-1/2} \right) \Veco(C_n) \\
&= \left[ \frac{r}{m} \Id_{m^2} - \Sigma_{1,n}  \left( C_n ^{-1/2} \otimes C_n ^{-1/2} \right)  \right]\Veco(C_n - \Id_m).
\end{split}
\end{equation}

We next consider the factor $\frac{r}{m}\Id_m-\Sigma_{1,n}\left( C_n ^{-1/2} \otimes C_n ^{-1/2} \right)$. Theorem
\ref{thm::LLN} gives us that $C_n  \xrightarrow[a.s]{n \to \infty} \Id_m$. Furthermore, the identity
$$(\Theta^T C_n^{-1}\Theta)^{-1} = (\Theta^T\Theta)^{-1} - (\Theta^T C_n^{-1}\Theta)^{-1}\Theta^T (C_n^{-1}-\Id_m)\Theta (\Theta^T\Theta)^{-1},$$
leads to
\begin{equation}\label{eq}
\Sigma_{1,n}=\Sigma_{0,n}-\frac{1}{n}\sum_{j=1}^n \Theta_j(\Theta_j^T\Theta_j)^{-1}\Theta_j^T
\otimes \Theta_j (\Theta_j^T C_n^{-1}\Theta_j)^{-1}\Theta_j^T \triangle^n \Theta_j(\Theta_j^T\Theta_j)^{-1}\Theta_j^T,
\end{equation}
where
$$\Sigma_{0,n}=\frac{1}{n}\sum_{j=1}^n \Theta_j(\Theta_j^T\Theta_j)^{-1}\Theta_j^T\otimes \Theta_j(\Theta_j^T\Theta_j)^{-1}\Theta_j^T,$$
and $\triangle^n = C_n^{-1}-\Id_m$. The strong law of large numbers  and Theorem \ref{thm::LLN} give that
$$\Sigma_{0,n} \xrightarrow[a.s]{n \to \infty} \Sigma_0=\E_\shp(\Theta (\Theta^T\Theta)^{-1}\Theta^T\otimes \Theta (\Theta^T\Theta)^{-1}\Theta^T),$$
and $\triangle^n  \xrightarrow[a.s]{n \to \infty} 0$. We next show that the second term of the right hand side of (\ref{eq}) $ \xrightarrow[a.s]{n \to \infty} 0$. For fixed $1\le j\le n$, any entry of the Kronecker product can be written as 
$$a_{ik}^j \sum_{1\le \nu,\mu\le m}\alpha_{k\nu}^j \triangle_{\nu\mu}^n a_{\mu l}^j,$$
 where
$a_{ik}^j$ is the $(ik)$ entry of $\Theta_j(\Theta_j^T\Theta_j)^{-1}\Theta_j^T$,  $\alpha_{k\nu}^j$ is the 
$(k\nu)$ entry of $\Theta_j (\Theta_j^T C_n^{-1}\Theta_j)^{-1}\Theta_j^T$, and $\triangle_{\nu\mu}^n$ is the $(\nu\mu)$ entry of $\triangle^{n}$. 
From construction, 
$$\Theta_j(\Theta_j^T\Theta_j)^{-1}\Theta_j^T = g^{-1} {\rm Pr}(U_j,\Sigma_\shp)g \hbox{ and }
\Theta_j (\Theta_j^T C_n^{-1}\Theta_j)^{-1}\Theta_j^T = g^{-1} {\rm Pr}(U_j,V_n)V_n (g^{-1})^T.$$
The linear operators ${\rm Pr}(U_j,\Sigma_\shp)$ and ${\rm Pr}(U_j, V_n)$ are projectors and thus have bounded entries (just look at the spectral decomposition relatively to orthonormal basis with eigenvalues equal to 0 or 1 ). Moreover, Theorem
\ref{thm::LLN} gives that $V_n \xrightarrow[a.s]{n \to \infty} \Sigma_\shp$. We can thus assume that all the entries 
$a_{ik}$ and $\alpha_{k\nu}^j$ are bounded by some positive constant  $M>0$. 
Hence
$$\left\vert\frac{1}{n}\sum_{j=1}^n a_{ik}^j \sum_{1\le \nu,\mu\le m}\alpha_{k\nu}^j \triangle_{\nu\mu}^n a_{\mu l}^j \right\vert
\le \sum_{1\le \nu,\mu\le m}\vert\triangle_{\nu\mu}^n\vert \cdot
\left\vert\frac{1}{n}\sum_{j=1}^n a_{ik}^j \alpha_{k\nu}^j a_{\mu l}^j \right\vert \xrightarrow[a.s]{n \to \infty} 0.$$

Then (\ref{vec_M_i_Sigma}) becomes asymptotically equivalent to
\begin{equation}\label{limiting}
 \left[ \frac{r}{m} \Id_{m^{2}} - \Sigma_0  \left( \Id_m  \otimes \Id_m \right)  \right]\sqrt{n}(\Veco(C_n - \Id_m)) \xrightarrow[distribution]{n \to \infty} \Veco(Z),
\end{equation}
where we recall that $\Veco(Z)$ is normal of law 
$\mathcal{N}(0, \sigma^2)$. The final step of the proof consists in inverting this last relation using the algebraic knowledge on the covariance matrix $\sigma^2$ and on $\Sigma_0$.

By hypothesis, $\supp(\shp)=G(m,r)$, so that, from Lemma \ref{lem::ker_sigma}, $W_\shp = \TT_{\Id_m} \mkern-4mu \Pos_{sym}^1(m)$.
Hence, $C_n - \Id_m \in W_{\shp}$. Indeed, by its definition, $C_n$ is symmetric with $tr(C_n)= m$. Then $C_n - \Id_m$ is also symmetric of trace zero, and thus  $C_n - \Id_m \in \TT_{\Id_m} \mkern-4mu \Pos_{sym}^1(m)= W_{\shp}$.
Lemma \ref{l.2} shows that the random vector $\Veco(Z)$ is supported by the full linear subspace $W_\shp$. Then, we claim the restriction $L_0$ of the linear operator $\frac{r}{m} \Id_{m^{2}} - \Sigma_0  \left( \Id_m  \otimes \Id_m \right) $ to $W_\shp$ is injective. Indeed, assuming the contrary, the image of $L_0$ would be strictly contained as a linear subspace in $W_\shp = \supp(P_Z)$, a contradiction with the weak convergence result given in (\ref{limiting}).
Then Lemmas \ref{lem::ker_sigma} and \ref{l.2} give that (\ref{limiting}) is equivalent to
\begin{equation}\label{limiting2}
 Q L_0 Q\sqrt{n}(\Veco(C_n - \Id_m)) \xrightarrow[distribution]{n \to \infty} \Veco(Z),
\end{equation}
so that
\begin{equation}\label{limiting3}
 \sqrt{n}(\Veco(C_n - \Id_m)) \xrightarrow[distribution]{n \to \infty}\left[ Q L_0 Q\right]^+\Veco(Z),
\end{equation}
where $\left[ A \right]^+$ denotes the Moore--Penrose inverse of a matrix $A$.
One obtains finally that $\sqrt{n}(\Veco(C_n - \Id_m))$ is asymptotically multivariate centred normal of covariance
$$\sigma_\infty^2 = \left[ Q L_0 Q \right]^+ \sigma^2 \Big(\left[ Q L_0 Q \right]^+\Big)^T.$$

\end{proof}


\section{Appendix}
\label{sec::Appendix}
\subsection{Proof of Lemma \ref{lem::diff_log_like}}
Using Lemma \ref{lem::diff_det} we have: 
\begin{equation*}
\label{equ::grad}
\begin{split}
\frac{d \ell_{U}(g \exp(tV)g^{T})}{dt} \vert_{t=0}&=\frac{1}{2} \frac{d \log \frac{\det(X^{T} (g^{T})^{-1} \exp(-tV)g^{-1} X)}{\det(X^{T}X)}}{dt} \vert_{t=0}\\
&= \frac{1}{2}  \frac{\det(X^{T}X)}{\det(X^{T} (g^{T})^{-1} \exp(-tV)g^{-1} X)}\vert_{t=0} \cdot  \frac{d \frac{\det(X^{T} (g^{T})^{-1} \exp(-tV)g^{-1} X)}{\det(X^{T}X)}}{dt}\vert_{t=0}\\
&=  \frac{1}{2}  \frac{\det(X^{T}X)}{\det(X^{T} \Sigma^{-1} X)} \cdot \frac{1}{\det(X^{T}X)} \cdot \frac{d \det(X^{T} (g^{T})^{-1} \exp(-tV)g^{-1} X)}{dt}\vert_{t=0}\\
&=  -\frac{1}{2} \frac{1}{\det(X^{T} \Sigma^{-1} X)} \cdot \det(X^{T} \Sigma^{-1} X) \cdot \tr((X^{T} \Sigma^{-1} X)^{-1} X^{T} (g^{T})^{-1} Vg^{-1} X)\\
&= -\frac{1}{2} \cdot \tr((X^{T} \Sigma^{-1} X)^{-1} X^{T} (g^{T})^{-1} Vg^{-1} X)\\
&= -\frac{1}{2} \cdot \tr(X (X^{T} \Sigma^{-1} X)^{-1} X^{T} (g^{T})^{-1} Vg^{-1} )\\
&= d(\ell_{U})_{\Sigma}(gVg^{T}).\\
\end{split}
\end{equation*}

\subsection{Proof of Lemma \ref{lem::grad_formula}}
By  Lemma \ref{lem::diff_log_like}, for $V \in \TT_{\Id_m}\mkern-4mu\Pos_{sym}^1(m)$ and $W = g V g^{T}$
 \begin{equation*}
 \label{equ::first_covar_deriv}
 \begin{split}
 d(\ell_{U})_{\Sigma}(W)&= -\frac{1}{2} \cdot \tr(X (X^{T} \Sigma^{-1} X)^{-1} X^{T} (g^{T})^{-1} Vg^{-1} )\\
 &= -\frac{1}{2} \cdot \tr(X (X^{T} \Sigma^{-1} X)^{-1} X^{T} \Sigma^{-1} W \Sigma^{-1} ).\\
 \end{split}
 \end{equation*}
We obtain
\begin{equation*}
\begin{split}
-\frac{1}{2} \cdot \tr(X (X^{T} \Sigma^{-1} X)^{-1} X^{T} (g^{T})^{-1} Vg^{-1})&=\langle \grad \ell_{U}(\Sigma) , g V g^{T}\rangle_{\Sigma}\\
&= \tr(\Sigma^{-1} \grad \ell_{U}(\Sigma) \Sigma^{-1} g V g^{T})\\
&= \tr((g^{T})^{-1} g^{-1}  \grad \ell_{U}(\Sigma) (g^{T})^{-1} g^{-1} g V g^{T})\\
&= \tr(g^{-1}  \grad \ell_{U}(\Sigma) (g^{T})^{-1}V)\\
&=\tr(\grad \ell_{U}(\Sigma) (g^{T})^{-1}Vg^{-1})\\
\end{split}
\end{equation*}
for every $V \in\TT_{\Id_m}\mkern-4mu\Pos_{sym}^1(m)$.

Notice, as  $\grad \ell_{U}(\Sigma) \in \TT_{\Sigma} \mkern-4mu \Pos_{sym}^1(m)=\{A \in \Sym(m) \; \vert \; \tr(g^{-1}A (g^{T})^{-1})=0\} $ we cannot have $X (X^{T} \Sigma^{-1} X)^{-1} X^{T} \in \TT_{\Sigma} \mkern-4mu \Pos_{sym}^1(m)$ as 
\begin{equation*}
\begin{split}
\tr(g^{-1} X (X^{T} \Sigma^{-1} X)^{-1} X^{T} (g^{T})^{-1})&=\tr((g^{T})^{-1}g^{-1} X (X^{T} \Sigma^{-1} X)^{-1} X^{T})\\
&= \tr(\Sigma^{-1}X (X^{T} \Sigma^{-1} X)^{-1} X^{T})\\
&= \tr(X^{T}\Sigma^{-1}X (X^{T} \Sigma^{-1} X)^{-1})\\
&=\tr(\Id_{r \times r})= r.
\end{split}
\end{equation*}
One sees that $\frac{r}{2m} \Sigma-\frac{1}{2} X (X^{T} \Sigma^{-1} X)^{-1} X^{T}$ is in $\Sym(m)$ and verifies 
$$\tr(g^{-1} (\frac{r}{2m} \Sigma-\frac{1}{2} X (X^{T} \Sigma^{-1} X)^{-1} X^{T})(g^{T})^{-1})=0.$$

Moreover, one can verify that indeed 
\begin{equation}
\label{equ::grad_formula}
\grad \ell_{U}(\Sigma)= \frac{r}{2m} \Sigma-\frac{1}{2} X (X^{T} \Sigma^{-1} X)^{-1} X^{T}
\end{equation} satisfies
$$-\frac{1}{2} \cdot \tr(X (X^{T} \Sigma^{-1} X)^{-1} X^{T} (g^{T})^{-1} Vg^{-1})=\langle \grad \ell_{U}(\Sigma) , g V g^{T}\rangle_{\Sigma}$$
as $\tr(\Sigma^{-1}\frac{r}{2m} \Sigma \Sigma^{-1} g V g^{T} )=m \frac{r}{2m} \tr(\Sigma^{-1} g V g^{T} )=\frac{r}{2} \tr(g^{T} \Sigma^{-1} g V )=0.$

\subsection{Proof of Lemma \ref{lem::covariant_deriv}}
Let $X, Y, Z$ be vector fields in $\TT \Pos_{sym}^1(m)$. It is well known the Levi--Civita connection $\nabla_{X}$ satisfies Koszul's formula:
\begin{equation}
\label{equ::Koszul}
\begin{split}
2 \langle  \nabla_{X}Y, Z\rangle &= X \langle Y,Z\rangle +  Y \langle X,Z\rangle - Z \langle X,Y\rangle\\
& \; \; - \langle X,[Y,Z]\rangle - \langle Y,[X,Z]\rangle + \langle Z,[X,Y]\rangle.\\
\end{split}
\end{equation}  
As in our case of $\Pos_{sym}^1(m)$  the metric is given by $\langle Y,Z\rangle_{\Sigma}= \tr(\Sigma^{-1}Y\Sigma^{-1}Z)$, for every $\Sigma \in \Pos_{sym}^1(m)$, the directional derivative along the vector field $X$ of $\langle Y,Z\rangle$ gives
\begin{equation}
\label{equ::directional_deriv_X}
\begin{split}
X(\langle Y,Z\rangle)_{\Sigma} &= \tr(\Sigma^{-1}Y\Sigma^{-1} (X(Z))_{\Sigma}) +  \tr(\Sigma^{-1}(X(Y))_{\Sigma}\Sigma^{-1} Z) \\
&\; \; - \tr(\Sigma^{-1}X\Sigma^{-1} Y \Sigma^{-1}Z) -  \tr(\Sigma^{-1}Y\Sigma^{-1} X \Sigma^{-1}Z).\\
\end{split}
\end{equation}
We have used the well known formula 
\begin{equation}
\label{equ::deriv_inverse}
\frac{d}{dt}E(t)^{-1}\vert_{t=0}= -E(t)^{-1} (\frac{d}{dt}E(t)) E(t)^{-1}  \vert_{t=0}. 
\end{equation}
By cyclic permutations one obtains similar formulas for $Y \langle X,Z\rangle$ and  $Z \langle X,Y\rangle$.

As $[Y,Z]= YZ -ZY$ we have
$$\langle X,[Y,Z]\rangle= \tr(\Sigma^{-1}X\Sigma^{-1}YZ) - \tr(\Sigma^{-1}X\Sigma^{-1}ZY)$$
and similarly for $\langle Y,[X,Z]\rangle$ and $\langle Z,[X,Y]\rangle$.

Adding everything together and canceling similar terms, from Koszul's formula (\ref{equ::Koszul}) we obtain 
$$\langle  \nabla_{X}Y, Z\rangle_{\Sigma}= \tr(\Sigma^{-1}(X(Y))_{\Sigma}\Sigma^{-1}Z)- \frac{1}{2} \tr(\Sigma^{-1}X\Sigma^{-1}Y\Sigma^{-1}Z) - \frac{1}{2} \tr(\Sigma^{-1}Y\Sigma^{-1}X\Sigma^{-1}Z) $$
for every $X,Y,Z$. Thus as desired 
$(\nabla_{X} Y)_\Sigma= (X(Y))_{\Sigma} - \frac{1}{2} X\Sigma^{-1}Y- \frac{1}{2} Y\Sigma^{-1}X \in  \TT_{\Sigma} \mkern-4mu \Pos_{sym}^1(m).$

Then by taking $Y= \grad \ell_{U}$, $Z$ a vector field and $\Sigma \in \Pos_{sym}^1(m)$
$$(\nabla_{Z} \grad \ell_{U})_\Sigma= (Z(\grad \ell_{U}))_{\Sigma} - \frac{1}{2} Z\Sigma^{-1}\grad \ell_{U}(\Sigma)- \frac{1}{2} \grad \ell_{U}(\Sigma)\Sigma^{-1}Z.$$

By computing $(Z(\grad \ell_{U}))_{\Sigma}= \frac{d}{dt}(\grad \ell_{U}(g \exp(tV) g^{T}))\vert_{t=0}$, where $Z_{\Sigma}= g V g^{T}$ for a unique $V \in\TT_{\Id_m}$, the desired formula is obtained
$$\nabla_{Z} \grad \ell_{U}(\Sigma) = \frac{1}{4}Z\pi_{U}(\Sigma) \Sigma+  \frac{1}{4} \Sigma \pi_{U}(\Sigma) Z -\frac{1}{2}  \Sigma \pi_{U}(\Sigma) Z \pi_{U}(\Sigma)\Sigma$$
where $\pi_{U}(\Sigma):= \Sigma^{-1} X (X^{T} \Sigma^{-1}  X)^{-1} X^{T} \Sigma^{-1}$. 

We have used again the chain rule and the well known formula (\ref{equ::deriv_inverse}).

\subsection{Proof of Lemma \ref{lem::grad_h_n}}
Let $\Gamma = gg^{T} \in \Pos_{sym}^1(m)$, $V \in \TT_{\Id_m}\mkern-4mu\Pos_{sym}^1(m)$ and $W = g V g^{T} \in \TT_{\Gamma}\mkern-4mu\Pos_{sym}^1(m)$. Then 
$$\nabla_{W}h_n(\Gamma) = \langle \grad(h_n)(\Gamma),W\rangle_{\Gamma} = \frac{d h_{n}(g \exp(tV)g^{T})}{dt} \vert_{t=0}.$$

Recall from Lemma \ref{lem::grad_formula} that $\grad(\ell_{\shp_n})(\Gamma)= \frac{r}{2m} \Gamma-\frac{1}{2n} \sum\limits_{j=1}^{n} \Gamma \pi_{U_{j}}(\Gamma) \Gamma$. So 
\begin{equation}
\label{equ::h_n_trace}
\begin{split}
h_n{\Gamma}&=\tr(\Gamma^{-1}\grad(\ell_{\shp_n})(\Gamma)\Gamma^{-1} \grad(\ell_{\shp_n})(\Gamma))\\
&=\tr\left( \left(\frac{r}{2m} \Id_m-\frac{1}{2n} \sum\limits_{j=1}^{n}\pi_{U_{j}}(\Gamma) \Gamma \right) \left(\frac{r}{2m} \Id_m-\frac{1}{2n} \sum\limits_{j=1}^{n}\pi_{U_{j}}(\Gamma) \Gamma \right) \right).
\end{split}
\end{equation}
We need to compute 
\begin{equation}
\label{equ::h_n_deriv}
\begin{split}
\frac{d h_{n}(g \exp(tV)g^{T})}{dt} \vert_{t=0} &= \frac{1}{4}\frac{d}{dt} \tr \left( \left(\frac{r}{m} \Id_m-\frac{1}{n} \sum\limits_{j=1}^{n}\pi_{U_{j}}(g \exp(tV)g^{T}) g \exp(tV)g^{T} \right)^2 \right) \vert_{t=0}\\
&= \frac{1}{4}\tr \left(\frac{d}{dt}  \left(\frac{r}{2m} \Id_m-\frac{1}{2n} \sum\limits_{j=1}^{n}\pi_{U_{j}}(g \exp(tV)g^{T}) g \exp(tV)g^{T} \right)^2 \right) \vert_{t=0}\\
&=- \frac{1}{2n} \tr \left(\frac{d}{dt}\left(\sum\limits_{j=1}^{n}\pi_{U_{j}}(g \exp(tV)g^{T}) g \exp(tV)g^{T} \right) \vert_{t=0} \left(\frac{r}{m} \Id_m-\frac{1}{n} \sum\limits_{j=1}^{n}\pi_{U_{j}}(\Gamma) \Gamma \right) \right).\\
\end{split}
\end{equation}

Recall 
$$\sum\limits_{j=1}^{n}\pi_{U_{j}}(\Gamma) \Gamma= \sum\limits_{j=1}^{n}\Gamma^{-1} X_{j} (X_{j}^{T} \Gamma^{-1}  X_{j})^{-1} X_{j}^{T}.$$

Using formula (\ref{equ::deriv_inverse}) and letting $A(n,\Gamma):= \frac{r}{m} \Id_m-\frac{1}{n} \sum\limits_{j=1}^{n}\pi_{U_{j}}(\Gamma) \Gamma $,  equation (\ref{equ::h_n_deriv}) equals further
\begin{equation}
\label{equ::h_n_deriv_further}
\begin{split}
&- \frac{1}{2n} \tr \left(\frac{d}{dt}\left(\sum\limits_{j=1}^{n}(g\exp(tV)g^{T})^{-1} X_{j} (X_{j}^{T} ( \exp(tV)g^{T})^{-1}  X_{j})^{-1} X_{j}^{T}  \right) \vert_{t=0} A(n,\Gamma) \right)\\
&= -\frac{1}{2n} \tr \left(\left(\sum\limits_{j=1}^{n}\Gamma^{-1} X_{j} (X_{j}^{T} \Gamma^{-1}  X_{j})^{-1} X_{j}^{T} \Gamma^{-1} W \Gamma^{-1}  X_{j} (X_{j}^{T} \Gamma^{-1}  X_{j})^{-1}X_{j}^{T}  \right) A(n,\Gamma) \right)\\
&\; \; + \frac{1}{2n} \tr \left(\left(\sum\limits_{j=1}^{n} \Gamma^{-1} W \Gamma^{-1} X_{j} (X_{j}^{T} \Gamma^{-1}  X_{j})^{-1} X_{j}^{T}  \right) A(n,\Gamma)  \right)\\
&=\frac{1}{2n} \tr \left(\sum\limits_{j=1}^{n} \Gamma^{-1} W \Gamma^{-1} X_{j} (X_{j}^{T} \Gamma^{-1}  X_{j})^{-1} X_{j}^{T}  A(n,\Gamma)  \right)\\
&\; \; - -\frac{1}{2n} \tr \left(\sum\limits_{j=1}^{n}\Gamma^{-1} W \Gamma^{-1}  X_{j} (X_{j}^{T} \Gamma^{-1}  X_{j})^{-1}X_{j}^{T}  A(n,\Gamma) \Gamma^{-1} X_{j} (X_{j}^{T} \Gamma^{-1}  X_{j})^{-1} X_{j}^{T} \right)\\
\end{split}
\end{equation}
\begin{equation*}
\begin{split}
&=\frac{1}{2n} \tr \left(\Gamma^{-1} W \Gamma^{-1}\left(\sum\limits_{j=1}^{n} X_{j} (X_{j}^{T} \Gamma^{-1}  X_{j})^{-1} X_{j}^{T}  A(n,\Gamma) \right) \right) \\
&\; \; -\frac{1}{2n} \tr \left(\Gamma^{-1} W \Gamma^{-1} \left( \sum\limits_{j=1}^{n} X_{j} (X_{j}^{T} \Gamma^{-1}  X_{j})^{-1}X_{j}^{T}  A(n,\Gamma) \Gamma^{-1} X_{j} (X_{j}^{T} \Gamma^{-1}  X_{j})^{-1} X_{j}^{T} \right) \right).\\
\end{split}
\end{equation*}
Now from the definition of $A(n,\Gamma)$ the term $\frac{1}{2n} \tr \left(\Gamma^{-1} W \Gamma^{-1}\left(\sum\limits_{j=1}^{n} X_{j} (X_{j}^{T} \Gamma^{-1}  X_{j})^{-1} X_{j}^{T} \frac{r}{m} \Id_m \right) \right) $ cancels the term $-\frac{1}{2n} \tr \left(\Gamma^{-1} W \Gamma^{-1} \left( \sum\limits_{j=1}^{n} X_{j} (X_{j}^{T} \Gamma^{-1}  X_{j})^{-1}X_{j}^{T}  \frac{r}{m} \Id_m \Gamma^{-1} X_{j} (X_{j}^{T} \Gamma^{-1}  X_{j})^{-1} X_{j}^{T} \right) \right)$.

Therefore, equation (\ref{equ::h_n_deriv_further}) equals further 
\begin{equation}
\label{equ::h_n_deriv_further2}
\begin{split}
&-\frac{1}{2n^2} \tr \left(\Gamma^{-1} W \Gamma^{-1}\left(\sum\limits_{j=1}^{n} X_{j} (X_{j}^{T} \Gamma^{-1}  X_{j})^{-1} X_{j}^{T}  B(n,\Gamma) \Gamma \right) \right) \\
&\; \; \frac{1}{2n^2} \tr \left(\Gamma^{-1} W \Gamma^{-1} \left( \sum\limits_{j=1}^{n} X_{j} (X_{j}^{T} \Gamma^{-1}  X_{j})^{-1}X_{j}^{T}  B(n,\Gamma)  X_{j} (X_{j}^{T} \Gamma^{-1}  X_{j})^{-1} X_{j}^{T} \right) \right),\\
\end{split}
\end{equation}
where $B(n,\Gamma):=\sum\limits_{j=1}^{n}\pi_{U_{j}}(\Gamma) .$

Then to finish the proof one needs to verify that 
\begin{equation}
\label{equ::grad_first_h_n}
\begin{split}
C:=&\frac{1}{2n^2} \left( \sum\limits_{j=1}^{n} \Gamma \pi_{U_{j}}(\Gamma) \Gamma B(n,\Gamma) \Gamma \pi_{U_{j}}(\Gamma) \Gamma \right)- \frac{1}{2n^2}\left(\sum\limits_{j=1}^{n} \Gamma \pi_{U_{j}}(\Gamma) \Gamma  B(n,\Gamma) \Gamma \right)\\
&= \frac{1}{2n^2} \left( \sum\limits_{j=1}^{n} \Gamma \pi_{U_{j}}(\Gamma) \Gamma B(n,\Gamma) \Gamma \pi_{U_{j}}(\Gamma) \Gamma \right)- \frac{1}{2n^2} \Gamma B(n,\Gamma) \Gamma  B(n,\Gamma) \Gamma\\
\end{split}
\end{equation}
is in the tangent space  $\TT_{\Gamma}\mkern-4mu\Pos_{sym}^1(m)$ at $\Gamma$. Indeed, it is easy to see that the matrix $C$ from equation (\ref{equ::grad_first_h_n}) is in $\Sym(m)$ and $\tr(C\Gamma^{-1})=0$. Therefore, $\grad(h_n)(\Gamma)=C$ and the lemma follows.

\end{document}